\documentclass[french, oneside, reqno]{amsart}
\usepackage[utf8]{inputenc}
\usepackage[T1]{fontenc}
\usepackage{etex}
\usepackage{lmodern}
\usepackage{xfrac}    
\usepackage{amssymb, amsmath, amsthm}
\usepackage{mathtools}
\usepackage{url}
\usepackage{varioref}
\usepackage[colorlinks, linktocpage, citecolor = red, linkcolor = red, urlcolor=red]{hyperref}
\usepackage{graphicx}
\usepackage{float}
\usepackage[french]{babel}
\usepackage[font=small,labelfont=bf]{caption}
\usepackage{wrapfig}
\usepackage{cases}
\usepackage{pdfpages}
\usepackage{tikz-cd}
\usepackage{emptypage}
\usepackage[shortlabels]{enumitem}
\setlist[1,itemize]{label=\footnotesize{$\bullet$}}
\setlist[2,itemize]{label=\footnotesize{$\ast$}}
\setlist[1,enumerate]{label=\arabic*)}
\setlist[2,enumerate]{label=\alph*)}

\numberwithin{equation}{section}

\usepackage{tikz}
\usetikzlibrary{decorations.pathmorphing, decorations.markings, arrows, matrix, fit, positioning}
\tikzstyle directed=[postaction={decorate,decoration={markings, mark=at position .55 with {\arrow{stealth}}}}]

\newtheorem{thm}[equation]{Théorème}
\newtheorem{prop}[equation]{Proposition}
\newtheorem{lemme}[equation]{Lemme}
\newtheorem{cor}[equation]{Corollaire}

\newtheorem{maintheorem}{Théorème}

\theoremstyle{definition}
\newtheorem{defi}[equation]{Définition}
\newtheorem*{defi*}{Définition}

\newtheorem*{conj*}{Conjecture}

\theoremstyle{remark}
\newtheorem{rmq}[equation]{Remarque}
\newtheorem*{rmq*}{Remarque}
\newtheorem{ex}[equation]{Exemple}
\newtheorem*{ex*}{Exemple}
\newtheorem{fait}[equation]{Fait}

\newcommand{\B}{\mathcal{B}}
\newcommand{\C}{\mathbb{C}}

\newcommand{\F}{\mathbb{F}}

\renewcommand{\j}{\mathfrak{j}}

\newcommand{\N}{\mathbb{N}}

\renewcommand{\P}{\mathbb{P}}
\newcommand{\q}{\mathfrak{q}}
\newcommand{\R}{\mathbb{R}}

\newcommand{\V}{\mathcal{V}}
\newcommand{\Z}{\mathbb{Z}}

\newcommand{\isome}[1]{#1_{{\scriptscriptstyle{\#}}}}
\newcommand{\orbl}{\mathcal{E}}

\renewcommand{\H}{\mathbb{H}^{\infty}}
\newcommand{\Hh}{\mathbb{H}}

\renewcommand{\epsilon}{\varepsilon}
\renewcommand{\l}{\ell}

\newcommand{\dashmapsto}{\mapstochar\dashrightarrow}

\DeclareMathOperator{\Bir}{Bir}

\DeclareMathOperator{\PGL}{PGL}
\DeclareMathOperator{\PSL}{PSL}

\DeclareMathOperator{\PM}{\mathcal{Z}}
\DeclareMathOperator{\Bs}{Bs}

\DeclareMathOperator{\NS}{N^1}
\DeclareMathOperator{\kk}{k}
\DeclareMathOperator{\can}{K}
\DeclareMathOperator{\fcan}{k_{\P^2}}
\DeclareMathOperator{\card}{Card}

\DeclareMathOperator{\car}{c}

\DeclareMathOperator{\dist}{d}

\DeclareMathOperator{\I}{I}
\DeclareMathOperator{\Ima}{Im}
\DeclareMathOperator{\Reel}{Re}
\DeclareMathOperator{\id}{id}

\DeclareMathOperator{\Maj}{Maj}

\DeclareMathOperator{\n}{\eta}

\DeclareMathOperator{\argcosh}{argcosh}

\DeclareMathOperator{\supp}{supp}

\title{Pavage de Voronoï associé au groupe de Cremona}
\author{Anne Lonjou}
\date{8 mai 2018}
\address{Universität Basel, Departement Mathematik und Informatik,
Spiegelgasse 1,,
4051 Basel, Switzerland}
\email{anne.lonjou@unibas.ch} 

\begin{document}

\maketitle

\begin{abstract}
L'action du groupe de Cremona de rang $2$ sur un espace hyperbolique de dimension infinie est récemment devenue centrale dans l’étude de ce groupe. Guidée par l'analogie avec l'action de $\PSL(2,\Z)$ sur le demi-plan de Poincaré, nous exhibons un domaine fondamental pour cette action en considérant un pavage de Voronoï.
Nous étudions ensuite les cellules adjacentes à une cellule de Voronoï donnée, ainsi que les cellules qui partagent des points communs au bord à l'infini.
\end{abstract}

\renewcommand{\abstractname}{Abstract}
\begin{abstract}
The action of the Cremona group of rank $2$ on an infinite dimensional hyperbolic space is the main recent tool to study the Cremona group. Following the analogy with the action of $\PSL(2,\Z)$ on the Poincaré half-plane, we exhibit a fundamental domain for this action by considering a Voronoï tessellation. Then we study adjacent cells to a given cell, as well as cells that share common points in the boundary at infinity.
\end{abstract}

\setcounter{tocdepth}{1}
\tableofcontents
\section*{Introduction}
Le groupe de Cremona de rang $2$ sur un corps $\kk$, noté $\Bir(\P^2_{\kk})$, est le groupe des transformations birationnelles du plan projectif. Dans l’étude qui nous intéresse, l'action du groupe modulaire $\PSL(2,\Z)$ sur le demi-plan de Poincaré s'est avérée un guide précieux. Dans la suite de l'introduction, nous décrivons diverses analogies entre $\PSL(2,\Z)$ et le groupe de Cremona. 

\subsection*{Générateurs}
Soient \[S=\begin{pmatrix}
0 & -1\\
1 & \hspace{0,25cm}0
\end{pmatrix},\  T=\begin{pmatrix}
1 & 1\\
0 & 1
\end{pmatrix} \text{ et }
U=\begin{pmatrix}
1 & -1\\
1 & \hspace{0,25cm}0
\end{pmatrix}\] trois éléments de $\PSL(2,\Z)$. Les matrices $S$ et $U$ sont respectivement d'ordre $2$ et $3$ alors que la matrice $T$ est d'ordre infini.
Le groupe modulaire est de présentation finie. Il est par exemple engendré par les matrices $S$ et $U$ et les relations sont engendrées par les relateurs $S^2=\I_2$ et $U^3=\I_2$, ou encore par les matrices $S$ et $T$ avec les relations engendrées par $S^2$ et $(TS)^3$ : 
\[\PSL(2,\Z)=\langle S,U\mid S^2 \text{ et } U^3\rangle \simeq \langle S,T\mid S^2 \text{ et } (TS)^3\rangle .\]

Bien que le groupe de Cremona possède un système de générateurs connu, ce groupe n'est pas de type fini et ceci est valable sur n'importe quel corps (voir \cite[Proposition 3.6]{Ca}). Lorsque le corps $\kk$ est algébriquement clos, d'après le théorème de Noether-Castelnuovo, le groupe de Cremona est engendré par le groupe des automorphismes du plan projectif isomorphe à $\PGL(3,\kk)$ et par l'involution quadratique standard qui dans une carte affine s'écrit : $\sigma : (x,y)\dashmapsto (\frac{1}{x},\frac{1}{y})$. Un autre système de générateurs qui se trouve être plus pratique dans certaines circonstances est $\PGL(3,\kk)$ ainsi que le sous-groupe de Jonquières. Un élément du groupe de Jonquières est une application préservant le pinceau de droites $\{y=\text{constante}\}$.
À noter que l'application $\sigma$ appartient au groupe de Jonquières.

\subsection*{Action sur un espace hyperbolique}
Considérons le plan hyperbolique $\Hh^2$. 
Un de ses modèles est le demi-plan de Poincaré qui est défini comme le sous-espace du plan complexe constitué des nombres complexes de partie imaginaire strictement positive : 
\[\Hh^2=\{z\in\C\mid \Ima(z)>0\}.\]
Il est muni de la métrique définie de la façon suivante. Pour tous $z_1,z_2\in \Hh^2$, \[\dist(z_1,z_2)=\argcosh\bigg(1+\frac{(\Ima(z_1)-\Ima(z_2))^2+(\Reel(z_1)-\Reel(z_2))^2}{2\Ima(z_1)\Ima(z_2)}\bigg).\]
Le groupe modulaire agit sur $\Hh^2$ par isométries : pour tout $\begin{pmatrix}
a & b\\
c & d
\end{pmatrix}\in\PSL(2,\Z)$ et pour tout point $z\in\Hh^2$,
\[ \begin{pmatrix}
a & b\\
c & d
\end{pmatrix}\cdot z =\frac{az+b}{cz+d}.  \]
Les points à l'infini ou appelés également points au bord sont les points dont la partie imaginaire est nulle ainsi que le point noté $\infty$ qui permet de compactifier la droite réelle $\{y=0\}$.
Pour cette action, les matrices $S$ et $U$ sont des isométries elliptiques puisqu'elles fixent respectivement les points $i\in\Hh^2$ et $\frac{1}{2}(1+i\sqrt{3})\in\Hh^2$.
La matrice $T$ est une isométrie parabolique car elle fixe un unique point au bord, le point $\infty$. Une matrice est hyperbolique si elle fixe deux points à l'infini comme par exemple la matrice $\begin{pmatrix}
1 & 2\\
2 & 5
\end{pmatrix}$ qui fixe les deux points réels $-1-\sqrt{2}$ et $-1+\sqrt{2}$.

Un autre de ses modèles est le modèle de l'hyperboloïde. Il est défini comme la nappe d'hyperboloïde
\[\Hh^2:=\{(x,y,z)\in\R^3\mid x^2-y^2-z^2=1 \text{ et }x>0\},\]
munie de la distance \[\dist(u,v)=\argcosh\B(u,v), \ \text{ pour tous } u,v\in\Hh^2\] où $\B(\cdot,\cdot)$ est la forme bilinéaire associée à la forme quadratique précédente.

Le groupe de Cremona agit également par isométries sur un espace hyperbolique qui est un analogue de dimension infinie au modèle de l'hyperboloïde de $\Hh^2$. L'espace de Picard-Manin associé à $\P^2$ est la limite inductive des groupes de Picard des surfaces obtenues en éclatant toute suite finie de points de $\P^2$, infiniment proches ou non. Il est muni d'une forme d'intersection de signature $(1,\infty)$. En considérant une nappe d'hyperboloïde, nous pouvons lui associer un espace hyperbolique de dimension infinie, noté $\H$.
Les éléments de $\PGL(3,\kk)$ sont tous elliptiques, ceux du groupe de Jonquières sont elliptiques lorsque les degrés de leurs itérées sont bornés et paraboliques sinon. L'application de Hénon $h_n : (x,y)  \mapsto  (y,y^n-x)$ est un exemple d'élément hyperbolique.

\subsection*{Non-simplicité du groupe de Cremona}
Le groupe $\PSL(2,\Z)$ n'est pas un groupe simple. En effet, pour tout entier $N>1$, le sous-groupe 
\[\Gamma(N)=\{A\in \PSL(2,\Z)\mid A\equiv \pm \I_2 \ (\text{mod } N)\},\]
est un sous-groupe distingué de $\PSL(2,\Z)$. En fait, il possède également de nombreux sous-groupes distingués d'indice infini puisque c'est un groupe SQ-universel, c'est-à-dire que tout groupe dénombrable se plonge dans un quotient de $\PSL(2,\Z)$.

En 2013, en faisant agir le groupe de Cremona sur $\H$, S. Cantat et S. Lamy montrent que lorsque le corps est algébriquement clos, le groupe de Cremona n'est pas simple (voir \cite{CL}). Ils construisent des sous-groupes propres dont tous les éléments sont de grands degrés. Pour montrer cela, ils élaborent une variante de la théorie de petite simplification. Récemment, nous avons étendu ce résultat à un corps quelconque (voir \cite{Lo}), en utilisant un théorème de petite simplification dû à F. Dahmani, V. Guirardel et D. Osin (voir \cite{DGO}) établi dans le cadre général de groupes agissant par isométries sur des espaces Gromov-hyperboliques. Nous sommes dans ce cadre d'une part car l'espace $\Hh^2$ est hyperbolique au sens de Gromov. En effet, chacun de ses triangles vérifie la propriété suivante : tout côté est contenu dans le $\ln(1+\sqrt{2})$-voisinage de la réunion de ses deux autres côtés. D'autre part, un triangle de $\H$ vivant dans une copie de $\Hh^2$, l'espace $\H$ est lui aussi $\ln(1+\sqrt{2})$-hyperbolique. En fait, comme conséquence des résultats de F. Dahmani, V. Guirardel et D. Osin, nous obtenons dans \cite{Lo} des propriétés plus précises que la seule non-simplicité, à savoir que pour $\kk$ un corps quelconque, le groupe $\Bir(\P^2_k)$ contient des sous-groupes distingués libres, et est SQ-universel.

\subsection*{Pavage de Voronoï}
L'espace $\H$ est central dans l'étude du groupe de Cremona, cependant un domaine fondamental pour l'action du groupe de Cremona sur $\H$ n'avait jusqu'alors pas été étudié. Un domaine fondamental est un sous-espace fermé de $\H$ tel que son orbite sous l'action du groupe de Cremona recouvre $\H$ et que les éléments de l'orbite de son intérieur soient deux à deux disjoints. En fait, il va s'avérer plus commode de travailler avec les cellules de Voronoï. Le but de cet article est de construire et étudier un pavage de Voronoï associé à cette action. 

Lorsqu'un groupe agit isométriquement et discrètement sur un espace métrique géodésique et qu'il existe un point de l'espace dont le stabilisateur est réduit au neutre, \og les cellules de Voronoï\fg \ sont un outil naturel pour construire un domaine fondamental pour cette action. 
\begin{defi*}
	Soit $\mathcal{P}$ un ensemble discret de points d'un espace métrique géodésique $X$. À tout point $p$ de $\mathcal{P}$ nous associons un ensemble de $X$, noté $\V(p)$ et appelé \emph{cellule de Voronoï} associée au point $p$, constitué des points de $X$ qui sont plus proches de $p$ que des autres points de $\mathcal{P}$ : \[\V(p)=\{x\in X\mid \text{ pour tout } q\in\mathcal{P},\  \dist(x,p)\leq \dist(x,q) \}.\] Les points $p$ sont les \og centres \fg{} des cellules de Voronoï.
\end{defi*}
En effet, en considérant l'orbite d'un point dont le stabilisateur est réduit au neutre, nous obtenons un ensemble discret de points de notre espace métrique. La cellule de Voronoï associée à un point de cette orbite est un domaine fondamental. 

Revenons à l'action du groupe modulaire sur le demi-plan de Poincaré $\Hh^2$. Il est bien connu que l'ensemble des points qui se situent au-dessus de la géodésique dont les points à l'infini sont $-1$ et $1$, et qui ont une partie réelle supérieure ou égale à $-\frac{1}{2}$ et inférieure ou égale à $\frac{1}{2}$ est un domaine fondamental. 
Considérons l'orbite du point $2i$ sous l'action de $\PSL(2,\Z)$. Les cellules de Voronoï correspondant à cet ensemble de points coïncident avec l'orbite du domaine fondamental donné plus haut. 
Cependant, si nous considérons l'orbite du point $i$ sous l'action de $\PSL(2,\Z)$, la cellule de Voronoï correspondant au point $i$ n'est plus le domaine fondamental de l'action du groupe modulaire sur $\Hh^2$, mais l'orbite de ce dernier par le stabilisateur de $i$ qui est isomorphe à $\Large{\sfrac{\Z}{2\Z}}$.

L'objet de la section \ref{section_def_orbl}, est de construire un pavage de Voronoï associé à l'action du groupe de Cremona sur $\H$. Notons $\l\in \H$ la classe d'une droite de $\P^2$. Nous considérons les cellules de Voronoï associées à l'orbite de $\l$ par l'action du groupe de Cremona. Le stabilisateur de $\l$ étant $\PGL(3,\kk)$, nous identifions deux applications qui diffèrent par un automorphisme : 
\[f\sim g \Leftrightarrow \text{ il existe }a\in\PGL(3,\kk), \ f=g\circ a.\]
La cellule de Voronoï associée à $\l$ correspond à l'orbite d'un domaine fondamental sous l'action de $\PGL(3,\kk)$. Remarquons que deux applications appartenant à la même classe d'équivalence donnent la même cellule de Voronoï. Soit $f$ appartenant à $\Bir(\P^2)$, nous notons $\V(f)$ la cellule associée à $f$ et l'application $f$ est appelée \og germe\fg.
Dans cette section nous nous restreignons à un sous-espace convexe de $\H$ qui est essentiellement l'enveloppe convexe de l'orbite de $\l$ sous l'action de $\Bir(\P^2)$. Nous prenons en fait un convexe $\orbl$ un peu plus gros pour des raisons techniques (voir la définition de $\orbl$ à la sous-section \ref{sous_section_restriction}). Nous montrons que les cellules de Voronoï recouvrent l'espace $\orbl$ (Corollaire \ref{cor_pavage}).

Nous étudions ensuite, lors de la section \ref{section_cellule_identite}, les classes appartenant à la cellule de Voronoï associée à l'application identité. Le théorème principal de cette partie dit qu'il suffit que les classes soient plus proches de $\l$ que de l'orbite de $\l$ sous l'action des applications de Jonquières (et non pas de toutes les applications du groupe de Cremona). 
\begin{maintheorem}
Une classe $c$ appartient à la cellule $\V(\id)$ si et seulement si pour toute application de Jonquières $\j$, \[\dist(\isome{\j}(c),\l)\geq \dist(c,\l).\]
\end{maintheorem}
En fait, le théorème \ref{prop_cellule_identite} est plus précis que cela. Il suffit de vérifier l'inégalité pour certaines applications de Jonquières.

Dans la section \ref{section_cellule_adjacente}, nous déterminons les cellules non disjointes de la cellule $\V(\id)$ (Corollaire \ref{cor_cellule_adjacente}) que nous appelons \og cellules adjacentes\fg. Les germes de telles cellules sont de deux types. Ceux qui sont de \emph{caractéristique Jonquières}, c'est-à-dire les applications du groupe de Cremona telles qu'il existe deux points $p$ et $q$ dans $\P^2$ et qui envoient le pinceau de droites passant par le point $p$ sur le pinceau de droites passant par le point $q$. L'autre type de germes des cellules adjacentes à la cellule $\V(\id)$ sont les applications du groupe de Cremona qui possèdent au plus huit points-base en position presque générale. Un ensemble de points $\{p_0,p_1,\dots, p_r\}$ est dit \emph{en position presque générale} si d'une part pour chaque $0\leq i\leq r$ le point $p_i$ vit soit dans $\P^2$ soit dans une surface qui est obtenue en éclatant seulement un sous-ensemble de points de $\{p_0,\dots,p_r\}$ et si d'autre part aucune des trois conditions suivantes n'est satisfaite : quatre des points de cet ensemble sont alignés, sept des points de cet ensemble sont sur une conique, deux des points de cet ensemble sont adhérents à un troisième point de cet ensemble.

\begin{maintheorem}\label{thm_intro_cel_adj}
L'ensemble des germes des cellules adjacentes à la cellule $\V(\id)$ est constitué de toutes :
\begin{itemize}[wide]
\item les applications de caractéristique Jonquières,
\item les applications qui possèdent au plus $8$ points-base en position presque générale.
\end{itemize} 
\end{maintheorem}
Nous déterminons également les classes qui se trouvent à l'intersection entre la cellule de Voronoï associée à $\l$ et une ou plusieurs de ses cellules adjacentes (Théorème \ref{thm_principal_germe_cellule_voisine_identite}).

Enfin dans la section \ref{section_cellule_quasiadjacente} nous étudions les cellules qui possèdent une classe en commun à l'infini avec la cellule associée à l'identité. Revenons un instant à l'action du groupe modulaire sur le demi-plan de Poincaré et considérons le pavage de Voronoï associé à l'orbite du point $2i$. Les cellules quasi-adjacentes à la cellule associée à $2i$ sont toutes les cellules obtenues comme image par le sous-groupe engendré par $z\rightarrow  z+1$ de la cellule associée à $2i$.

Le théorème principal de cette section est 
\begin{maintheorem}\label{thm_intro_cel_quasi_adj}
L'ensemble des germes des cellules quasi-adjacentes à la cellule $\V(\id)$ est constitué de toutes :
\begin{itemize}[wide]
\item les applications de caractéristique Jonquières,
\item les applications qui possèdent au plus $9$ points-base en position presque générale.
\end{itemize} 
\end{maintheorem}
De même que dans le cas précédent, nous étudions les classes à l'infini qui se trouvent dans l'intersection des bords à l'infini de la cellule associée à $\l $ et d'une ou de plusieurs de ses cellules quasi-adjacentes.
\medskip

Une motivation pour cette étude était la question de l’hyperbolicité de certains graphes associés au groupe de Cremona, en particulier un graphe dû à Wright. Par rapport à l'analogie avec le groupe modulaire, ce graphe joue le rôle de l'arbre de Bass-Serre associé à la structure de produit amalgamé de $\PSL(2,\Z)$.
Dans un article à venir (voir également \cite{Lothese}), nous ferons le lien entre le graphe de Wright et le graphe dual du pavage de Voronoï, et nous étudierons son hyperbolicité au sens de Gromov.

\section*{Remerciements}
Je remercie vivement Stéphane Lamy, mon directeur de thèse, pour sa grande disponibilité et ses relectures minutieuses.
Je remercie également les rapporteurs de ma thèse, Charles Favre et Yves de Cornulier, pour leurs remarques qui ont permis de rendre certains passages plus clairs.

\section{Préliminaires}
Dans cet article le corps de base, noté $\kk$, est algébriquement clos. Les surfaces considérées sont projectives et lisses. 

Le \emph{groupe de Cremona} $\Bir(\P^2_{\kk})$ est le groupe des applications birationnelles de $\P^2=\P^2_{\kk}$ vers lui-même. 
Un élément du groupe de Cremona s'écrit : 
\[\begin{array}{cccc}
f : & \P^2 &\dashrightarrow & \P^2\\
& [x:y:z] & \dashmapsto &[f_0(x,y,z):f_1(x,y,z):f_2(x,y,z)],
\end{array}\]
où $f_0,f_1,f_2\in\kk[x,y,z]$ sont des polynômes homogènes, de même degré et sans facteur commun. Nous appelons degré de $f$ le degré des polynômes homogènes: 
\[\deg(f):=\deg(f_i) \text{ pour } i\in\{0,1,2\}.\]
Par exemple, l'application quadratique standard 
\[\sigma : [x:y:z] \dashrightarrow [yz :xz:xy]\] est un élément du groupe de Cremona. C'est une \emph{application de Jonquières} puisqu'elle préserve le pinceau de droites passant par un point. En fait elle préserve trois pinceaux de droites, ceux passant par les points $[0:0:1]$, $[0:1:0]$ et $[1:0:0]$. 
Remarquons que l'inverse d'une application de Jonquières est encore une application de Jonquières.
Soit $S$ une surface. Une surface $S'$ domine $S$ s'il existe un morphisme birationnel allant de $S'$ vers $S$. Considérons $S_1$ et $S_2$ deux surfaces dominant $S$ et $\pi_1$ et $\pi_2$ leur morphisme respectif vers $S$. Nous disons que deux points $p_1\in S_1$ et $p_2\in S_2$ sont équivalents si $\pi_1^{-1}\circ \pi_2$ est un isomorphisme local sur un voisinage de $p_2$ et envoie $p_2$ sur $p_1$. L'\emph{espace des bulles} (\og Bubble space\fg\ en anglais), noté $\B(S)$, est l'union de tous les points de toutes les surfaces dominant $S$ modulo cette relation d'équivalence. 
Soient $S$ une surface et $p\in  S$. Tous les points appartenant au diviseur exceptionnel $E_p$ obtenu en éclatant le point $p$ sont dits \emph{ infiniment proches d'ordre un} de $p$. Un point est \emph{infiniment proche d'ordre $r\geq 2$} de $p$ si c'est un point infiniment proche d'ordre un d'un point infiniment proche d'ordre $r-1$ de $p$.

Le théorème suivant est dû à O. Zariski et permet de décomposer toute application birationnelle comme composée d'éclatements et d'inverses d'éclatements. 
\begin{thm}\label{thm_Zariski}
Soient $S_1$ et $S_2$ deux surfaces, et $f :  S_1\dashrightarrow S_2$ une application birationnelle. Alors il existe une troisième surface $S_3$ et deux composées d'éclatements $\pi : S_3\rightarrow S_1$ et $\sigma : S_3\rightarrow
 S_2$ telles que le diagramme suivant commute :  
	\begin{center}
		\begin{tikzcd}[column sep=small]
			& S_3\arrow{dl}[above left]{\pi} \arrow{dr}{\sigma} & \\
			S_1 \arrow[dashrightarrow]{rr}{f} & & S_2.
		\end{tikzcd}		
	\end{center}
\end{thm}
Les points éclatés lors de la résolution minimale de $f$ sont appelés \emph{points-base} de $f$ et vivent dans l'espace des bulles. Ils correspondent aux points-base du système linéaire associé à $f$. Cet ensemble de points est noté $\Bs(f)$. À tout point-base de $f$ est associé un entier appelé \emph{multiplicité} et qui correspond à la multiplicité du système linéaire en ce point, c'est-à-dire à la plus petite multiplicité en ce point des polynômes définissant $f$.

Le groupe de Cremona agit sur un hyperboloïde dans l'espace de Picard-Manin, noté $\H$. 
Dans ces préliminaires, nous rappelons la définition de $\H$ ainsi que l'action du groupe de Cremona sur $\H$. Nous énonçons aussi des définitions et des propriétés sur les points-base des applications du groupe de Cremona qui nous seront utiles par la suite.

\subsection{Hyperboloïde dans l'espace de Picard-Manin}

Nous rappelons ici la construction de l'espace de Picard-Manin. Nous définissons ensuite l'action du groupe $\Bir(\P^2)$ sur cet espace. Enfin, nous nous intéressons à un sous-espace de l'espace de Picard-Manin qui est un espace hyperbolique de dimension infinie. 
Plus de précisions se trouvent dans \cite[Section 4]{BC}, \cite[Part II.4]{CL}, \cite[Section 3]{C} et \cite[Section 1.2.3]{Lothese}. 

\subsubsection{Espace de Picard-Manin}	
	Soit $S$ une surface. Nous considérons le groupe de Néron-Severi associé à $S$ et tensorisé par $\R$. Nous le notons encore $\NS(S)$. C'est donc le groupe des diviseurs à coefficients réels sur $S$ à équivalence numérique près. Il est muni d'une forme bilinéaire symétrique, la forme d'intersection. Pour tout diviseur $D$ sur $S$ nous notons $\{D\}_S$ sa classe de Néron-Severi ou $\{D\}$ s'il n'y a pas d’ambiguïté sur la surface. Si $\pi : S'\longrightarrow S$ est un morphisme birationnel entre deux surfaces, alors le tiré en arrière \[\pi^* : \NS(S)\hookrightarrow \NS(S')\] qui à la classe d'un diviseur associe la classe de sa transformée totale, est un morphisme injectif qui préserve la forme d'intersection. De plus, $\NS(S')$ est isomorphe à
	\begin{equation*}
	\pi^*(\NS(S))\oplus(\underset{p\in\Bs(\pi^{-1})}{\oplus}\R\{E_p^*\}),
	\end{equation*}
	 où $\Bs(\pi^{-1})$ est l'ensemble des points-base de $\pi^{-1}$ (infiniment proches ou pas) et $E_p^*$ est la transformée totale, vue dans $S'$, du diviseur exceptionnel $E_p$ obtenu en éclatant le point $p$. Cette somme est orthogonale relativement à la forme d'intersection.

	Considérons la limite inductive des groupes de Néron-Severi des surfaces $S'$ dominant $S$ : \[\PM_C(S)=\lim\limits_{\underset{S'\rightarrow S}{\longrightarrow}}\NS(S'),\] 
	où l'indice $C$ fait référence aux b-diviseurs de Cartier (pour plus de précisions voir \cite{Fa}).
	Remarquons que pour toute surface $S'$ dominant $S$, le groupe $\NS(S')$ est plongé dans $\PM_C(S)$.
	En fait, si nous considérons un diviseur $D$ sur $S$, à chaque surface $S'$ dominant $S$, nous pouvons lui faire correspondre une classe de Néron-Severi $\{D \}_{S'}$ dans $\NS(S')$. Ces éléments sont tous identifiés dans $\PM_C(S)$ et correspondent à une classe $d$ de $\PM_C(S)$ notée en lettre minuscule.
	\begin{ex}\label{ex_clasNS}
		Considérons un point $q$ appartenant au diviseur exceptionnel $E_p$, issu de l'éclatement d'une surface $S$ au point $p$. Notons $S_p$ la surface obtenue en éclatant le point $p$ et $S_{p,q}$ celle en éclatant successivement les points $p$ et $q$. La classe $e_p$ correspond à $\{E_p\}_{S_p}$ dans $\NS(S_p)$ et à $\{\tilde{E}_p+E_q\}_{S_{p,q}}$ dans $\NS(S_{p,q})$ où $\tilde{E}_p$ est la transformée stricte de $E_p$ dans $S_{p,q}$.
	\end{ex}
	 Définissons la forme d'intersection sur $\PM_C(S)$. Pour cela, considérons $c$ et $d$ deux éléments de $\PM_C(S)$. Il existe une surface $S_1$ dominant $S$ telles que les classes $c$ et $d$ correspondent respectivement à $\{C\}_{S_1}$ et $\{D\}_{S_1}$ dans $\NS(S_1)$. La forme d'intersection est donnée par : $c \cdot d=\{C\}_{S_1}\cdot \{D\}_{S_1}$. Elle ne dépend pas du choix de la surface $S_1$.

	Par la suite, nous nous intéressons à l'espace de Hilbert défini par  \[\PM(S)=\{\{D_0\}_{S}+\sum\limits_{p\in\B(S)}\lambda_pe_p\mid \lambda_p\in\R,\ \sum\limits_{p\in\B(S)}\lambda_p^2<\infty \text{ et }\{D_0\}_{S} \in\NS(S)\},\]
	que nous appelons l'\emph{espace de Picard-Manin} (voir \cite{CL} et \cite{C} ou encore \cite{BFJ}). C'est le complété $L^2$ de $\PM_C(S)$. Ses éléments sont appelés \og classes de Picard-Manin\fg \ ou plus simplement \og classes\fg. Les classes $e_p$ (nous gardons les notations introduites dans l'exemple \ref{ex_clasNS}) où $p$ est un point de $S$ ou d'une surface dominant $S$, sont d'auto-intersection $-1$, orthogonales deux à deux et orthogonales à $\NS(S)$. La forme d'intersection est donc de signature $(1,\infty)$ et préserve la décomposition orthogonale :
	\begin{equation*}
	\PM(S)=\NS(S)\oplus \big(\underset{p\in\B(S)}{\bigoplus}\R e_p\big).
	\end{equation*}

 	Tout  morphisme birationnel $\pi:S'\rightarrow S$ induit un isomorphisme $\isome{\pi}$ de $\PM(S')$ vers $\PM(S)$ qui consiste à considérer comme exceptionnelles au-dessus de $S$ les classes $e_q$ qui étaient dans l'espace de Néron-Severi de $S'$ : 
 	 \[\begin{array}{rrcc}
 		\isome{\pi}^{-1} :& \PM(S)=\NS(S)\oplus (\underset{p\in\B(S)}{\oplus}\R e_p) & \longrightarrow & \PM(S')=\NS(S')\oplus (\underset{\substack{r\in\B(S)\\r\notin\Bs(\pi^{-1})}}{\oplus}\R e_r)\\
 	&	\{D_0\}_{S}+\sum\limits_{p\in\B(S)}\lambda_pe_p  & \mapsto &  \left(\{\tilde{D_0}\}_{S'}+\sum\limits_{q\in\Bs(\pi^{-1})}(m_q(D_0)+\lambda_q)e_q\right)\\
 	&	& & \hspace{4cm}+\sum\limits_{\substack{r\in\B(S)\\r\notin\Bs(\pi^{-1})}}\lambda_re_r,
 		\end{array}\] 
 		où $\tilde{D}_0$ est la transformée stricte dans $S'$ de $D_0$ et $m_q(D_0)$ la multiplicité de $D_0$ au point $q$.
 	
Dans le cas où la surface considérée est $\P^2$, nous notons simplement $\PM$ l'espace de Picard-Manin associé : 

\[\PM=\{n\l+\sum\limits_{p\in\B(\P^2)}\lambda_pe_p\mid n,\lambda_p\in\R,\ \sum\limits_{p\in\B(\P^2)}\lambda_p^2<\infty \},\] où $\l$ est la classe de la droite dans $\P^2$.	

\subsubsection{Forme canonique}
Notons $\PM_{L^1}$ l'ensemble des classes de Picard-Manin qui sont $L^1$ :
\[\PM_{L^1}=\{c=n\l+\sum\limits_{p\in\B(\P^2)}\lambda_pe_p\in\PM \mid \sum\limits_{p\in\B(\P^2)}\lvert\lambda_p\rvert<\infty\}.\] 
Nous définissons la forme canonique $\fcan$ sur $\PM_{L^1}$, comme la forme linéaire définie par :
\[\text{ pour tout } c\in \PM_{L^1},\   \fcan(c)=\fcan\cdot c= -3n-\sum\limits_{p\in\B(\P^2)}\lambda_p,\]
où $\fcan=-3\l+\sum\limits_{p\in\B(\P^2)}e_p$. Plus généralement, en notant $\can_S$ le diviseur canonique d'une surface $S$ dominant $\P^2$, la forme canonique s'écrit $\kk_S=\can_S+\sum\limits_{p\in\B(S)}e_p$.

	\subsubsection{Action du groupe de Cremona sur l'espace de Picard-Manin}	
Considérons une résolution de $f\in\Bir(\P^2)$ :
\begin{center}
	\begin{tikzcd}[column sep=small]
		& S\arrow{dl}[above left]{\pi} \arrow{dr}{\sigma} & \\
		\P^2 \arrow[dashrightarrow]{rr}{f} & & \P^2.
	\end{tikzcd}		
\end{center}
Le groupe $\Bir(\P^2)$ agit sur $\PM(\P^2)$ via l'application $(f,c)\mapsto \isome{f}(c)$ où $\isome{f}$ est définie par \[\isome{f}=\isome{\sigma}\circ (\isome{\pi})^{-1}.\] Remarquons que $(\isome{f})^{-1}=\isome{(f^{-1})}$ et que l'action de $f$ préserve la forme d'intersection : 
\[\text{pour tous }c_1,c_2\in\PM,\ \isome{f}(c_1)\cdot \isome{f}(c_2)=c_1\cdot c_2.  \]

\begin{rmq}\label{rmq_canonique_stable_orbite}
La forme canonique est constante sur l'orbite d'une classe par le groupe de Cremona, c'est-à-dire pour toute classe $c\in\PM_{L^1}$ et pour tout $f\in\Bir(\P^2)$, nous avons :
\[\fcan\cdot c=\fcan\cdot\isome{f}(c).\]
\end{rmq}

Soit $d$ le degré de $f$ et notons $p_0,p_1,\dots, p_{r-1}$ ses points-base de multiplicité respective $\{m_i\}_{0\leq i\leq r-1}$ et $q_0,q_1,\dots q_{r-1}$ ceux de $f^{-1}$ de multiplicité $\{m_i'\}_{0\leq i\leq r-1}$.
L'action de $f$ sur $\l$ et sur les classes $\{e_{p_j}\}_{0\leq j\leq r-1}$ est donnée par : \begin{align}
\isome{f}(\l) & =d\l-\sum\limits_{i=0}^{r-1}m_i'e_{q_i},\label{action_f_sur_l}\\
\isome{f}(e_{p_j}) & =m_j\l-\sum\limits_{i=0}^{r-1}a_{i,j}e_{q_i}.\label{action_f_sur_pt_base}
\end{align}
Remarquons que les coefficients $a_{i,j}$ correspondent au nombre d'intersection des transformées totales des diviseurs exceptionnels obtenus en éclatant respectivement les points $p_j$ et $q_i$, dans la résolution de $f$. 

Par exemple, soit $\j$ une application de Jonquières de degré $d>1$. Notons $p_0$ et $q_0$ les points-base maximaux respectifs de $\j$ et $\j^{-1}$ et $p_1,\dots,p_{2d-2}$ et $q_1,\dots,q_{2d-2}$ les petits points-base respectifs de $\j$ et $\j^{-1}$, nous avons : 
	\begin{equation}\label{lemme_action_jonq_l}
	\begin{cases}
		\isome{\j}(\l)=d\ell-(d-1)e_{q_0}-\sum\limits_{i=1}^{2d-2}e_{q_i}\\
		\isome{\j}(e_{p_0})=(d-1)\ell-(d-2)e_{q_0}-\sum\limits_{i=1}^{2d-2}e_{q_i}\\
		\isome{\j}(e_{p_i})=\l-e_{q_0}-e_{q_i} \text{ pour } 1\leq i\leq 2d-2\\
	\end{cases}.
	\end{equation}

\begin{rmq}\label{rmq_aij_positifs}
D'après \cite[Proposition 2.2.21]{AC}, les coefficients $a_{i,j}$ sont positifs pour tous $0\leq i,j\leq r-1$.
\end{rmq}

\begin{rmq}\label{rmq_systeme_lineaire_action_sur_l_pt_base}
L'action de $f$ sur $\l$ correspond au système linéaire associé à la transformée par $f^{-1}$ d'une droite ne passant pas par les points-base de $f$. Si tous les points-base de $f$ sont dans $\P^2$ alors pour chaque point $p_j$ il existe une courbe contractée par $f^{-1}$ sur le point $p_j$ de degré $m_j$ et passant avec multiplicités $a_{i,j}$ aux points $q_j$.  
\end{rmq}
Ces informations sur $f$ se lisent dans sa \emph{matrice caractéristique} : 
\[\begin{pmatrix}
d & m_0 & m_1 & \dots & m_{r-1}\\
-m_0' & - a_{0,0}& - a_{0,1} & \ldots & - a_{0,r-1}\\
-m_1' & - a_{1,0}& - a_{1,1} & \ldots & - a_{1,r-1}\\
\vdots &\vdots & \vdots & \ddots &\vdots \\
-m_{r_1}' & - a_{r-1,0}& - a_{r-1,1} & \ldots & - a_{r-1,r-1}
\end{pmatrix}.\]

Celles pour $f^{-1}$ se lisent en ligne en changeant le signe des $m_i$ et des $m_i'$ : 
\begin{align*}
\isome{f}^{-1}(\l)& =d\l-\sum\limits_{j=0}^{r-1}m_je_{p_j},\\
\isome{f}^{-1}(e_{q_i}) & =m_i'\l-\sum\limits_{j=0}^{r-1}a_{i,j}e_{p_j}.
\end{align*}

Pour plus de détails sur la matrice caractéristique, nous renvoyons à \cite[Section 2.4]{AC}. La première ligne de la matrice caractéristique de $f$ s'appelle la \emph{caractéristique} de $f$. Elle est notée $(d;m_0,\dots,m_{r-1})$.

Il ne nous reste plus qu'à regarder l'action de $f$ sur les classes $e_p$ où $p\in\B(\P^2)$ n'est pas un point-base de $f$.
\begin{rmq}\label{rmq_action_cremona_sur_classe_exceptionnelles}
 Si $f$ est un isomorphisme d'un voisinage $U$ de $p\in\P^2$ sur un voisinage $V$ de $f(p)\in\P^2$ alors \[\isome{f}(e_p)=e_{f(p)}.\]
 Sinon, quitte à éclater des points-base de $f$ et des points-base de $f^{-1}$, il existe un point $q$ sur une surface dominant $\P^2$ tel que l'application induite par $f$ envoie $p$ sur $q$ et est un isomorphisme local entre des voisinages de $p$ et de $q$ et donc 
	\[\isome{f}(e_p)=e_q.\] 
\end{rmq}

Considérons une classe $c$ de l'espace de Picard-Manin $\PM$ : \[c=n\l-\sum\limits_{i=0}^{r-1}\lambda_ie_{p_i}-\sum\limits_{\substack{p\in\B(\P^2)\\p\notin\supp(f)}}\lambda_pe_{p}.\] L'action de $f$ sur une classe $c$ de l'espace de Picard-Manin s'obtient par linéarité.
\begin{align}
\isome{f}(c)&=n\left(d\l-\sum\limits_{i=0}^{r-1}m_i'e_{q_i}\right)-\sum\limits_{j=0}^{r-1}\lambda_j\left(m_j\l-\sum\limits_{i=0}^{r-1}a_{i,j}e_{q_i}\right)-\sum\limits_{\substack{p\in\B(\P^2)\\p\notin\supp(f)}}\lambda_p\isome{f}(e_{p})\nonumber \\
& = \left(nd-\sum\limits_{j=0}^{r-1}\lambda_jm_j\right)\l-\sum\limits_{i=0}^{r-1}\left(nm_i'-\sum\limits_{j=0}^{r-1}\lambda_ja_{i,j}\right)e_{q_i}-\sum\limits_{\substack{p\in\B(\P^2)\\p\notin\supp(f)}}\lambda_p\isome{f}(e_{p}). \label{eq_action_f_sur_c}
\end{align}

	\subsubsection{Espace hyperbolique de dimension infinie dans l'espace de Picard-Manin}\label{sous-section_hyperobloide_PM}
	
	À présent, considérons l'espace \[\H(S)=\{c\in \PM(S)\mid c\cdot c=1 \text{ et } c\cdot d_0>0\},\] où $d_0\in\NS(S)$ est une classe ample. Muni de la distance définie par $\dist(c,c')=\argcosh(c\cdot c')$ pour tous $c,c'\in \H(S)$, c'est un espace hyperbolique de dimension infinie. Nous nous intéressons plus particulièrement à $\H(\P^2)$ que nous notons $\H$. Tout élément de $\H$ est de la forme \[n\ell +\sum_{p\in\B(\P^2)}\lambda_pe_p \text{ où } n>0 \text{ et } n^2-\sum_{p\in\B(\P^2)}\lambda_p^2=1.\]
	Comme le groupe de Cremona agit sur l'espace de Picard-Manin et que l'action préserve la forme d'intersection, pour montrer que le groupe de Cremona agit sur $\H$, il suffit de montrer que $\isome{f}(c)\cdot \l>0$. Comme $\isome{f}^{-1}(\l)\in\H$ et que le nombre d'intersection entre deux classes est supérieur ou égal à $1$, nous avons comme attendu : \[\isome{f}(c)\cdot \l=c\cdot \isome{f}^{-1}(\l)\geq 1.\]

\subsection{Propriétés des applications du groupe de Cremona}
Dans cette sous-section, nous nous intéressons aux propriétés que vérifient les points-base et les multiplicités des applications du groupe de Cremona. Pour cela, nous introduisons dans un premier temps un vocabulaire général utile pour les points-base d'une application, mais pas seulement. Enfin, nous nous concentrons sur les applications de Jonquières. 
\subsubsection{Vocabulaire}\label{sous-section_voc} 
Soient $\pi : S'\rightarrow S$ une composée d'éclatements, avec $p$ l'un des points éclatés dans cette suite, et $E_p$ la transformée stricte sur $S'$ du diviseur exceptionnel obtenu en éclatant le point $p$. Tout point $q$ de $S'$ et appartenant à $E_p$ est dit \emph{adhérent} à $p$. Nous notons cette relation $q \rightarrow p$. Plus généralement, tout point $q$ de $S'$ et appartenant à la transformée totale $\pi^{*}(E_p)$ de $E_p$, est dit \textit{voisin} de $p$.
Considérons un ensemble $P$ de points de $\B(\P^2)$. Nous disons qu'il est \emph{pré-consistant} si tous ses points vivent sur une surface obtenue en éclatant uniquement des points de l'ensemble $P$ à partir de $\P^2$ (voir l'exemple \ref{ex_non_consistant} ci-dessous pour un exemple d'ensemble qui n'est pas pré-consistant). Tout point de $P$ qui est adhérent à un seul point de $P$ est dit \textit{libre} contrairement à un point adhérent à deux points distincts de $P$ qui est appelé \textit{satellite}. Ces deux dernières définitions sont relatives à la donnée d'un ensemble pré-consistant de points. Remarquons qu'un point peut être adhérent à au plus deux points. Un point adhérent et libre est habituellement appelé voisin du premier ordre, mais nous n'utilisons pas cette terminologie. Nous disons qu'une suite pré-consistante de points forme une \emph{tour} si chaque point de la suite est adhérent au point précédent.

\begin{ex}\label{ex_def_pt_adh_voisin} Soit $p_0$ un point de $\P^2$. Éclatons-le. Notons $p_1$ un point sur le diviseur exceptionnel $E_{p_0}$ obtenu en éclatant le point $p_0$. Éclatons le point $p_1$. Considérons deux points distincts $p_2$ et $p_3$ sur le diviseur exceptionnel $E_{p_1}$ tel que $p_2$ appartienne également à la transformée stricte de $E_{p_0}$. 
	\begin{center}
		\scalebox{0.7}
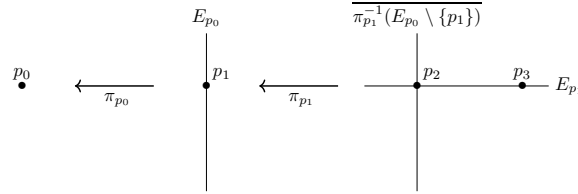
{
			\begin{tikzpicture}
			\def\ymin{0}	\def\ymax{3}
			\def\xmin{-1}	\def\xmax{2.5}
			
			\draw (\xmin, 2) -- (\xmax, 2) node[right] {$E_{p_1}$};
			\draw (0, \ymin) -- (0, \ymax) node[above] {$\overline{\pi_{p_1}^{-1}(E_{p_0}\setminus\{p_1\})}$};
			\draw (0,2) node {$\bullet$} node[above right] {$p_2$};
			\draw (2,2) node {$\bullet$} node[above] {$p_3$};
			
			\draw (-4, \ymin) -- (-4, \ymax) node[above] {$E_{p_0}$};
			\draw (-4,2) node {$\bullet$} node[above right] {$p_1$};
			\draw[thick, <-] (-3,2) -- (-1.5,2) node[below, midway] {$\ \pi_{p_1}$};
			\draw[thick, <-] (-6.5,2) -- (-5,2) node[below, midway] {$\ \pi_{p_0}$};
			\draw (-7.5,2) node {$\bullet$} node[above] {$p_0$};
			\end{tikzpicture}}
	\end{center}
L'ensemble des points $\{p_0,p_1,p_2,p_3\}$ est pré-consistant. Les points $p_1$, $p_2$ et $p_3$ sont voisins de $p_0$. Le point $p_1$ est un point libre adhérent à $p_0$ alors que le point $p_2$ est satellite et adhérent aux points $p_0$ et $p_1$. Le point $p_3$ est libre non adhérent à $p_0$ et il est adhérent à $p_1$.

Pour plus de facilité, nous représentons ce genre de situation par un graphe où les sommets sont les points. Et il y a une flèche orientée entre deux sommets si celui du dessus est adhérent à celui du dessous. Ainsi la situation précédente se réécrit :

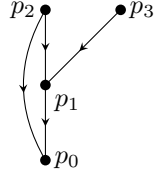
\begin{center}
\begin{tikzpicture}
\coordinate (p0) at (0,0);
\coordinate (p1) at (0,1);
\coordinate (p2) at (0,2);
\coordinate (p3) at (1,2);
\fill (p0) circle (2pt);
\fill (p1) circle (2pt);
\fill (p2) circle (2pt);
\fill (p3) circle (2pt);
\node at (p0) [right] {$p_0$};
\node at (p1) [below right] {$p_1$};
\node at (p2) [left] {$p_2$};
\node at (p3) [right] {$p_3$};
\draw [directed](p1) to (p0);
\draw [directed](p2) to (p1);
\draw [directed](p3) to (p1);
\draw[directed] (p2) to[bend right=30] (p0);
\end{tikzpicture}
\end{center}
\end{ex}

Trois points ou plus sont dits \emph{alignés} s'ils appartiennent à une droite de $\P^2$ ou à la transformée stricte d'une droite de $\P^2$.

Une \emph{suite pondérée} est la donnée de points de $\B(\P^2)$ où chacun est muni d'une multiplicité réelle positive. Soit $p$ un point d'une suite pondérée $P$. La différence entre la multiplicité $m_p$, associée au point $p$, et la somme des multiplicités des points de $P$ adhérents à $p$ est appelée \emph{l'excès} du point $p$ : \[m_p-\sum\limits_{\substack{q\in P\\q\rightarrow p}}m_q.\] Une suite pondérée est dite \emph{consistante} si l'excès en tout point de cet ensemble est positif. Remarquons qu'une suite de points est pré-consistante s'il est possible de pondérer tous les points de cet ensemble de façon strictement positive de sorte que si nous associons à tout autre point de $\B(\P^2)$ une multiplicité nulle alors l'excès en tous les points de $\B(\P^2)$ est positif ou nul. 
\begin{ex}\label{ex_non_consistant}
Si nous reprenons l'exemple précédent \ref{ex_def_pt_adh_voisin}, l'ensemble des points $\{p_0,p_1,p_2,p_3\}$ est pré-consistant car ils peuvent être respectivement pondérés par $3,2,1,1$ et alors tous les points de $\B(\P^2)$ ont un excès positif ou nul. Par contre les points $p_0,p_2,p_3$ ne forment pas un ensemble de points pré-consistant. En effet, quelque soit le choix de multiplicité associé aux points $p_2$ et $p_3$ l'excès au point $p_1$ est toujours strictement négatif. Géométriquement, cela correspond au fait que toute surface dominant $\P^2$ et contenant les points $p_2$ et $p_3$ a été obtenue par une suite d'éclatements dont l'un est l'éclatement du point $p_1$, qui n'appartient pas à l'ensemble ${p_0,p_2,p_3}$.
\end{ex} 

\subsubsection{Propriété des points-base et de leur multiplicité}
Cette partie regroupe plusieurs résultats de base qui seront utilisés par la suite. Une application du groupe de Cremona est dite de caractéristique Jonquières si sa caractéristique est la même que celle d'une application de Jonquières à savoir : $(d;d-1,1^{2d-2})$ où $d$ est le degré de l'application. En fait ce sont des applications qui envoient un pinceau de droites sur un autre pinceau de droites.
Les relations suivantes sont obtenues à partir de simples calculs d'intersection. La plupart d'entre elles sont connues et peuvent par exemple se trouver dans \cite{AC}.

\begin{lemme}\label{lemme_equations_evidentes} Soit $f\in\Bir(\P^2)$. Notons $(d;m_0,\dots, m_{r-1})$ la caractéristique de $f$, $(d;m_0',\dots, m_{r-1}')$ celle de $f^{-1}$ et $(a_{i,j})_{0\leq i,j\leq r-1}$ les coefficients de la sous-matrice de la matrice caractéristique ne contenant pas la première ligne et la première colonne. Nous avons les relations :
\begin{enumerate}
\item \label{eq_Noethercanonique} $\sum\limits_{i=0}^{r-1}m_i=3d-3$.
\item\label{eq_dcarre_micarre} $\sum\limits_{i=0}^{r-1}m_i^2=d^2-1$. 
\item\label{eq_mi_pluspetit_d}  Pour tout $0\leq i\leq r-1$, $m_i\leq d-1$.
\item\label{eq_m_0_jonquieres} Si $m_0=d-1$ alors $f$ est une application de caractéristique Jonquières. 
\item\label{eq_aij} $\sum\limits_{i=0}^{r-1}a_{i,j}=3m_j-1.$ 
\item\label{eq_mjcarre} $ \sum\limits_{i=0}^{r-1}a_{i,j}^2=m_j^2+1$.
\item \label{eq_produitmimj}$ \sum\limits_{i=0}^{r-1}a_{i,j}a_{i,k}=m_jm_k $.
\item\label{eq_produitdmj} $\sum\limits_{i=0}^{r-1}m_i'a_{i,j}=dm_j$. 
\end{enumerate}
\end{lemme}
    
\begin{rmq*}
Les équations \ref{eq_Noethercanonique} et \ref{eq_dcarre_micarre} sont souvent appelées \og équations de Noether\fg. Une façon d'énoncer la seconde égalité est de dire que le système linéaire associé à $f$ est homaloïdal, c'est-à-dire que deux courbes de ce système ont un unique point d'intersection hors des points-base.
\end{rmq*}

\begin{proof}
Les égalités \ref{eq_Noethercanonique} et \ref{eq_aij} s'obtiennent respectivement en intersectant la forme canonique contre $\isome{f}^{-1}(\l)$ et $\l$ puis contre $\isome{f}(e_{p_j})$ et $e_{p_j}$ et en utilisant la remarque \ref{rmq_canonique_stable_orbite}.
Auto-intersecter $\isome{f}^{-1}(\l)$ et $\isome{f}(e_{p_j})$, puis intersecter $\isome{f}(e_{p_j})$ contre $\isome{f}(e_{p_k})$ et $\isome{f}(e_{p_j})$ contre $\isome{f}(\l)$ nous donne respectivement \ref{eq_dcarre_micarre}, \ref{eq_mjcarre}, \ref{eq_produitmimj} et \ref{eq_produitdmj}. Le point \ref{eq_mi_pluspetit_d} découle du fait que la pré-image par $f$ d'une droite générale est irréductible.
Étudions le point \ref{eq_m_0_jonquieres}. Si $m_0=d-1$ alors en utilisant les égalités \ref{eq_Noethercanonique} et \ref{eq_dcarre_micarre} nous obtenons que \[\sum\limits_{i=1}^{r-1}m_i(m_i-1)=0,\] par conséquent, les multiplicités $m_i$ pour $1\leq i\leq r-1$ sont toutes égales à $1$. En utilisant à nouveau le point \ref{eq_Noethercanonique}, nous obtenons que $r=2d-1$ ce qui permet de conclure que $f$ est de caractéristique Jonquières. 
\end{proof}

La preuve du lemme suivant est une application directe des relations de Noether vues dans le lemme \ref{lemme_equations_evidentes}.
\begin{lemme}\label{lemme_nbr_pt_base_degre} Soit $f\in\Bir(\P^2)$ de degré $d$ et admettant $r$ points-base.
    	\begin{enumerate}
    		\item\label{lemme_nbre_pt_base_degre4} Si $d\geq 4$ alors $r\geq 6$.
    		\item\label{lemme_nbre_pt_base8_degre} Si $r\leq 8$ alors $d\leq17$.
    	\end{enumerate}
    \end{lemme}

    \begin{proof}
    	Soient $\{m_i\}_{0\leq i\leq r-1}$ les multiplicités des points-base de $f$. L'inégalité de Cauchy-Schwarz donne :\[\bigg(\sum\limits_{i=0}^{r-1}m_i\bigg)^2\leq r\sum\limits_{i=0}^{r-1}m_i^2 .\] Ainsi d'après les équations de Noether \ref{eq_Noethercanonique} et \ref{eq_dcarre_micarre} du lemme \ref{lemme_equations_evidentes}, nous obtenons : 
    	\[9(d-1)^2\leq r(d^2-1),\]
    	qui implique en simplifiant par le facteur $d-1$	
    	\[9(d-1)\leq r(d+1).\]
    	Comme la suite $u_d= \frac{9(d-1)}{d+1}$ est croissante, si $d\geq4$ alors \[5<u_4\leq u_d\leq r\] et $f$ a donc au moins six points-base. 
    	
    	Si $r\leq 8$ alors $9(d-1)\leq 8d+8$ qui nous donne $d\leq 17$ comme attendu.
    \end{proof}

Une application $f\in \Bir(\P^2)$ est dite \emph{symétrique} si elle est de degré strictement supérieur à $1$ et si toutes ses multiplicités sont égales. 

\begin{lemme}[{\cite[Lemma 2.5.5]{AC}}]\label{lemme_appli_sym}
La caractéristique d'une application symétrique $f\in \Bir(\P^2)$ est forcément de l'une des formes suivantes :
$(2;1^3)$, $(5;2^6)$, $(8;3^7)$ ou $(17;6^8)$. En particulier, la multiplicité $m$ des points-base de $f$ est liée au degré $d$ par la relation \[m=\frac{d+1}{3}.\]
\end{lemme}

\begin{rmq}
Soit $f\in\Bir(\P^2)$ de degré $d$.
La somme des trois plus grandes multiplicités de ses points-base est supérieure ou égale à $d+1$. En fait, il y a égalité si et seulement si $f$ est une application symétrique ou de caractéristique Jonquières. Cela s'obtient en utilisant uniquement les égalités \ref{eq_Noethercanonique} et \ref{eq_dcarre_micarre} du lemme \ref{lemme_equations_evidentes}.
\end{rmq}

\begin{prop}[Positivité des excès]\label{prop_consistence}
Soit $f\in\Bir(\P^2)$ les excès des points-base de $f$ sont tous positifs : 
\[\text{ pour tout } p\in\Bs(f),\  m_p\geq\sum\limits_{\substack{q\in\Bs(f)\\q\rightarrow p}}m_q.\]
\end{prop}
\begin{proof}
Nous avons vu que la multiplicité $m_p$ d'un point-base $p$ de $f$ correspond à la multiplicité au point $p$ des courbes du système linéaire associé à $f$. Considérons une courbe $C$ générale de ce système linéaire.
D'après \cite[p.252]{Sha}, la multiplicité de $C$ au point $p$, notée $m_p(C)$ est supérieure ou égale à la somme des multiplicités de cette courbe aux points adhérents à $p$. Par conséquent, nous avons l'égalité attendue :
\[m_p=m_p(C)\geq \sum\limits_{\substack{q\in\B(\P^2)\\q\rightarrow p}}m_q(C)\geq \sum\limits_{\substack{q\in\Bs(f)\\q\rightarrow p}}m_q.\qedhere\]
\end{proof}

\begin{prop}[Bézout]\label{prop_Bezout} Soit $f\in\Bir(\P^2)$ de degré $d$ dont les points-base sont les points $\{p_i\}_{0\leq i\leq r-1}$ de multiplicité respective $\{m_i\}_{0\leq i\leq r-1}$.
Pour toute courbe $C'$ de degré $d'$ passant par les points $p_i$ avec multiplicité $\mu_i$ nous avons :
\[dd'-\sum\limits_{i=0}^{r-1}m_i\mu_i\geq 0.\]
\end{prop}

\begin{rmq}\label{rmq_pas_deuxpoints_adherents_petit_jonquieres}
Une application de caractéristique Jonquières de degré $d$ et de point-base de multiplicité maximale $p_0$, ne possède pas deux points-base adhérents à un même troisième point-base différent de $p_0$ sinon cela contredirait la positivité des excès (Proposition \ref{prop_consistence}) en ce point puisque sa caractéristique est $(d;d-1,1^{2d-2})$. De plus, par le même argument, il y a au plus $d-1$ points adhérents au point de multiplicité maximale. 
Par la proposition \ref{prop_Bezout}, deux points de multiplicité $1$ ne peuvent pas être alignés avec le point de multiplicité maximale.
\end{rmq}

Un ensemble de points de $\B(\P^2)$ est dit \emph{en position presque générale} s'il est pré-consistant et qu'aucune des situations suivantes n'est satisfaite :\begin{itemize}[wide]
    		\item quatre des points de cet ensemble sont alignés,
    		\item sept des points de cet ensemble sont sur une conique,
    		\item deux des points de cet ensemble sont adhérents à un troisième point.
    	\end{itemize}
 C'est en fait une terminologie utilisée par Dolgachev \cite[p.397]{Dol}. Cependant, comme il s'intéresse aux surfaces del Pezzo, il impose en plus le fait qu'il y ait au plus $8$ points.
 Une surface $S$ est \emph{faiblement del Pezzo} si son diviseur canonique est numériquement effectif et big, c'est à dire que pour toute courbe $C$ de $S$ le nombre d’intersection entre la courbe $C$ et le diviseur canonique est positif ou nul, et que ce dernier est d'auto-intersection strictement positive : \[C\cdot \can_S \geq 0 \text{ et } \can_S^2>0.\]

\begin{prop}[{\cite[Corollary 8.1.17]{Dol}}]\label{prop_weak_del_Pezzo_presque_general}
Une surface rationnelle est faiblement del Pezzo si et seulement si elle a été obtenue en éclatant $k\leq 8$ points de $\P^2$ en position presque générale.
\end{prop}

\begin{lemme}\label{lemme_config_f_config_inverse}
	Soit $f\in\Bir(\P^2)$ une application ayant au plus $8$ points-base. Les points-base de $f$ sont en position presque générale si et seulement si les points-base de $f^{-1}$ le sont.
\end{lemme}
 \begin{proof}
D'après la proposition \ref{prop_weak_del_Pezzo_presque_general}, une surface est faiblement del Pezzo si et seulement si la suite de points éclatés satisfait les conditions du lemme. La surface obtenue en éclatant les points-base de $f$ et de $f^{-1}$ étant la même nous obtenons le résultat.
 \end{proof}

\subsubsection{Générateurs du groupe de Cremona}\label{subsection_Jonquieres}

Soient $f\in\Bir(\P^2)$ et $p_0$ un de ses points de multiplicité maximale $m_0$.
En suivant la terminologie de \cite{AC}, nous appelons \emph{complexité} de $f$ le nombre \[\car_{f}=\frac{d-m_0}{2}.\] Un point-base $p$ de $f$ différent de $p_0$ est dit \emph{majeur} si sa multiplicité $m_p$ est strictement supérieure à la complexité de $f$ : \[m_p>\frac{d-m_0}{2}.\] Remarquons que si $f$ possède plusieurs points de multiplicité maximale, si nous changeons le choix du point maximal, la complexité de $f$ reste identique ainsi que le nombre de points-base majeurs. Nous notons $\Maj(f)$ cet ensemble et $h$ son cardinal.

	L'involution standard de Cremona, qui s'écrit dans une carte locale : \[\sigma : (x,y) \dashrightarrow (\frac{1}{x},\frac{1}{y})\] est un exemple d'application qui appartient au groupe de Jonquières.
	Elle a une importance particulière dans le groupe de Cremona, puisque comme le corps de base est algébriquement clos elle engendre avec le groupe des automorphismes de $\P^2$ tout le groupe. Le théorème qui suit n'est plus valable si le corps de base n'est pas algébriquement clos. 
	
	\begin{thm}[M. Noether et G. Castelnuovo \cite{Castelnuovo}]\label{thm_générateurs_Cremona}
	Le groupe de Cremona $\Bir(\P^2)$ est engendré par $\sigma$ et $\PGL(3,\kk)$.
	\end{thm}
	M. Noether fut le premier a énoncer ce théorème à la fin du 19\ieme \ siècle. Cependant, la première preuve exacte est due à G. Castelnuovo. L'idée de \og la preuve\fg \  de M. Noether est la suivante. Considérons une application $f$. En pré-composant par une application quadratique dont les trois points-base sont trois points-base de $f$ de plus grande multiplicité, le degré de $f$ diminue. En réitérant ce procédé le degré devient $1$ et la composée est une application linéaire. Cependant il n'existe pas toujours une telle application quadratique. En effet, une application quadratique ne peut pas avoir deux de ses points-base adhérents au troisième (Remarque \ref{rmq_pas_deuxpoints_adherents_petit_jonquieres}). Or il est possible que les points-base de multiplicité maximale d'une application soit dans cette configuration. C'est le cas notamment des automorphisme polynomiaux, comme par exemple $(x,y)\mapsto (y^3-x,y)$. Ce problème n'avait pas été vu par M. Noether car le formalisme des points adhérents n'existait pas.
	Dans \cite{Alex}, J. Alexander corrige la preuve de M. Noether en introduisant la complexité d'une application. S'il existe une application quadratique dont les points-base sont le point-base maximal de $f$ et deux points-base majeurs alors la complexité diminue strictement. Sinon, il faut utiliser un troisième point qui vit dans $\P^2$. Dans ce cas, la complexité reste identique mais le nombre de points-base majeurs diminue. La récurrence se fait sur ces deux entiers positifs. 
	
 	Une conséquence de ce théorème est que la réunion du groupe de Jonquières et du groupe des automorphismes de $\P^2$ est également un système de générateurs du groupe de Cremona. Ce résultat est plus faible mais parfois plus maniable. Il a également l'avantage de pouvoir décomposer une application $f\in\Bir(\P^2)$ en produit de transformations de Jonquières de sorte que le degré augmente à chaque pré-composition.

\begin{thm}[{\cite[Theorem 8.3.4]{AC}}]\label{thm_caraminana}
Toute application de Cremona $f$ est composée d'applications de Jonquières $\j_i$ et d'un élément $a\in\PGL(3,\kk)$ :
\[f=a\circ \j_n\circ\cdots\circ\j_1,\]
de sorte qu'à chaque pré-composition par une application de Jonquières le degré augmente strictement, pour tout $1\leq i\leq n-1$ : 
\[\deg(a\circ\j_n\circ\cdots\circ\j_{i+1}\circ\j_i)>\deg(a\circ\j_n\cdots\circ\j_{i+1}).\]
\end{thm}
Dans {\cite[Theorem 8.3.4]{AC}}, M. Alberich-Carrami\~{n}ana n'énonce pas la seconde partie du théorème, qu'elle démontre pourtant. 
Il est facile de voir qu'en pré-composant $f$ par une application de Jonquières dont l'inverse a pour point-base maximal le point-base de $f$ de multiplicité maximale, et ses petits points-base sont parmi les points-base majeurs de $f$, alors le degré diminue strictement. La partie difficile est de montrer qu'il existe une telle application de Jonquières $\j$. Elle montre que c'est le cas si les petits points-base de $\j^{-1}$ sont choisis comme étant tous les points-base majeurs de $f$, à l'exception d'un lorsque le cardinal des points-base majeurs de $f$ est impair. L'algorithme qu'elle utilise et qui est dû à G. Castelnuovo consiste à considérer une telle application de Jonquières $\j_1$. Puis elle considère $f\circ \j_1^{-1}$et réitère le procédé jusqu'à obtenir un automorphisme.

\subsubsection{Support d'une transformation de Jonquières}
Dans cette partie, nous déterminons dans un cas très spécial, celui où il existe au moins un point adhérent à un autre point, à quelles conditions un ensemble pré-consistant de points de $\B(\P^2)$ est le support d'une application de Jonquières. La preuve proposée ici est plus simple que celle présentée dans \cite{AC}, mais elle ne traite que d'un cas particulier. Le problème dans le cas général est de faire attention au fait que les points-base ne doivent pas se trouver sur une courbe de sorte qu'en éclatant les points et en contractant les fibres, la courbe devienne d'auto-intersection strictement plus petite que $-1$. Par exemple, trois points alignés ne forment pas le support d'une transformation quadratique. 

\begin{lemme}\label{lemme_nbr_pt_adh_majeur}
Soient $f\in\Bir(\P^2)$ et $p$ un de ses points-base de multiplicité maximale. Parmi les points-base majeurs de $f$, le nombre de points-base adhérents à $p$ est inférieur ou égal au nombre de points-base non adhérents à $p$. 
\end{lemme}

\begin{proof}
Nous reprenons lors des deux premiers points des calculs connus et faits par exemple dans la Section 8.2 de \cite[Lemma 8.2.3 et 8.2.6]{AC}. Notons $r$ le nombre de points-base de $f$ et $h$ le nombre de points-base majeurs de $f$. 
\begin{itemize}[wide]
\item D'après le théorème de Bézout (Proposition \ref{prop_Bezout}), nous avons $d\geq m_i+m_0$ pour tout $1\leq i\leq r-1$, ce qui implique : 
\[ 2\car_{f}\geq m_i.\]
\item Multiplions par $\car_{f}$ l'équation \ref{eq_Noethercanonique} du lemme \ref{lemme_equations_evidentes} et soustrayons cela à l'équation \ref{eq_dcarre_micarre} du même lemme :
\begin{align*}
\sum\limits_{i=0}^{r-1}m_i(m_i-\car_{f})&=d^2-1-3d\car_{f}+3\car_{f}\\
&=d(d-3\car_{f})+3\car_{f}-1\\
&=(m_0+2\car_{f})(m_0-\car_{f})+3\car_{f}-1,
\end{align*}

ce qui implique : 
\[\sum\limits_{i=1}^{r-1}m_i(m_i-\car_{f})=2\car_{f}(m_0-\car_{f})+3\car_{f}-1.\]
Comme pour $i>h$ les points ne sont pas majeurs nous avons 
\[\sum\limits_{i=1}^{h}m_i(m_i-\car_{f})>2\car_{f}(m_0-\car_{f}),\]
ce qui donne finalement par le premier point \[\sum\limits_{i=1}^{h}(m_i-\car_{f})>m_0-\car_{f}.\]
\item Notons $h_a$ le nombre de points-base de $f$ majeurs et adhérents à $p_0$. Quitte à réordonner les indices, l'équation précédente se ré-écrit :
\[\sum\limits_{i=1}^{h_a}m_i+ \sum\limits_{i=h_a+1}^{h}m_i>m_0+(h-1)\car_{f}.\]
Par positivité des excès en $p_0$ pour $f$ (Proposition \ref{prop_consistence}) et par le premier point de la preuve : 
\[m_0+(h-h_a)2\car_{f}>m_0+(h-1)\car_{f}\]
ce qui entraîne :\[2(h-h_a)\geq h.\] Par conséquent $\frac{h}{2}\geq h_a$ ce qui signifie qu'il y a au plus autant de points-base majeurs adhérents à $p_0$ que de points-base majeurs non adhérents.\qedhere \end{itemize}
\end{proof}

\begin{lemme}\label{lemme_support_Jonquieres}
Considérons un ensemble pré-consistant de points $\{p_i\}_{0\leq i\leq 2\delta}$ de $\B(\P^2)$ satisfaisant \begin{enumerate}[wide, label=\alph*)]
\item\label{condition_1} exactement $\delta$ points sont adhérents à $p_0$,
\item\label{condition_2} pour toute paire $(i,j)$ où $1\leq i<j\leq 2\delta$, les points $p_i$, $p_j$ et $p_0$ ne sont pas alignés,
\item\label{condition_3} si deux points de cet ensemble sont adhérents à un même troisième point $p_i$ de cet ensemble alors $i=0$.
\end{enumerate}
Alors il existe une application de Jonquières de degré $\delta+1$ qui possède cet ensemble comme points-base et dont $p_0$ est le point-base majeur.
\end{lemme}
\begin{proof}
Lorsque nous éclatons le point $p_0$, il n'y a pas deux points sur une même fibre d'après la condition \ref{condition_2}. Nous éclatons ensuite les $\delta$ points adhérents, et nous contractons les transformées strictes des fibres correspondantes. Par la condition \ref{condition_3}, il n'y a toujours pas deux points sur une même fibre. La surface obtenue est une surface de Hirzebruch $\F_{\delta+1}$ de section exceptionnelle d'auto-intersection $-(\delta+1)$. Ensuite, nous éclatons un point non adhérent puis contractons la fibre passant par ce point. En faisant cela, nous obtenons la surface de Hirzeburch dont la section exceptionnelle est d'auto-intersection un de plus par rapport à la surface précédente et il n'y a toujours pas deux points sur une même fibre. Ainsi en répétant cette opération pour les $\delta$ points non adhérents nous obtenons une surface de Hirzeburch $\F_1$ et en contractant la section exceptionnelle nous obtenons une application de Jonquières de degré $\delta+1$.
\end{proof}

\begin{prop}\label{prop_supp_jonq_points_majeurs}
Soit $f\in\Bir(\P^2)$ possédant un point-base majeur adhérent à son (ou à un des ses) point-base maximal $p_0$. 
Pour tout sous-ensemble de l'ensemble des points-base majeurs de $f$, pré-consistant, de cardinal pair $2\delta$ et possédant un nombre de points adhérents à $p_0$ égal à $\delta$, il existe une application de Jonquières de degré $\delta+1$ ayant $p_0$ comme point-base maximal et cet ensemble de points comme autres points-base. De plus, un tel sous-ensemble existe toujours.
\end{prop}
\begin{proof}
Nous cherchons à utiliser le lemme \ref{lemme_support_Jonquieres}. Pour cela, plusieurs points sont à vérifier : 
\begin{itemize}[wide]
\item L'hypothèse sur $f$ et le lemme \ref{lemme_nbr_pt_adh_majeur} justifient le fait qu'il existe un sous-ensemble de l'ensemble des points-base majeurs de $f$ qui possède autant de points adhérents à $p_0$ que de points non adhérents.
\item D'après le théorème de Bézout \ref{prop_Bezout}, il n'existe pas deux points majeurs alignés avec le point $p_0$. 
\item Il n'y a pas deux points majeurs adhérents à un point majeur. Sinon, par la proposition \ref{prop_consistence}, il existerait un point majeur $p_1$ de multiplicité $m_1>d-m_0$. Mais dans ce cas, il y aurait un point $q$ libre et adhérent à $p_0$ (si $p_1$ est voisin de $p_0$) ou dans $\P^2$ qui serait de multiplicité strictement supérieure à $d-m_0$. Ceci contredit le théorème de Bézout (proposition \ref{prop_Bezout}) en considérant la droite passant par les points $p_0$ et $q$.
\end{itemize}
Ainsi, d'après le lemme \ref{lemme_support_Jonquieres}, il existe une application de Jonquières de degré $\delta+1$ ayant $p_0$ comme point-base de multiplicité maximal et possédant cet ensemble comme petits points-base.
\end{proof}

\begin{rmq}
Sans la condition qu'il y ait autant de points adhérents à $p_0$ que de points non adhérents ce n'est pas possible de conclure. En effet, considérons les points $\{p_0,\dots,p_6\}$ de $\P^2$ tels qu'il n'y ait pas deux points alignés avec $p_0$, que les points $p_1,\dots,p_4$ soient alignés et que les points $p_5$ et $p_6$ n'appartiennent pas à la droite contenant les points $p_1,\dots,p_4$.
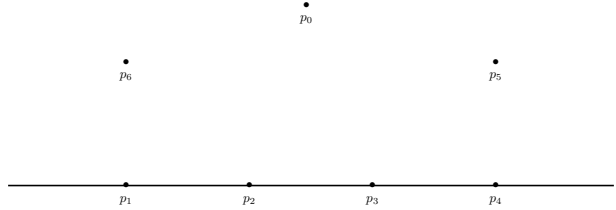
\begin{center}
\scalebox{0.5}{
\begin{tikzpicture}[node distance=3cm, font=\normalsize, very thick]
\node (p3)  [label=below:{$p_1$}] {};

\node (p4)  [right=of p3,
			 label=below:{$p_2$}] {};
			 
\node (p5)  [right=of p4,
			 label=below:{$p_3$}] {};			 
			 
\node (p6)  [right=of p5,
			 label=below:{$p_4$}] {};
			 
\node (p2)  [above=of p3,
			 label=below:{$p_6$}] {};

\node (p1)  [above=of p6,
			 label=below:{$p_5$}] {};
			 
\node (p7)  [above right=of p3] {};	

\node (p0)  [above right=of p7,
			  label=below:{$p_0$}] {};
\node (p8)  [left=of p3] {};
\node (p9)  [right=of p6] {};
			 
\draw (p2) node {$\bullet$};
\draw (p3) node {$\bullet$};
\draw (p4) node {$\bullet$};
\draw (p5) node {$\bullet$};
\draw (p6) node {$\bullet$};
\draw (p1) node {$\bullet$};
\draw (p0) node {$\bullet$};
\draw (p8)--(p9);
\end{tikzpicture} }
\end{center} 
Ils constituent le support d'une application de Jonquières de degré $4$ et de point-base maximal $p_0$. Mais si nous considérons seulement les points $p_0,p_1,p_2,p_3,p_4$ ils ne forment plus le support d'une application de Jonquières. En effet, si c'était le cas, ce serait une transformation de Jonquières de degré $3$. Ainsi la pré-image d'une droite générale serait une cubique et elle aurait $4$ points d'intersection avec la droite passant par les points $p_1$, $p_2$, $p_3$, $p_4$, ce qui contredirait le théorème de Bézout. L'obstruction géométrique est que si nous éclatons les points $p_0$ et $p_1$ et que nous contractons la droite passant pas ces deux points, nous nous retrouvons sur $\P^1\times \P^1$. En éclatant ensuite le point $p_2$, nous nous retrouvons sur $\F_1$ mais les deux points restants sont sur la section exceptionnelle. Ainsi, il est impossible de revenir sur $\P^2$ en éclatant les deux points restants et en contractant des courbes.
\end{rmq}

 \section{Construction d'un pavage de Voronoï}\label{section_def_orbl}
Considérons l'action du groupe de Cremona sur $\H$. L'objet de cette section est de construire un pavage de Voronoï associé à cette action. Dans la sous-section \ref{sous_section_restriction}, nous nous restreignons à un sous-espace convexe $\orbl$ de $\H$ contenant l'enveloppe convexe de l'orbite de la droite $\l$. Nous ne nous restreignons pas seulement à l'enveloppe convexe de l'orbite de $\l$ car il est parfois difficile de vérifier qu'un élément y appartient. Pour cette raison, nous élargissons l'espace à étudier. Nous construisons alors un pavage de Voronoï de la façon suivante. 
  Nous considérons l'orbite de la classe de la droite $\l$. Cela nous donne un ensemble discret de points de $\orbl$.  
  Cependant le stabilisateur de $\l$ est $\PGL(3,\kk)$. Nous identifions donc deux applications qui diffèrent par un automorphisme : 
  \[f\sim g \Leftrightarrow \text{ il existe }a\in\PGL(3,\kk), \ f=g\circ a.\]
  Nous notons $\overline{f}$ une telle classe.
  Par construction, toute application d'une même classe d'équivalence agit de la même façon sur $\l$. Pour toute classe $\overline{f}$ où $f\in\Bir(\P^2)$, nous associons une cellule de Voronoï notée $\V(f)$ et définie comme : 
  \[\V(f)=\{c\in\orbl\mid \dist(c,\isome{f}(\l)) \leq \dist(c,\isome{g}(\l)) \text{ pour tout }g\in\Bir(\P^2)\}.\]
  Remarquons que toutes les applications du groupe de Cremona appartenant à la même classe d'équivalence indexent la même cellule de Voronoï. 
  Les applications $f'\in\overline{f}$ sont appelées les \emph{germes} associés à $\V(f)$ et les classes $\isome{f}(\l)$ sont les \emph{centres} des cellules de Voronoï.
  Une cellule de Voronoï ne correspond pas tout-à-fait à un domaine fondamental mais c'est l'orbite d'un domaine fondamental sous l'action de $\PGL(3,\kk)$ puisque $\PGL(3,\kk)$ agit non-trivialement sur la cellule $\V(\id)$. 
  
  Nous définissons dans une première sous-section l'espace sur lequel nous faisons agir le groupe de Cremona, puis nous montrons que les cellules de Voronoï recouvrent l'espace construit et enfin nous montrons que les cellules de Voronoï ne s'accumulent pas sur une cellule.

\subsection{Restriction à un sous-espace} \label{sous_section_restriction}
Considérons l'espace hyperbolique $\H$ de dimension infinie obtenu comme hyperboloïde dans l'espace de Picard-Manin qui a été construit à la sous-section \ref{sous-section_hyperobloide_PM}. Nous considérons le sous-espace suivant.
  
\begin{defi}\label{proprietes_c}
   	L'ensemble $\orbl$ est le sous-espace de $\H$ constitué des classes \[c=n\ell-\sum\limits_{p\in\B(\P^2)}\lambda_pe_{p} \ \ \text{ (}n\text{ réel}\geq 1\text{)}\] satisfaisant : 
   	\begin{enumerate}
   		\item\label{propriete_coeff_classe} $\lambda_p\geq 0$ pour tout $p\in \B(\P^2)$,
   		\item\label{propriete_classe_canonique} la \emph{positivité contre la classe anti-canonique} : \[ 3n-\sum\limits_{p\in \B(\P^2)}\lambda_p\geq 0,\]
   		\item\label{propriete_classe_exces} la positivité des excès de tout point $p\in\B(\P^2)$ : \[\lambda_{p} - \sum\limits_{\substack{q\in \B(\P^2)\\ q \to p }}\lambda_{q}\geq0,\]
   		\item\label{propriete_classe_bezout} la \emph{condition de Bézout} : pour toute courbe de $\P^2$ de degré $d$ passant avec multiplicité $\mu_p$ en chaque point $p\in\B(\P^2)$ : \[nd-\sum\limits_{p\in \B(\P^2)}\lambda_p\mu_p\geq0.\]
   	\end{enumerate}
\end{defi}
L'inégalité de la propriété \ref{propriete_classe_canonique} implique que les classes de $\orbl$ sont en fait $L^1$. Par conséquent, la forme canonique est bien définie sur les classes de $\orbl$ et la propriété \ref{propriete_classe_canonique} revient à demander que les classes soient positives contre la forme anti-canonique qui est l'opposé de la forme canonique : \[-\fcan \cdot c\geq 0.\] Lorsqu'une classe satisfait cette inégalité, nous disons qu'elle est \emph{positive contre la classe anti-canonique}.
Le coefficient $n$ est appelé le \emph{degré} de $c$ et les coefficients $\lambda_p$ les \emph{multiplicités} de $c$ aux points $p\in\B(\P^2)$. Les points pour lesquels la multiplicité associée est strictement positive constituent le \emph{support} de $c$. Cet ensemble est noté $\supp(c)$. Remarquons que par définition de $\H$ celui-ci est dénombrable. Grâce à la propriété \ref{propriete_classe_exces}, le support de $c$ est un ensemble pré-consistant de points, au sens de la définition introduite dans la section \ref{sous-section_voc}. 
\begin{rmq}\label{propriete_classe_carre}
   	Par définition, une classe $c\in\orbl$ appartient à $\H$ et possède donc un degré supérieur ou égal à $1$ et strictement supérieur à chacune de ses multiplicités  : \[1=c^2=n^2-\sum\limits_{p\in\B(\P^2)}\lambda_p^2\leq n^2-\lambda_q^2 \text{ pour tout } q\in\B(\P^2).\]
\end{rmq}
Notons $\PM_{>0}$ le sous-espace de l'espace de Picard-Manin constitué des classes d'auto-intersection strictement positive. Une classe de $\PM_{>0}$ satisfaisant les points \ref{propriete_coeff_classe}, \ref{propriete_classe_canonique}, \ref{propriete_classe_exces} et \ref{propriete_classe_bezout} de la définition \ref{proprietes_c} est dite \emph{proportionnelle} à une classe de $\orbl$. Une telle classe appartient à la demi-droite ouverte issue de la classe nulle et passant par une classe de $\orbl$. Plus généralement, étant donné un sous-ensemble $\mathcal{F}$ de l'espace de Picard-Manin, une classe d'auto-intersection strictement positive est dite proportionnelle à une classe de $\mathcal{F}$ si elle appartient à la demi-droite ouverte issue de la classe nulle et passant par une classe de $\mathcal{F}$. 
 L'application de normalisation :
 \begin{align}\label{application_normalisation}
\n : \PM_{>0}&\to\H \\
x &\mapsto \frac{x}{\sqrt{\mathcal{B}(x,x)}}\notag.
\end{align}
envoie toute classe de $\PM_{>0}$ sur sa classe proportionnelle dans $\H$. 

\begin{prop}\label{action_cremona_orbl}
L'action du groupe de Cremona stabilise $\orbl$.
\end{prop}
\begin{proof}
	Le groupe de Cremona étant engendré par $\PGL(3,\kk)$ et l'involution standard $\sigma$ (Théorème \ref{thm_générateurs_Cremona}) il suffit de vérifier que l'action de ces applications préserve $\orbl$. 
	C'est le cas pour les automorphismes. Soit \[c=n\l-\sum\limits_{p\in\B(\P^2)}\lambda_pe_{p}\in\orbl.\] Montrons que $\isome{\sigma}(c)$ appartient à $\orbl$.	
	Notons $p_0$, $p_1$ et $p_2$ les points-base de $\sigma$ (qui sont en fait $[0:0:1]$, $[0:1:0]$ et $[1:0:0]$) et $q_0$, $q_1$ et $q_2$ ceux de $\sigma^{-1}$. Bien que $\sigma$ soit une involution nous les différencions pour plus de clarté. 
	D'après l'équation \eqref{eq_action_f_sur_c}, nous avons :
	\begin{multline*}
	\isome{\sigma}(c)=(2n-\lambda_{p_0}-\lambda_{p_1}-\lambda_{p_2})\l-(n-\lambda_{p_1}-\lambda_{p_2})e_{q_0} -(n-\lambda_{p_0}-\lambda_{p_2})e_{q_1}\\-(n-\lambda_{p_0}-\lambda_{p_1})e_{q_2}-\sum\limits_{\substack{p\in\B(\P^2)\\ p\notin\{p_0,p_1,p_2\}}}\lambda_p\isome{\sigma}(e_{p}).
	\end{multline*} 
	Le groupe de Cremona agit sur $\H$ (voir le paragraphe \ref{sous-section_hyperobloide_PM}) par conséquent la classe $\isome{\sigma}(c)$ appartient à $\H$ et possède donc un degré strictement positif. Intéressons-nous aux autres points de la définition. 
	\begin{enumerate}[wide]
		\item Comme $c\in\orbl$ les multiplicités $\lambda_p$ sont positives ou nulles. Montrons que c'est le cas des multiplicités de $\isome{\sigma}(c)$ pour les points $q_0$, $q_1$ et $q_2$. Les points $p_0$ et $p_1$ sont dans $\P^2$ il existe donc une droite $L$ passant par ces deux points et comme $c$ satisfait la condition de Bézout (Définition \ref{proprietes_c}.\ref{propriete_classe_bezout}), nous avons :
		\[n-\lambda_{p_0}-\lambda_{p_1}\geq 0.\] Les deux autres multiplicités sont positives avec le même argument.
		\item La condition de positivité contre l'anti-canonique est satisfaite puisque la forme canonique est constante sur l'orbite de $c$ sous l'action du groupe de Cremona (Remarque \ref{rmq_canonique_stable_orbite}).

		\item Montrons que les excès de $\isome{\sigma}(c)$ sont positifs en tout point $q\in\B(\P^2)$. Soit $q\in\B(\P^2)$. Deux cas se présentent selon si $q$ est un point-base de $\sigma^{-1}$ ou pas. 
		
		Dans le cas où $q$ n'est pas un point-base de $\sigma^{-1}$, il existe $p\in\B(\P^2)$ tel que $e_p=\isome{\sigma}^{-1}(e_q)$. Quitte à éclater au départ et à l'arrivée, $\sigma^{-1}$ induit un isomorphisme local envoyant $q$ sur $p$ et par conséquent les points adhérents à $q$ sont envoyés sur des points adhérents à $p$ et donc la positivité des excès est préservée. 
		
		Il nous reste à considérer les excès au dessus des points $q_0$, $q_1$ et $q_2$. Notons $S$ la surface obtenue en éclatant les points $p_0$, $p_1$ et $p_2$. Les points adhérents à $q_0$ sont les points appartenant à la transformée stricte de la droite $L$ passant par les points $p_1$ et $p_2$ sur $S$. Ainsi la droite $L$ passe par les points $p_1$ et $p_2$ ainsi que par les points adhérents à $q_0$. Par conséquent en considérant cette droite et le fait que $c$ satisfait la condition de Bézout (Définition \ref{proprietes_c}.\ref{propriete_classe_bezout}), nous avons : \[(n-\lambda_1-\lambda_2)-\sum\limits_{p\mapsto q_0}\lambda_p=n-(\lambda_1+\lambda_2+\sum\limits_{p\mapsto q_0}\lambda_p)\underset{\ref{proprietes_c}.\ref{propriete_classe_bezout}}{\geq} 0.\] Ainsi, l'excès en $q_0$ est positif. Il en est de même pour les points $q_1$ et $q_2$. 
		\item Montrons que $\isome{\sigma}(c)$ satisfait la condition de Bézout pour toute courbe de $\P^2$. 
		Considérons dans un premier temps une courbe $C_1$ qui n'est contractée par $\sigma^{-1}$. Notons $c_1$ sa classe dans Picard-Manin. Dans ce cas, la classe $\isome{\sigma}^{-1}(c_1)$ est la classe de Picard-Manin correspondant à la courbe de $\P^2$ obtenue comme adhérence de $\sigma^{-1}(C_1\setminus\{q_0,q_1,q_2\})$. Par conséquent son nombre d'intersection avec $c$ est positif ou nul puisque $c$ satisfait la condition de Bézout :
		\[\isome{\sigma}(c)\cdot c_1 =c\cdot \isome{\sigma}^{-1}(c_1)\underset{\ref{proprietes_c}.\ref{propriete_classe_bezout}}{\geq} 0.\]
		Considérons à présent les droites contractées par $\sigma^{-1}$ comme par exemple la droite passant par les points $q_1$ et $q_2$. Alors nous obtenons :
		\[2n-\lambda_{p_0}-\lambda_{p_1}-\lambda_{p_2}-(n-\lambda_{p_0}-\lambda_{p_2}) -(n-\lambda_{p_0}-\lambda_{p_1})=\lambda_{p_0}\geq 0.\]
		Il en est de même des deux autres droites contractées.

		Par conséquent $\isome{\sigma}(c)$ satisfait également la condition de Bézout (Définition \ref{proprietes_c}.\ref{propriete_classe_bezout}).
	\end{enumerate}
	Nous avons ainsi montré que l'action de $\sigma$ préserve $\orbl$ et par conséquent que l'action du groupe de Cremona stabilise $\orbl$.
\end{proof}

\begin{prop}\label{prop_orbl_convexe}
   	L'espace $\orbl$ est un sous-espace convexe fermé de $\H$. 
\end{prop}
\begin{proof}
Les conditions \ref{propriete_coeff_classe} et \ref{propriete_classe_bezout} définissent des demi-espaces fermés de l'espace de Picard-Manin. 
Comme les multiplicités des classes de $\orbl$ sont positives d'après la condition \ref{propriete_coeff_classe}, la condition \ref{propriete_classe_canonique} se réécrit de la façon suivante :  
\[\text{ pour toute partie finie } F \text{ de } \B(\P^2),\ \ \sum\limits_{p\in F}\lambda_p\leq3n.\]
Par conséquent la condition \ref{propriete_classe_canonique} correspond à une intersection infinie de demi-espaces fermés qui est un espace fermé et convexe. Nous raisonnons pareil pour la condition \ref{propriete_classe_exces}. 
En intersectant ce fermé convexe avec l'hyperboloïde $\H$ nous obtenons que $\orbl$ est fermé et convexe comme annoncé. 
\end{proof}

\begin{prop}
   	L'enveloppe convexe dans $\H$ de l'orbite de $\l$ sous l'action du groupe de Cremona est incluse dans $\orbl$.
\end{prop}
\begin{proof}
Comme par la proposition \ref{prop_orbl_convexe} l'espace $\orbl$ est convexe, il suffit de montrer que toutes les classes de l'orbite de $\l$ sous l'action du groupe de Cremona sont dans $\orbl$. De plus comme d'après la proposition \ref{action_cremona_orbl} l'action du groupe de Cremona préserve $\orbl$, il suffit de vérifier que la classe de la droite $\l$ appartient à $\orbl$. 
La classe $\l$ appartient à $\H$ et vérifie le point \ref{propriete_coeff_classe}. La classe de la droite contre la classe anti-canonique vaut $3$ ainsi la positivité contre l'anti-canonique est vérifiée (point \ref{propriete_classe_canonique}).  
Les points \ref{propriete_classe_exces} et \ref{propriete_classe_bezout} sont également vérifiés.
\end{proof}
\begin{rmq} Soit $p\in\P^2$.
  Il existe $t_0\in[0,1[$ tel que pour tout $t_0\leq t <1$ la classe $\frac{1}{\sqrt{1-t^2}}(\l-te_p)$ n'est pas dans l'enveloppe convexe de l'orbite de $\l$ puisque l'intersection de cette classe contre la classe anti-canonique est strictement supérieure à $3$. 
\end{rmq}

\subsection{Pavage}
Montrons que les cellules de Voronoï recouvrent l'espace $\orbl$ (Corollaire \ref{cor_pavage}).

\begin{prop}\label{prop_inf_atteint}Pour toute classe $c$ appartenant à $\orbl$, les deux infimums suivants sont atteints :
\begin{enumerate}
\item\label{item1_pavage} $\inf\big\{\dist(c,\isome{f}(\l))\mid f\in\Bir(\P^2)\big\}$
\item\label{item2_nonacc} $\inf\big\{\dist(c,\isome{f}(\l))\mid f\in\Bir(\P^2) \text{ et }c\notin \V(f)\big\}$.
\end{enumerate}
\end{prop}
\begin{proof}
	Soit $c=n\l-\sum\limits_{p\in\B(\P^2)}\lambda_pe_p$ une classe de $\orbl$. Montrons que l'infimum suivant est atteint : 
	\begin{align}\label{eq_inf}
	\inf\big\{\dist(c,\isome{f}(\l))\mid f\in\Bir(\P^2)\big\}.
	\end{align}
	
	Considérons les classes de l'orbite de la droite $\l$ se trouvant dans la boule fermée de $\orbl$ centrée en $c$ et de rayon $\argcosh (n)$. La classe $\l$ est dans cette boule. De plus, par l'inégalité triangulaire, toute application $f$ telle que la classe $\isome{f}(\l)$ se situe dans cette boule est de degré au plus $\cosh(2\argcosh(n))=2n^2-1$. Notons $B$ le sous-ensemble du groupe de Cremona constitué des applications de degré au plus $2n^2-1$.
	Il n'y a qu'un nombre fini de caractéristiques pour des applications du groupe de Cremona de degré au plus $2n^2-1$. Remarquons que si le support de $c$ est fini alors l'infimum est atteint. Nous considérons donc le cas où $c$ est de support infini. Notons $C_1,\dots,C_k$ les caractéristiques des éléments de $B$. Pour $1\leq j\leq k$, posons $C_j=(d^j;m_0^j,\dots,m_{r_j}^j)$. Notons $B_j$ l'ensemble des applications de $B$ ayant $C_j$ pour caractéristique.
	Montrons que pour tout $1\leq j\leq k$
	\begin{align}\label{eq_inf_j}
	\inf\big\{\dist(c,\isome{f}(\l))\mid f\in B_j\big\}
	\end{align}
	est atteint. Rappelons que pour tout $1\leq j\leq k$ et pour toute application $f\in B_j$, $\dist(c,\isome{f}(\l))\leq \argcosh (n)$. Fixons $1\leq j\leq k$.
	Raisonnons par l’absurde et supposons qu'il existe une suite $(f_i)_{i\in\N}$ d'applications de $B_j$ telle que la suite $(\dist(c,\isome{{f_i}}(\l)))_{i\in \N}$ soit strictement décroissante. Pour chaque $i$, nous regardons le point-base $q$ de $f_i^{-1}$ de multiplicité $m_0^j$. Nous notons $\lambda_{0,i}$ la multiplicité $\lambda_q$ de $c$ associée au point $q$. À noter que si $q$ n'est pas un point du support de $c$, la multiplicité $\lambda_{0,i}$ est nulle. Nous faisons de même pour chaque $0\leq s \leq r_j$. Nous obtenons ainsi $r_j+1$ suites $\{\lambda_{s,i}\}_{i\in\N}$ de coefficients. Si nous l'ordonnons, la suite des multiplicités de $c$ est décroissante et tend vers $0$. Par conséquent quitte à prendre des sous-suites nous pouvons supposer que chacune des $r_j+1$ suites est ou bien constante ou bien strictement décroissante. 
	Or la suite $(\dist(c,\isome{{f_i}}(\l)))_{i\in \N}$ étant strictement décroissante cela implique que la suite $S_i=\sum\limits_{s=0}^{r_j}m_s\lambda_{s,i}$ est strictement croissante ce qui est absurde.
	Par conséquent, pour tout $1\leq j\leq k$ il existe une application $f^j\in B_j$ pour laquelle l'infimum \eqref{eq_inf_j} est atteint :	\[\dist(c,\isome{f}^j(\l))=\inf\big\{\dist(c,\isome{f}(\l))\mid f\in B_j\big\}.\]
	Notons $j_0\in\{1,\dots,k\}$ un indice tel que \[\dist(c,\isome{f}^{j_0}(\l))= \underset{1\leq j\leq k}{\min}\dist(c,\isome{f}^j(\l)).\]
	L'infimum \eqref{eq_inf} est atteint pour $f^{j_0}$.

	Le point \ref{item2_nonacc} se démontre de la même façon.
	\end{proof}
	\begin{cor}\label{cor_pavage}
	Les cellules de Voronoï pavent l'espace $\orbl$.
	\end{cor}
	\begin{proof}
	Soit $c\in\orbl$. Le point \ref{item1_pavage} de la proposition \ref{prop_inf_atteint} implique qu'il existe $f\in\Bir(\P^2)$ telle que $c$ appartient à $\V(f)$.
	\end{proof}

\subsection{Non-accumulation des cellules}
Dans cette partie, nous montrons que les cellules de Voronoï ne s'accumulent pas sur cellule : pour tout $f\in\Bir(\P^2)$, pour toute classe $c\in\V(f)$, il existe $\epsilon>0$ tel que les cellules de Voronoï ne contenant pas $c$ mais tel que $c$ soit dans leur $\epsilon$-voisinage sont en nombre fini. 
 La condition \og ne contenant pas $c$ \fg{} est nécessaire car nous verrons que dans certains cas $c$ appartient à un nombre infini de cellules. 

\begin{prop}\label{prop_discret}
Pour tout $c\in\orbl$, il existe $\epsilon>0$ tel que pour tout $f\in\Bir(\P^2)$ soit $c$ appartient à $\V(f)$, soit la distance entre $c$ et $\V(f)$ est strictement supérieure à $\epsilon$ : \[\dist(c,\V(f))> \epsilon.\]
\end{prop}
La notation \og$\dist(c,\V(f))$ \fg\ est un abus de notation et signifie l'infimum sur toutes les classes $c'$ de $\V(f)$ des distances entre $c$ et $c'$ : \[\dist(c,\V(f))=\inf\{\dist(c,c')\mid c'\in\V(f)\}.\]
\begin{proof}
Quitte à faire agir le groupe de Cremona sur le pavage de Voronoï, nous pouvons supposer que $c$ appartient à la cellule $\V(\id)$.

Si la classe $c$ est à l'intérieur (pour le pavage de Voronoï) de $\V(\id)$ c'est-à-dire si $c$ n'appartient à aucune autre cellule de Voronoï, tout $\epsilon$ strictement positif tel que la boule fermée centrée en $c$ et de rayon $\epsilon$ est incluse dans l'intérieur (pour le pavage) de $\V(\id)$ convient. Intéressons-nous à présent aux classes qui sont au bord de la cellule $\V(\id)$, c'est-à-dire aux classes qui sont à l'intersection de plusieurs cellules de Voronoï dont la cellule $\V(\id)$.
Soit $c=n\l-\sum\limits_{p\in\B(\P^2)}\lambda_pe_p\in\orbl$ une classe appartenant au bord de la cellule $\V(\id)$. D'après le point \ref{item2_nonacc} de la proposition \ref{prop_inf_atteint}, l'infimum suivant est atteint \[D:=\inf\{\dist(c,\isome{f}(\l))\mid f\in\Bir(\P^2) \text{ et }c\notin \V(f)\}.\]  Posons \[\epsilon=\frac{D-\argcosh(n)}{6}.\]  
Comme $c$ appartient à $\V(\id)$, $\epsilon$ est strictement positif.
Montrons que pour toute application $f\in\Bir(\P^2)$ telle que $c$ n'appartient pas à $\V(f)$, nous avons :
\[\dist(c,\V(f))>\epsilon.\]
Raisonnons par l'absurde et supposons qu'il existe $f\in\Bir(\P^2)$ telle que $c$ n'appartient pas à $\V(f)$ et qu'il existe $c'\in\V(f)$ tel que $\dist(c,c')\leq \epsilon$. 
Nous avons alors par définition de $D$ et par l'inégalité triangulaire la contradiction suivante : 
\begin{align*}
D\leq \dist(c,\isome{f}(\l))\leq \dist(c,c')+\dist(c',\isome{f}(\l))&\leq \epsilon +\dist(c',\l)\\ 
&\leq \epsilon +\dist(c',c)+\dist(c,\l)\\
&\leq 2\epsilon+\argcosh(n)\\
&= \frac{D}{3}+\frac{2\argcosh(n)}{3}\\
&< \frac{D}{3}+\frac{2D}{3}=D.\qedhere
\end{align*}
\end{proof}
Nous disons qu’un segment de $\orbl$ \emph{traverse} une cellule de Voronoï s’il existe un
sous-segment contenu dans l'intérieur (pour le pavage) de cette cellule. Le corollaire suivant dit qu'un segment ne traverse qu'un nombre fini de cellules de Voronoï. Plus généralement, il donne un analogue de ce résultat lorsque le segment reste dans le bord de Voronoï de plusieurs cellules.
\begin{cor}\label{cor_nbre_fini_cell_traversees}
Soit $[c,c']$ un segment géodésique de $\orbl$. Toute suite $(c_i)_{i\in I}$ de classes du segment $[c,c']$ satisfaisant les trois points suivants est finie : 
\begin{itemize}
\item $c_0=c$,
\item $c_i<c_{i+1}$ et
\item pour tout $i\geq 1$ il existe $f_i\in\Bir(\P^2)$ satisfaisant $c_i\in \V(f_i)$ et $c_{i-1}\notin \V(f_i)$.
\end{itemize}
En particulier, un segment géodésique ne traverse qu'un nombre fini de cellules de Voronoï.
\end{cor}

\begin{proof}
Paramétrons le segment $[c, c']$ par $\gamma(t)$ pour $t \in [0, 1]$ avec $\gamma(0) = c$ et $\gamma(1) = c'$.
  	Nous construisons une suite de points $c_i=(\gamma(t_i))$ par le procédé de récurrence suivant.
  	Initialisons en posant $t_0 = 0$ et $c = \gamma(0)$.
  	Pour $i \ge 1$, si $t_{i-1} \in [0,1[$ est déjà construit, nous posons 
  	\[ t_i = \inf \{t \in ]t_{i-1},1]\mid \exists f \in \Bir(\P^2), \gamma(t) \in \V(f) \text{ et } \gamma(t_{i-1})\notin \V(f)\}.\]
  	Remarquons que par construction, les classes de cette suite, excepté peut-être le premier et le dernier terme, appartiennent à plusieurs cellules (au moins à deux : une contenant la classe précédente dans la suite et une ne la contenant pas). De plus, pour tout $1\leq i$ et pour toute application $f$ telle que $\V(f)$ contient la classe $c_i$ mais pas la classe $c_{i-1}$, la cellule $\V(f)$ ne contient pas les classes $c_j$ pour $j\leq i-1$ par convexité des cellules de Voronoï.
  	
  	Pour tout $0\leq t\leq 1$, d'après la proposition \ref{prop_discret} et le paragraphe précédent, il existe $\epsilon_t>0$ tel que dans l'intervalle ouvert centré en $\gamma(t)$ et de diamètre $2\epsilon$ il existe au plus une classe de la suite $\{c_i\}_{i\in I}$. Ces ouverts forment un recouvrement du segment géodésique compact $[c,c']$, par conséquent il existe un sous-recouvrement fini. Or chacun de ces ouverts contient au plus un point de la suite. Il y a donc un nombre fini d'éléments dans cette suite.
\end{proof}

\subsection{Bord à l'infini}
Le bord de $\H$, noté $\partial_{\infty}\H$, est constitué des classes d'auto-intersection nulle. Ces classes peuvent être vues comme limites de classes de Picard-Manin vivant dans l'hyperplan affine d'équation $\{\l=1\}$ et proportionnelles à des classes de $\H$.
À noter que le nombre d'intersection entre $c'\in\partial_{\infty}\H$ et $c\in\H$ n'est bien défini que si $c$ est de degré $1$.
\begin{rmq}\label{rmq_limite_intersection_avec_classe_infini}
Si $(c_i)_{i\in\N}$ est une suite de classes de $\H$ convergeant vers une classe $c\in\partial_{\infty}\H$ de degré $1$ alors \[\lim\limits_{i\rightarrow \infty}c\cdot c_i =0.\]
En effet, pour tout $i\in\N$, notons $d_i$ le degré de $c_i$. Ainsi la suite $(d_i)_{i\in\N}$ tend vers l'infini puisque c'est le cas de la distance entre $c_i$ et $\l$. Par conséquent, nous avons : \[\lim\limits_{i\rightarrow \infty}c\cdot c_i =\lim\limits_{i\rightarrow \infty}\frac{1}{d_i}c_i\cdot c_i=\lim\limits_{i\rightarrow\infty}\frac{1}{d_i}=0.\]
\end{rmq}
 Nous définissons \emph{le bord à l'infini} de $\orbl$, noté $\partial_{\infty}\orbl$ comme l'ensemble des classes du bord de $\H$ qui vérifient les points \ref{propriete_coeff_classe} à \ref{propriete_classe_bezout} de la définition \ref{proprietes_c} de $\orbl$. De même que précédemment ces classes sont limites de classes de Picard-Manin vivant dans l'hyperplan affine d'équation $\{\l=1\}$ et proportionnelles à des classes de $\orbl$.

\subsection{Classes symétriques}

Une classe dont les $r$ plus grandes multiplicités sont égales et la $(r+1)$-ème est strictement plus petite est appelée \emph{$r$-symétrique}. Une classe $r$-symétrique est dite \emph{pure} si les $r$ plus grandes multiplicités sont égales et si toutes les autres sont nulles.
\begin{lemme}\label{lemme_symetrique}
	Soit $c$ une classe $r$-symétrique avec la plus grande multiplicité qui est égale au tiers du degré de $c$ : \[c=n\l-\frac{n}{3}\sum\limits_{i=0}^{r-1}e_{p_i}-\sum_{\substack{p\in\B(\P^2)\\p\notin\{p_0,\dots,p_{r-1}\}}}\lambda_pe_{p},\] où $\lambda_p<\frac{n}{3}$ pour tout $p\notin\{p_0,\dots,p_{r-1}\}$. Nous avons alors :\begin{enumerate}[wide]
		\item Si $c\in\orbl$ alors $r\leq 8$. 
		\item Si $c\in\partial_{\infty}\orbl$ alors $r=9$ et $c$ est une classe $9$-symétrique pure.
	\end{enumerate} 
\end{lemme}
\begin{proof} Considérons une telle classe $c$.
	Dans le cas où $c\in\orbl$, si $r\geq 9$ alors nous obtenons la contradiction suivante : \[1=c^2\leq n^2-9\frac{n^2}{9}-\sum_{p\notin\{p_0,\dots,p_{r-1}\}}\lambda_p^2\leq 0.\]
		
	Si maintenant $c\in\partial_{\infty}\orbl$ alors $c^2=0$ et par le même argument, nous obtenons $r\leq 9$. Dans le cas où $r=9$, notons que la classe $c$ est forcément symétrique pure.
	Supposons que $r\leq 8$. Alors il existe $p\notin\{p_0,\dots,p_{r-1}\}$ tel que $\lambda_p>0$ puisque : 
	\begin{equation}\label{eq1}
	n^2-r\frac{n^2}{9}-\sum\limits_{p\notin\{p_0,\dots,p_{r-1}\}}\lambda_p^2=0.
	\end{equation}
	Comme $c$ appartient au bord à l'infini, nous avons par définition :
	\begin{equation}\label{eq2}
	0\leq -\fcan\cdot c= 3n-r\frac{n}{3}-\sum\limits_{p\notin\{p_0,\dots,p_{r-1}\}}\lambda_p.
	\end{equation}
	En multipliant l'égalité \ref{eq1} par $3$, nous avons $3n^2=3r\frac{n^2}{9}+3\sum\limits_{p\notin\{p_0,\dots,p_{r-1}\}}\lambda_p^2$.
	En multipliant l'équation \ref{eq2} par $n$ qui est positif et en remplaçant dans cette équation $3n^2$ par la valeur que nous venons d'obtenir, nous avons 
	\[0\leq 3r\frac{n^2}{9}+3\sum\limits_{p\notin\{p_0,\dots,p_{r-1}\}}\lambda_p^2 - r\frac{n^2}{3}-n\sum\limits_{p\notin\{p_0,\dots,p_{r-1}\}}\lambda_p  =  \sum\limits_{p\notin\{p_0,\dots,p_{r-1}\}}\lambda_p(3\lambda_p-n),\]
	ce qui contredit le fait que $\lambda_p<\frac{n}{3}$ pour tout $p\notin\{p_0,\dots,p_{r-1}\}$. Ainsi $r=9$ comme attendu.
\end{proof}
Remarquons que ce lemme est l'analogue du lemme \ref{lemme_appli_sym} sur les caractéristiques des applications symétriques. Cependant, comme les multiplicités des classes ne sont pas entières cela enlève de la rigidité, d'où l'hypothèse imposée sur les coefficients.

\section{Cellule de Voronoï associée à l'identité}\label{section_cellule_identite}
 Le groupe de Cremona agissant transitivement sur les cellules de Voronoï, il suffit d'étudier la cellule $\V(\id)$ pour comprendre toutes les cellules. Dans cette section, nous caractérisons les classes se trouvant dans la cellule $\V(\id)$ (Proposition \ref{prop_points_alignes_posgen} et Théorème \ref{prop_cellule_identite}). 
    Par définition, une classe $c$ est dans la cellule de Voronoï associée à l'identité si elle ne peut pas être rapprochée de $\l$ en faisant agir un élément de $\Bir(\P^2)$ : 
    \[c\in\V(\id) \Leftrightarrow  \text{ pour tout } f\in\Bir(\P^2),\ \isome{f}(c)\cdot \l\geq c\cdot\l.\]

    Avant cela, introduisons des notations utilisées dans cette section et par la suite. Nous avons vu que le support d'une classe $c\in\orbl$ (ou plus généralement dans $\H$) est dénombrable. Par facilité d'écriture, étant donnée une classe $c\in\orbl$, nous renumérotons toujours dans cette section les points de son support et les multiplicités correspondantes, de sorte qu'ils soient indicés par un sous-ensemble $I'$ de $\N$ qui est soit un intervalle commençant à $0$ soit $\N$ tout entier :  \[c=n\l-\sum\limits_{i\in I'}\lambda_ie_{p_i}.\] Remarquons que pour alléger les notations, nous écrivons $\lambda_i$ au lieu de $\lambda_{p_i}$. 
    Par la suite, lorsque nous ferons agir une application $f\in\Bir(\P^2)$ sur $c$, nous aurons également besoin d'une notation pour les points-base de $f$. S'ils sont dans le support de $c$, la notation des points-base de $f$ sera induite par celle sur les points du support de $c$. Si ce n'est pas le cas, nous aurons besoin d'introduire un ensemble fini $J$ disjoint de $\N$, tel que pour $j\in J$, $\lambda_{j}$ est la multiplicité nulle de $c$ correspondant au point-base $p_j$ de $f$ qui n'est pas dans le support de $c$. De fait, pour tout $i\in I'$ et pour tout $j\in J$, $\lambda_i> \lambda_j$. Ainsi, pour alléger les notation, $I$ correspondra suivant le contexte à $I'$ ou à $I'\cup J$ et les indices de $J$ seront considérés comme plus grands que ceux de $I'$. Une classe est dite \emph{ordonnée} lorsque la numérotation induit un ordre décroissant sur les multiplicités de $c$ : pour tout $i,j\in I$ tels que $i<j$, $\lambda_i\geq \lambda_j$.
    
    Considérons l'action d'une application $f\in\Bir(\P^2)$ de degré $d$ sur une classe $c=n\l-\sum_{i\in I}\lambda_ie_{p_i}\in \orbl$. Pour tout $i\in\I$ tel que $p_i$ est un point-base de $f$, nous notons $m_i$ sa multiplicité pour $f$ en tant que point-base. Alors nous avons l'égalité suivante : 
    \begin{align}\label{lemme_calcul_fc_id_c}
    \begin{split}
    \isome{f}(c)\cdot \l-c\cdot\l &= c\cdot \isome{f}^{-1}(\l)-c\cdot\l\\
    &=c\cdot (d\l-\sum\limits_{p_i\in\supp(f)}m_ie_{p_i})-n\\
    &=(d-1)n-\sum\limits_{p_i\in\supp(f)}m_i\lambda_i
    \end{split}
    \end{align}
    
    Dans le cas où la somme des trois plus grandes multiplicités d'une classe ordonnée est inférieure ou égale à son degré, le lemme suivant permet de vérifier rapidement que cette classe appartient à $\V(\id)$. 
    
    \begin{lemme}\label{lemme_reformulation_identite}
    	Soit $f\in\Bir(\P^2)$ de degré $d$ et $c=n\l-\sum_{i\in I}\lambda_ie_{p_i}\in \orbl$. Notons, pour tout $i\in I$ tel que $p_i\in\supp(f)$, $m_i$ la multiplicité du point-base $p_i$. Considérons le multi-ensemble $E$ constitué des indices de $I$, chacun apparaissant avec multiplicité $m_i$. Alors, il est possible de partitionner $E$ en $d-1$ triplets et d'avoir ainsi, en notant $T$ l'ensemble de ces triplets : 
    	\[\isome{f}(c)\cdot \l-c\cdot\l=\sum\limits_{ \{\!\{ i,j,k\}\!\}\in T} n-\lambda_i-\lambda_j-\lambda_k.\]
    	De plus, il est possible de choisir la partition de $E$ de sorte que chaque triplet de $E$ soit constitué de trois indices deux à deux distincts. 
    \end{lemme}
    Rappelons que les indices de $I$ correspondent à des points du support de $c$ ou du support de $f$.
    \begin{proof}
    	D'après l'équation \eqref{lemme_calcul_fc_id_c}, nous avons l'égalité : 
    	\begin{equation*}\tag{$\star$}
    	\isome{f}(c)\cdot \l-c\cdot\l=(d-1)n-\sum\limits_{p_i\in\supp(f)}m_i\lambda_i.
    	\label{align_equation_lemme_reecriture}
    	\end{equation*}
    	L'égalité du lemme \ref{lemme_equations_evidentes}.\ref{eq_Noethercanonique} \[3(d-1)=\sum_{p_i\in\supp(f)}m_i\] implique que nous pouvons constituer $d-1$ triplets dans $E$ afin d'arranger le terme de droite de \eqref{align_equation_lemme_reecriture} comme une somme de $d-1$ termes de la forme $n-\lambda_i-\lambda_j-\lambda_k$, où $\{\!\{ i,j,k\}\!\}\in T$. Ainsi 
    	\[ (d-1)n-\sum\limits_{p_i\in\supp(f)}m_i\lambda_i=\sum\limits_{\{\!\{ i,j,k\}\!\}\in T} n-\lambda_i-\lambda_j-\lambda_k.\]    De plus, pour chaque $i$, l'inégalité du lemme \ref{lemme_equations_evidentes}.\ref{eq_mi_pluspetit_d}\[m_i\leq d-1,\] permet de répartir les $m_i$ multiplicités $\lambda_i$ de sorte qu'il n'y en ait pas deux dans le même terme de la somme et donc que les triplets de $T$ soient formés d'indices deux à deux distincts. Une façon de former ces triplets est de faire une \emph{répartition en suivant}. Nous répartissons le premier indice $i_0$ de multiplicité non nulle $m_{i_0}$ dans les $m_{i_0}$ premiers triplets. Puis nous répartissons le second indice $i_1$ dans les $m_{i_1}$ triplets suivants en revenant au premier quand nous avons complété le dernier. En itérant ce procédé nous obtenons la partition en triplets attendue.
    \end{proof}
    \begin{lemme}\label{lemme_pointsalignes_posgenerale}
    	Soit $c\in\orbl$ une classe de degré $n$. Supposons que l'une des deux conditions suivantes est réalisée.
    	\begin{enumerate}[wide]
    		\item Les points $p_{i}$, $p_{j}$ et $p_{k}$ sont alignés.
    		\item Les points $p_{i}$, $p_{j}$ et $p_{k}$ sont le support d'une application quadratique $\q$ et \[\isome{\q}(c)\cdot\l\geq c\cdot\l.\]
    	\end{enumerate} Alors, en notant $\lambda_{i}$, $\lambda_{j}$ et $\lambda_{k}$ les multiplicités respectives, éventuellement nulles, des points $p_{i}$, $p_{j}$ et $p_{k}$ pour $c$, nous avons : \[n-\lambda_{i}-\lambda_{j}-\lambda_{k}\geq 0.\]
    \end{lemme} 
    \begin{proof}
    	Si les points $p_{i}$, $p_{j}$ et $p_{k}$ sont alignés alors en considérant la droite passant par ces trois points, nous avons d'après l'inégalité de Bézout (Définition \ref{proprietes_c}.\ref{propriete_classe_bezout}) le résultat attendu.
    	Si les points $p_{i}$, $p_{j}$ et $p_{k}$ sont les points-base de l'application quadratique $\q$, l'équation \eqref{lemme_calcul_fc_id_c} nous donne : \[n-\lambda_{i}-\lambda_{j}-\lambda_{k}=\isome{\q}(c)\cdot\l- c\cdot\l \geq 0\] ce qui prouve le lemme sous la seconde condition.
    \end{proof}
    
     L'objet des énoncés qui suivent est de caractériser les classes appartenant à $\V(\id)$. Nous étudions dans un premier temps (Lemme \ref{lemme_3_grand_coeff_pos_c_in_id} et Proposition \ref{prop_points_alignes_posgen}) des conditions suffisantes pour appartenir à la cellule $\V(\id)$. C'est la partie facile puisque n'importe quel choix de partition de $E$ en triplets d'indices deux à deux disjoints convient. Cependant, ces conditions ne sont pas nécessaires à cause de classes particulières. Nous étudions ces classes lors du Théorème \ref{prop_cellule_identite}.
    
    \begin{lemme}\label{lemme_3_grand_coeff_pos_c_in_id}
        	Soit $c=n\l-\sum_{i\in I}\lambda_ie_{p_i}\in \orbl$ une classe ordonnée. Si la somme des trois plus grandes multiplicités de $c$ est inférieure ou égale à son degré : \[n-\lambda_0-\lambda_1-\lambda_2\geq 0,\]
        	alors $c\in\V(\id)$.
        \end{lemme}
        \begin{proof}
        	Considérons une telle classe $c$. L'hypothèse ainsi que la décroissance des multiplicités de $c$ impliquent que pour tout triplet $\{\!\{ i,j,k\}\!\}$ d'indices deux à deux distincts de $I$, $n-\lambda_i-\lambda_j-\lambda_k\geq 0$.  Ainsi, par le lemme \ref{lemme_reformulation_identite}, pour tout $f\in\Bir(\P^2)$, \[\isome{f}(c)\cdot\l\geq c\cdot \l.\] Par conséquent $c\in\V(\id)$ comme attendu. 
        \end{proof}
    
    \begin{prop}\label{prop_points_alignes_posgen}
        	Soit $c=n\l-\sum_{i\in I}\lambda_ie_{p_i}\in \orbl$ une classe ordonnée.
        	\begin{enumerate}
        		\item Si les points $p_0$, $p_1$ et $p_2$ sont alignés alors $n-\lambda_0-\lambda_1-\lambda_2\geq 0$ et $c\in\V(\id)$.
        		\item Si les points $p_0$, $p_1$ et $p_2$ sont le support d'une application quadratique alors la classe $c$ appartient à $\V(\id)$ si et seulement si \[n-\lambda_0-\lambda_1-\lambda_2\geq 0.\]
        		\item Si les points $p_1$ et $p_2$ sont adhérents à $p_0$ et si $n-\lambda_0-\lambda_1-\lambda_2\geq 0$ alors $c\in\V(\id)$.
        	\end{enumerate}
        \end{prop}
        \begin{proof}
            	\begin{enumerate}[wide,itemsep=0pt]
            		\item Le premier point découle des lemmes \ref{lemme_pointsalignes_posgenerale} et \ref{lemme_3_grand_coeff_pos_c_in_id}.
            		\item Soit $\q$ une application quadratique ayant $p_0$, $p_1$ et $p_2$ comme points-base. Si $c\in\V(\id)$, la définition de cellule de Voronoï implique qu'en particulier, \[\isome{q}(c)\cdot \l=c\cdot\isome{q}^{-1}(\l) \geq c\cdot \l .\] Le lemme \ref{lemme_pointsalignes_posgenerale} permet d'obtenir l'inégalité souhaitée. La réciproque découle du lemme \ref{lemme_3_grand_coeff_pos_c_in_id}.
            		\item Le troisième point découle du lemme \ref{lemme_3_grand_coeff_pos_c_in_id}.\qedhere
            	\end{enumerate}
            \end{proof}
        
    Une classe ordonnée qui possède au moins trois points dans son support est appelée \emph{spéciale} si les points $p_1$ et $p_2$ sont adhérents à $p_0$ et si ses trois plus grandes multiplicités vérifient : $0>n-\lambda_0-\lambda_1-\lambda_2$. 
    \begin{rmq}\label{rmq_classe_speciale_lambda0_nsur2}
    En particulier, une classe spéciale ne possède qu'une seule multiplicité maximale. En effet, si $c$ est une classe spéciale alors $\lambda_0>\frac{n}{2}$ puisque par positivité des excès (Définition \ref{proprietes_c}.\ref{propriete_classe_exces}) $\lambda_0\geq \lambda_1+\lambda_2$ et donc \[0> n-\lambda_0-\lambda_1-\lambda_2\geq n-2\lambda_0.\] La conclusion s'obtient avec l'inégalité de Bézout (Définition \ref{proprietes_c}.\ref{propriete_classe_bezout}).
    \end{rmq}
       
      Si une classe $c$ ordonnée n'est pas spéciale alors soit les points $p_0$, $p_1$ et $p_2$ sont alignés, soit ils forment le support d'une application quadratique, soit les points $p_1$ et $p_2$ sont adhérents à $p_0$ et $n-\lambda_0-\lambda_1-\lambda_2\geq 0$. Par positivité des excès en $p_0$ (Définition \ref{proprietes_c}.\ref{propriete_classe_exces}) il n'y a pas d'autre possibilité. Par conséquent, d'après la proposition \ref{prop_points_alignes_posgen} et le lemme \ref{lemme_reformulation_identite}, nous obtenons une caractérisation immédiate pour contrôler si une classe non spéciale ordonnée appartient à $\V(\id)$ :
    
    \begin{cor}\label{rmq_classe_non-spe_V(id)}
    Une classe $c$ ordonnée et non spéciale appartient à $\V(\id)$ si et seulement si elle a un degré supérieur ou égal à la somme de ses trois plus grandes multiplicités : \[n-\lambda_0-\lambda_1-\lambda_2\geq 0.\]
    \end{cor}

    Comme nous le verrons dans les exemples \ref{ex_classe_speciale2} et \ref{classe_spéciale1}, il existe des classes spéciales appartenant à $\V(\id)$. Par conséquent, la caractérisation du corollaire \ref{rmq_classe_non-spe_V(id)} ne caractérise pas toutes les classes de $\V(\id)$. Ceci nous oblige à faire une disjonction de cas suivant si les classes sont spéciales ou non. De plus, il est difficile d'obtenir un critère aussi simple que dans le cas des classes non spéciales pour vérifier qu'une classe spéciale appartient à $\V(\id)$. Le théorème \ref{prop_cellule_identite} permet de réduire les applications birationnelles à tester pour vérifier qu'une classe spéciale est dans $\V(\id)$. Il suffit de faire agir les applications de caractéristique Jonquières ayant comme point-base maximal le point correspondant à la plus grande multiplicité de $c$. En fait ce n'est pas si surprenant que le cas des classes spéciales soit plus compliqué que le cas des classes non spéciales. Les trois premiers points du support des classes spéciales sont précisément dans la configuration qui avait été oubliée par M. Noether et qui a posé problème dans sa preuve du théorème de Noether-Castelnuovo (voir le commentaire fait après le théorème \ref{thm_générateurs_Cremona}).
    \begin{thm}\label{prop_cellule_identite}
    	Soit $c=n\l-\sum_{i\in I}\lambda_ie_{p_i}\in \orbl$ une classe spéciale ordonnée.
    	\begin{enumerate}
    		\item\label{prop_cellule_identite_appartenir} La classe $c$ est dans $\V(\id)$ si et seulement si pour toute application de Jonquières $\j$ et de point-base maximal $p_0$ nous avons 
    		\[\isome{j}(c)\cdot\l\geq c\cdot\l.\]
    		\item\label{prop_cellule_identite_bord} De plus, si $c\in\V(\id)\cap\V(f)$ alors $f$ est une application de caractéristique Jonquières dont l'inverse a comme point-base maximal $p_0$. En notant $p_{i_1},\dots,p_{i_{2d-2}}$ les petits points-base de $f^{-1}$, $c$ vérifie : 
    		\[n=\lambda_0+\frac{1}{d-1}\sum\limits_{i\in\{i_1,\dots,i_{2d-2}\}}\lambda_i.\]
    	\end{enumerate}
    \end{thm}
    Si $f$ est une application de caractéristique Jonquières alors elle envoie un pinceau de droites sur un autre pinceau de droites. Par conséquent, en la pré-composant par un automorphisme, c'est une application de Jonquières. Ainsi il existe un représentant de $\overline{f}$ qui est une application de Jonquières. 
    Le reste de cette section est consacré à la preuve de ce théorème, et sera complété par deux exemples de classes spéciales.
    La démonstration est technique et est décomposée en faits pour faciliter la lecture. L'implication rapide est faite en premier. Pour montrer la réciproque, l'astuce est d'écrire la condition \[\isome{f}(c)\cdot\l\geq c\cdot\l,\] sous forme d'une somme mettant en jeu les multiplicités de $c$ qui sont supportées par les points-base de $f$ de sorte que tous les termes de la somme soient positifs. 
    \begin{proof}
    	Prouvons dans un premier temps le point \ref{prop_cellule_identite_appartenir}.
    	L'implication découle de la définition de cellule de Voronoï. En effet, une classe est dans la cellule associée à l'identité si et seulement si pour toute application $f\in\Bir(\P^2)$, \[\isome{f}(c)\cdot \l \geq c\cdot \l.\] Par conséquent, si $c\in\V(\id)$ alors en particulier, pour toute application de Jonquières $\j$ de point-base maximal $p_0$ nous avons : 
    	\[\isome{\j}(c)\cdot \l \geq c\cdot\l.\]

    	Intéressons nous à la réciproque du point \ref{prop_cellule_identite_appartenir}. Supposons que pour toute application de Jonquières $\j$ de point-base maximal $p_0$, nous ayons :\begin{equation}\label{eq_condition_sur_c}
    	\isome{\j}(c)\cdot \l \geq c\cdot\l.
    	\end{equation} Nous devons montrer que pour toute application $f\in\Bir(\P^2)$, \[\isome{f}(c)\cdot \l \geq c\cdot\l.\] 
    	Fixons $f\in\Bir(\P^2)$ de degré $d$ dont les $r$ points-base $\{p_i\}$ sont de multiplicité respective $\{m_i\}$. La numérotation est induite par celle déjà faite sur les points $p_i$ du support de $c$ : $i\leq j$ si et seulement si $\lambda_i\geq \lambda_j$. Rappelons que si le point $p_i$ n'est pas dans le support de $c$ alors nous posons $\lambda_i=0$. D'après l'égalité \eqref{lemme_calcul_fc_id_c} et le lemme \ref{lemme_reformulation_identite}, et en gardant ses notations, nous devons montrer que la somme suivante est positive :
    	\begin{equation}\tag{$\star \star $}
    	(d-1)n-\sum_{p_i\in\supp(f)}m_i\lambda_i =\sum\limits_{\{\!\{ i,j,k\}\!\}\in T} n-\lambda_i-\lambda_j-\lambda_k\geq 0.
    	\label{equation_somme_terme_n_lambda}
    	\end{equation}
    	Le reste de la preuve consiste à choisir, lorsque c'est possible, les triplets de $T$ de sorte que chaque terme soit positif. Lorsque ça ne sera pas le cas, nous partitionnerons également $E$ avec des sous-ensembles de tailles différentes afin d'obtenir également des sommes de termes positifs.
    	
    	\begin{fait}\label{fait_positif_0}
    		Pour tout triplet d'indices $\{i_1,i_2,i_3\}$ tels que $0<i_1< i_2\leq i_3$, nous avons : \[n-\lambda_{i_1}-\lambda_{i_2}-\lambda_{i_3}\geq 0.\] De plus, si ce terme est nul alors pour tout $1\leq j\leq i_3$, 
    		$\lambda_j =n-\lambda_0$. 
    	\end{fait}
    	\begin{proof}
    		Comme la classe $c$ est spéciale les points $p_1$ et $p_2$ sont adhérents à $p_0$. Par positivité des excès en $p_0$ pour $c$ et par l'inégalité de Bézout pour la droite passant par les points $p_0$ et $p_2$, nous avons : \[\lambda_{i_1}+\lambda_{i_2}+\lambda_{i_3}\leq \lambda_1+2\lambda_2\underset{\ref{proprietes_c}.\ref{propriete_classe_exces}}{\leq} \lambda_0+\lambda_2\underset{\ref{proprietes_c}.\ref{propriete_classe_bezout}}{\leq}  n.\] 
    		Si $n-\lambda_{i_1}-\lambda_{i_2}-\lambda_{i_3}=0$ alors en considérant le cas d'égalité dans les inégalités ci-dessus nous obtenons \[\lambda_{i_3}=\lambda_2=n-\lambda_0.\] De plus, en considérant la droite passant par les points $p_0$ et $p_1$, l'inégalité de Bézout implique que $\lambda_1\leq n-\lambda_0$. Nous concluons par décroissance des $\lambda_i$.
    	\end{proof}

    	\begin{fait}\label{fait_p_0_pas_maximal}
    		Si le point $p_0$ n'est pas un point-base de multiplicité maximale pour $f$ alors \[\isome{f}(c)\cdot \l> c\cdot \l.\]
    	\end{fait}
    	
    	\begin{proof}
    		Si le point $p_0$ n'est pas un point-base de multiplicité maximale pour $f$, alors il existe un point-base de $f$ dans $\P^2$, noté $p_{r}$, tel que \[d-1\underset{\ref{lemme_equations_evidentes}.\ref{eq_mi_pluspetit_d}}{\geq} m_{r}>m_0.\]Nous posons comme convention que si $p_0$ n'est pas un point-base de $f$, $m_0=0$. Nous pouvons réécrire l'égalité \eqref{equation_somme_terme_n_lambda} : 
    		\begin{align*}
    		\sum\limits_{\{\!\{ i,j,k\}\!\}\in T} n-\lambda_i-\lambda_j-\lambda_k=(d-1)n-\sum_{p_i\in\supp(f)}m_i\lambda_i & = m_0(n-\lambda_0-\lambda_{r})\\
    		&+(m_{r}-m_0)(n-\lambda_{r})\\
    		& +(d-1-m_r)n \\
    		& -\sum_{\substack{p_i\in\supp(f)\\i\notin \{ 0,r\}}}m_i\lambda_i.
    		\end{align*}
    		Rappelons que nous avons noté $E$ le multi-ensemble constitué des indices des points-base de $f$, chacun apparaissant avec multiplicité $m_i$ et que $T$ est une partition de $E$ en $d-1$ triplets. Nous voulons modifier $T$ de sorte qu'il y ait $m_0$ triplets possédant les indices $0$ et $r$, $m_r-m_0$ triplets possédant l'indice $r$, $d-1-m_r$ triplets ne possédant aucun des indices $0$ et $r$ et que tous les triplets $\{\!\{ i,j,k\}\!\}$ aient la propriété suivante : $i<j\leq k$. Nous construisons trois sous-ensembles de triplets nommés $T_{0,r}$, $T_r$ et $T_{\emptyset}$ constitués respectivement des triplets contenant les indices $0$ et $r$, $r$ mais pas $0$, et enfin ni $0$ ni $r$.   
   Nous répartissons les indices $i \in E\setminus\{0,r\}$ de la façon suivante. Considérons les plus petits indices $i_0\in E\setminus \{0,r\}$. Si $d-1>m_r$, nous commençons par placer exactement un des $m_{i_0}$ coefficients $i_0$ dans chacun des $\min(d-1-m_r, m_{i_0})$ premiers triplets de $T_{\emptyset}$. Si $m_{i_0}>d-1-m_r$ nous continuons de la même manière à répartir exactement un indice $i_0$ dans les $\min(m_r-m_0,d-1-m_0)$ premiers triplets de $T_{r}$. Enfin, si $m_{i_0}>d-1-m_0$ nous répartissons les derniers indices $i_0$ dans les premiers triplets de $T_{0,r}$. Remarquons que $m_{i_0}<d-1$ et par conséquent les $m_{i_0}$ indices ${i_0}$ ont été répartis dans des termes différents. Puis nous considérons les plus petits indices $i_1\in E\setminus\{0,r,i_0\}$. Nous commençons à répartir les indices $i_1$ à partir de là où nous nous sommes arrêté et nous suivons la même stratégie. Une fois que nous avons complété les $m_0$ triplets de $T_{0,r}$, nous recommençons à compléter les triplets de $T_{\emptyset}$. Nous continuons le procédé, en remarquant qu'à présent il n'est plus possible de compléter les triplets de $T_{0,r}$ puisqu'ils possèdent trois indices. Une fois que nous avons ajouté un second indice aux triplets $T_{\emptyset}$ et $T_{r}$, les triplets de $T_r$ sont complets et nous répartissons les indices restants dans les triplets de $T_{\emptyset}$.
   Par la suite, nous ne serons pas aussi précis et nous dirons que nous avons « complété les termes des trois premières lignes à l'aide des coefficients $\lambda_i$ restants ».
     		
    		Ainsi $T_{\emptyset}\cup T_{r}\cup T_{0,r}$ forme une partition de $E$ en triplets. De plus chaque triplet $\{\!\{ i,j,k\}\!\}$ satisfait $i<j\leq k$. Montrons que les trois types de termes sont positifs ou nuls et qu'il existe au moins un terme strictement positif.\begin{itemize}[wide]
    			\item Par le fait \ref{fait_positif_0}, les $(d-1-m_r)$ termes de la forme $n-\lambda_{i}-\lambda_{j}-\lambda_{k}$ où $\{\!\{ i,j,k\}\!\}\in T_{\emptyset} $, sont positifs ou nuls. 
    			\item Par le fait \ref{fait_positif_0}, les $m_r-m_0$ termes de la forme $n-\lambda_{r}-\lambda_{i}-\lambda_{j}$ où $\{\!\{ i,j,r\}\!\}\in T_r$ sont positifs. Ils sont en fait strictement positifs. Sinon, toujours par le fait \ref{fait_positif_0}, nous aurions, en notant $k$ le plus petit indice strictement plus grand que $0$ et différent de $r$ (pas forcément correspondant à un point-base de $f$), $\lambda_k=\lambda_{r}=n-\lambda_0$. Ceci implique : \[n-\lambda_0-\lambda_k-\lambda_{r}=\lambda_0-n<0.\] Or les points $p_0$ et $p_{r}$ étant distincts dans $\P^2$, les points $p_0$, $p_k$ et $p_{r}$ sont soit le support d'une application quadratique, soit alignés ce qui contredit le lemme \ref{lemme_pointsalignes_posgenerale}. De plus, comme $m_r-m_0>0$ il existe au moins un terme de cette forme.

    			\item Intéressons-nous aux $m_0$ autres termes : $n-\lambda_0-\lambda_{r}-\lambda_k$ où $\{\!\{ 0,r,k\}\!\}\in T_{0,r}$. 
    			Il existe un point $p_{k_1}$ tel que $p_{k}$ soit voisin de, ou égal à $p_{k_1}$ et que les points $p_0$, $p_{r}$ et $p_{k_1}$ soient le support d'une application quadratique ou soient alignés. Ainsi par le lemme \ref{lemme_pointsalignes_posgenerale} et par la positivité de l'excès au point $p_{k_1}$ (Définition \ref{proprietes_c}.\ref{propriete_classe_exces}), nous avons : 
    			\[n-\lambda_0-\lambda_{r}-\lambda_k\underset{\ref{proprietes_c}.\ref{propriete_classe_exces}}{\geq} n-\lambda_0-\lambda_{r}-\lambda_{k_1}\underset{\ref{lemme_pointsalignes_posgenerale}}{\geq}0.\]
    		\end{itemize} 
    		Par conséquent, avec l'ensemble des triplets $T_{\emptyset}\cup T_{r}\cup T_{0,r}$ définis ainsi, les termes de \eqref{equation_somme_terme_n_lambda} sont positifs et au moins un terme est strictement positif ce qui implique que cette somme est strictement positive.
    	\end{proof}
   
    	Il nous reste à étudier le cas où le point $p_0$ est de multiplicité maximale pour $f$ qui va se révéler être le cas le plus difficile.
    	Considérons les sous-ensembles disjoints d'indices des points-base de $f$ suivants : 
    	\begin{itemize}\phantomsection\label{famille_poins}
    		\item MIN est l'ensemble des indices $i$ tels que $p_i\in \supp(f)$ et est un point mineur de $c$, c'est-à-dire $\lambda_i\leq\frac{n-\lambda_0}{2}$.
    		\item ADH est l'ensemble des indices $i\notin \text{MIN}$ tels que $p_i\in \supp(f)$ et est un point adhérent à $p_0$.
    		\item VNA est l'ensemble des indices $i\notin \text{MIN}$ tels que $p_i\in\supp(f)$ et est un point voisin et non adhérent à $p_0$.
    	\end{itemize}
    	 	\begin{figure}[h]
    	     	    		\begin{center}
    	     	    		    	\scalebox{.90}{
    	    	    	       	    			\begin{tikzpicture}
    	    	    	       	    			\coordinate (p0) at (0,0);
    	    	    	       	    			\coordinate (pi1) at (0,1);
    	    	    	       	    			\coordinate (pi2) at (0,2);
    	    	    	       	    			\coordinate (pi4) at (0,3);
    	    	    	       	    			\coordinate (pi3) at (1,1);
    	    	    	       	    			\coordinate (pi8) at (0,4);
    	    	    	       	    			\coordinate (pi5) at (-1,1);
    	    	    	       	    			\coordinate (pi6) at (-2,2);
    	    	    	       	    			\coordinate (pi7) at (2,2);
    	    	    	       	    			\fill (p0) circle (2pt);
    	    	    	       	    			\fill (pi1) circle (2pt);
    	    	    	       	    			\fill (pi2) circle (2pt);
    	    	    	       	    			\fill (pi3) circle (2pt);
    	    	    	       	    			\fill (pi4) circle (2pt);
    	    	    	       	    			\fill (pi5) circle (2pt);
    	    	    	       	    			\fill (pi6) circle (2pt);
    	    	    	       	    			\fill (pi7) circle (2pt);
    	    	    	   						\fill (pi8) circle (2pt);
    	    	    	       	    			\node at (p0) [below right] {$p_0$};
    	    	    	       	    			\node at (pi1) [right] {$p_{1}$};
    	    	    	       	    			\node at (pi2) [right] {$p_{2}$};
    	    	    	       	    			\node at (pi3) [right] {$p_{3}$};
    	    	    	       	    			\node at (pi4) [right] {$p_{4}$};
    	    	    	       	    			\node at (pi5) [right] {$p_{5}$};
    	    	    	       	    			\node at (pi6) [right] {$p_{6}$};
    	    	    	       	    			\node at (pi7) [right] {$p_{7}$};
    	    	    	       	    			\node at (pi8) [right] {$p_{8}$};
    	    	    	       	    			\draw [directed](pi1) to (p0);
    	    	    	       	    			\draw [directed](pi2) to[bend right=30] (p0);
    	    	    	       	    			\draw [directed](pi2) to (pi1);
    	    	    	       	    			\draw [directed](pi4) to (pi2);
    	    	    	       	    			\draw [directed](pi8) to (pi4);
    	    	    	       	    			\draw [directed](pi7) to (pi3);
    	    	    	       	    			\draw [directed](pi3) to (p0);
    	    	    	       	    			\draw [directed](pi5) to (p0);
    	    	    	       	    			\draw [directed](pi6) to (pi5);
    	    	    	       	    			\draw [directed](pi6) to[bend right=30] (p0);
    	    	    	       	    			\end{tikzpicture}}
    	    	    	       	    			\caption{Exemples de points indicés par MIN, ADH et VNA \label{ex_pt_min_adh_vna}}
    	    	    	       	    		\end{center}
    	    	    	       	    	\end{figure}
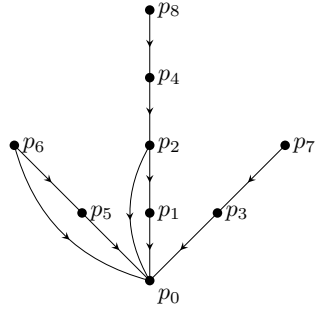
    	\begin{ex}
    	Considérons la classe ordonnée \[c=4\l-3e_{p_0}-\frac{2}{3}e_{p_1}-\frac{2}{3}e_{p_2}-\frac{2}{3}e_{p_3}-\frac{7}{12}e_{p_4}-\frac{1}{3}e_{p_5}-\frac{1}{3}e_{p_6}-\frac{1}{3}e_{p_7}-\frac{1}{12}e_{p_8}-\sum\lambda_ie_{p_i},\]
    	où les points $p_1$, $p_2$, $p_3$, $p_5$ et $p_6$ sont adhérents au point $p_0$, le point $p_8$ est adhérent au point $p_4$ qui est adhérent au point $p_2$, le point $p_7$ est adhérent au point $p_3$ et le point $p_6$ est adhérent au point $p_5$ (voir Figure \ref{ex_pt_min_adh_vna}). Supposons de plus que l'ensemble $\{p_0,p_1,p_2,\cdots, p_8\}$ est un sous-ensemble des points-base de $f$ dont $p_0$ est un point-base de multiplicité maximale.
    	Alors $\{1,2,3\}\subset \text{ADH}$, $4\in\text{VNA}$, $\{i\geq 5\mid p_i\in\supp(c)\cap\supp(f)\}\in\text{MIN}$.
    	\end{ex}
       	
    	\begin{fait}\label{fait_adh_vna_min}
    	Les ensembles MIN, ADH et VNA correspondent à une partition des points-base de $f$ hors $p_0$.
    	\end{fait}
    	
    	\begin{proof}
    	Ces ensembles d'indices étant disjoints, il suffit de montrer que tout point-base de $f$ hors $p_0$ est indicé par un élément de $\text{ADH}\cup\text{VNA}\cup\text{MIN}$. 
    	
    	Tous les points-base de $f$ voisins du point $p_0$ sont indicés par un élément de ces trois sous-ensemble. Montrons que les points de $\P^2$ ou voisins d'un point de $\P^2$ différent de $p_0$ sont mineurs pour $c$.    	
    	Soit $p_m\in\P^2$ alors les points $p_0$, $p_1$, $p_m$ sont soit alignés soit le support d'une application quadratique donc d'après le lemme \ref{lemme_pointsalignes_posgenerale} nous obtenons l'inégalité : 
    	\[0\underset{\ref{lemme_pointsalignes_posgenerale}}{\leq} n-\lambda_0-\lambda_1-\lambda_m\leq n-\lambda_0-2\lambda_m,\] qui permet de conclure que le point $p_m$ est mineur pour $c$. De plus, tout point $p_{m'}$ voisin de $p_m$ est mineur puisque $\lambda_{m'}\leq \lambda_m$ par positivité des excès pour $c$ en $p_m$ et en tous les points qu'il faut éclater pour obtenir $p_m'$ (Définition \ref{proprietes_c}.\ref{propriete_classe_exces}).
    	\end{proof}
    	
    	Remarquons que par définition de l'ensemble MIN, tout terme de la forme suivante est positif : \[n-\lambda_0-\lambda_i-\lambda_j \geq 0, \text{ pour tous } i,j\in\text{MIN}.\] 
    	Notons que les indices $i$ et $j$ ne sont pas forcément distincts.
    	Si l'ensemble ADH est vide, il en est de même de VNA par positivité des excès pour $c$ et pour $f$ (Définition \ref{proprietes_c}.\ref{propriete_classe_exces} et Proposition \ref{prop_consistence}). Par conséquent, dans ce cas-là tous les indices hors $0$ sont dans MIN et la somme de \eqref{equation_somme_terme_n_lambda} est positive car tous les termes le sont. Supposons donc que ADH est non vide. Le but dans ce qui suit est d'arranger de façon adéquate les multiplicités des points indicés par ADH et par VNA afin d'obtenir une somme dont chaque terme est positif.
    	
    	\begin{fait}\label{fait_max_1ptindice_ADH_VNA}
    		Chaque point voisin de $p_0$ possède au plus un point adhérent indicé par $\text{ADH}\cup\text{VNA}$, les autres, s'ils existent, sont indicés par MIN. 
    	\end{fait}    	
    	\begin{proof}
    		Soient $p_{m_1}$ et $p_{m_2}$ deux points adhérents au point $p_m$ voisin de $p_0$, et tels que $\lambda_{m_2}\leq \lambda_{m_1}$ (voir Figure \ref{figure_fait_max_1ptindice_ADH_VNA}). Notons $p_r$, possiblement égal au point $p_m$, le point libre et adhérent à $p_0$ tel que $p_m$ soit voisin de $p_r$. Alors d'après l'inégalité de Bézout pour la droite passant par les points $p_0$ et $p_r$ et la positivité de l'excès en $p_r$ et en $p_m$, nous avons : \[0\underset{\ref{proprietes_c}.\ref{propriete_classe_bezout}}{\leq} n-\lambda_0-\lambda_r\underset{\ref{proprietes_c}.\ref{propriete_classe_exces}}{\leq} n-\lambda_0-\lambda_{m}\underset{\ref{proprietes_c}.\ref{propriete_classe_exces}}{\leq} n-\lambda_0-\lambda_{m_1}-\lambda_{m_2}\leq n-\lambda_0-2\lambda_{m_2}.\]Par conséquent, $m_2\in\text{MIN}$.
    	\end{proof}
    	
    	    	\begin{minipage}[t]{0.51\textwidth}
    	    	\begin{center}
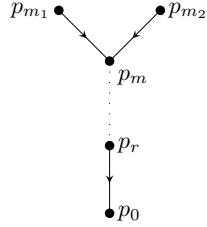

    	    		\scalebox{.9}{
    	    	\begin{tikzpicture}
    	    	    			\coordinate (p0) at (0,0);
    	    	    			\coordinate (pr) at (0,1);
    	    	    			\coordinate (pm) at (0,2.25);
    	    	    			\coordinate (pm1) at (-0.75,3);
    	    	    			\coordinate (pm2) at (0.75,3);
    	    	    			\fill (p0) circle (2pt);
    	    	    			\fill (pr) circle (2pt);
    	    	    			\fill (pm) circle (2pt);
    	    	    			\fill (pm1) circle (2pt);
    	    	    			\fill (pm2) circle (2pt);
    	    	    			\node at (p0) [right] {$p_0$};
    	    	    			\node at (pr) [right] {$p_r$};
    	    	    			\node at (pm) [below right] {$p_m$};
    	    	    			\node at (pm1) [left] {$p_{m_1}$};
    	    	    			\node at (pm2) [right] {$p_{m_2}$};
    	    	    			\draw [loosely dotted](pm) to (pr);
    	    	    			\draw [directed](pr) to (p0);
    	    	    			\draw [directed](pm2) to (pm);
    	    	    			\draw [directed](pm1) to (pm);
    	    	    			\end{tikzpicture}}
    	    	    			\captionof{figure}{Les points $p_{m_1}$ et $p_{m_2}$ sont adhérents à $p_m$. \label{figure_fait_max_1ptindice_ADH_VNA}}
    	    	    			    	\end{center}
    	    	\end{minipage}
    	    	\begin{minipage}[t]{0.48\textwidth}
    	    	\begin{center}
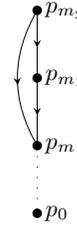

    	    		\scalebox{.9}{	\begin{tikzpicture}
    	    	    	    			\coordinate (p0) at (0,0);
    	    	    	    			\coordinate (pm) at (0,1);
    	    	    	    			\coordinate (pm1) at (0,2);
    	    	    	    			\coordinate (pm2) at (0,3);
    	    	    	    			\fill (p0) circle (2pt);
    	    	    	    			\fill (pm) circle (2pt);
    	    	    	    			\fill (pm1) circle (2pt);
    	    	    	    			\fill (pm2) circle (2pt);
    	    	    	    			\node at (p0) [right] {$p_0$};
    	    	    	    			\node at (pm) [right] {$p_m$};
    	    	    	    			\node at (pm1) [right] {$p_{m_1}$};
    	    	    	    			\node at (pm2) [right] {$p_{m_2}$};
    	    	    	    			\draw [loosely dotted](pm) to (p0);
    	    	    	    			\draw [directed](pm1) to (pm);
    	    	    	    			\draw [directed](pm2) to (pm1);
    	    	    	    			\draw [directed](pm2) to[bend right=30] (pm);
    	    	    	    			\end{tikzpicture}}
    	    	    	    			\captionof{figure}{Le point $p_{m_2}$ est un point satellite. \label{figure_fait_adh_vna_min}}
    	    	    	    			    	\end{center}    	    	\end{minipage}

    	Remarquons que le fait \ref{fait_max_1ptindice_ADH_VNA} implique que les points indicés par VNA sont tous libres. En effet, si un point $p_{m_2}$ est satellite cela signifie que $p_{m_2}$ est adhérent à deux points $p_{m_1}$ et $p_{m}$ et que $p_{m_1}$ est également adhérent à $p_{m}$ (voir Figure \ref{figure_fait_adh_vna_min}). Par le fait \ref{fait_max_1ptindice_ADH_VNA} et par positivité des excès en $p_{m_1}$ nous obtenons que $m_2\in\text{MIN}$.
    	Par contre, les points indicés par ADH peuvent être libres ou satellites.

    	\begin{rmq}\label{rmq_point_satellite_libre_ADH}
    		Un point libre indicé par ADH est un point infiniment proche du premier ordre de $p_0$. De plus, tout point satellite indicé par ADH est voisin d'un point libre indicé par ADH. Par exemple sur la figure \ref{figure_rmq_pt_libre_satellite_adh}, les points $p_{i_1}$, $p_{i_2}$, $p_{i_3}$ et $p_{i_4}$ sont adhérents au point $p_0$. Les points $p_{i_2}$ et $p_{i_3}$ sont des points satellites et sont voisins du point $p_{i_1}$ qui est libre. Le point $p_{i_4}$ est lui aussi libre.
    	\end{rmq}
    	
   	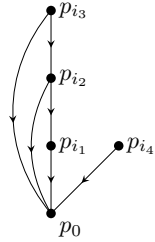
\begin{figure}[h]
   	    		\begin{center}
   	    		\scalebox{0.9}{
   	    			\begin{tikzpicture}
   	    			\coordinate (p0) at (0,0);
   	    			\coordinate (pi1) at (0,1);
   	    			\coordinate (pi2) at (0,2);
   	    			\coordinate (pi3) at (0,3);
   	    			\coordinate (pi4) at (1,1);
   	    			\fill (p0) circle (2pt);
   	    			\fill (pi1) circle (2pt);
   	    			\fill (pi2) circle (2pt);
   	    			\fill (pi3) circle (2pt);
   	    			\fill (pi4) circle (2pt);
   	    			\node at (p0) [below right] {$p_0$};
   	    			\node at (pi1) [right] {$p_{i_1}$};
   	    			\node at (pi2) [right] {$p_{i_2}$};
   	    			\node at (pi3) [right] {$p_{i_3}$};
   	    			\node at (pi4) [right] {$p_{i_4}$};
   	    			\draw [directed](pi1) to (p0);
   	    			\draw [directed](pi4) to (p0);
   	    			\draw [directed](pi2) to (pi1);
   	    			\draw [directed](pi3) to (pi2);
   	    			\draw [directed](pi2) to[bend right=30] (p0);
   	    			\draw [directed](pi3) to[bend right=40] (p0);
   	    			\end{tikzpicture}}
   	    			\caption{Une configuration de points adhérents à $p_0$\label{figure_rmq_pt_libre_satellite_adh}}
   	    		\end{center}
   	    	\end{figure}

    	\begin{fait}\label{fait_support_Jonq_pt_min}
    		Supposons qu'il existe un sous-ensemble $\{p_0,p_{j_1},\dots,p_{j_{2\delta}}\}$ de l'ensemble des points-base de $f$ qui est le support d'une application de Jonquières ayant pour point-base maximal $p_0$. Notons $j_{\text{min}}\in\{j_1,\dots,j_{2\delta}\}$ un indice tel que $\lambda_{j_{\text{min}}}= \min\{\lambda_i\mid\ i\in\{j_1,\dots,j_{2\delta}\}\}$. Alors pour tout $i\geq j_{\text{min}}$, l'indice $i$ est dans l'ensemble MIN.
    	\end{fait}
    	\begin{proof}
    		Soit $\j$ une application de caractéristique Jonquières ayant cet ensemble de points comme points-base. Notons qu'elle est de degré $\delta+1$. Nous avons alors par hypothèse sur $c$ : \[0\underset{\eqref{eq_condition_sur_c}}{\leq} \isome{\j}(c)\cdot\l-c\cdot\l\underset{\eqref{lemme_calcul_fc_id_c}}{=} \delta n-\delta\lambda_0-\sum\limits_{i=1}^{2\delta}\lambda_{j_i}\leq \delta(n-\lambda_0-2\lambda_{j_{\text{min}}}).\] Par conséquent l'indice $j_{\text{min}}$ est dans MIN, et par décroissance des $\lambda_i$, c'est le cas de tous les indices plus grands que $j_{\text{min}}$. 
    	\end{proof}
    	
    	\begin{fait}\label{fait_card_VNA_inferieur_ADH}
    		Parmi les points-base de $f$ voisins de tout point libre indicé par ADH il y a au moins autant de points indicés par ADH que par VNA. En particulier, le cardinal de VNA est strictement inférieur à celui de ADH.
    	\end{fait}
    	
      	\begin{proof}
    		Soit ${a_0}\in\text{ADH}$ tel que $p_{a_0}$ est un point libre. Notons $a_1\leq\dots\leq a_k$ les indices de ADH et $v_1\leq\dots\leq v_{r}$ les indices de VNA tels que les points $p_{a_1},\dots, p_{a_k},p_{v_1},\dots,p_{v_{r}}$ soient les voisins de $p_{a_0}$. Raisonnons par l'absurde et supposons que $k$ soit strictement plus petit que $r$ : $k<r$. Par le fait \ref{fait_max_1ptindice_ADH_VNA}, ces points forment une tour où chaque point est adhérent au précédent mais n'est adhérent à aucun autre point de la tour, excepté de $p_0$ lorsqu'ils sont indicés par ADH (voir Figure \ref{figure_fait_card_VNA_inferieur_ADH}). Considérons l'ensemble pré-consistant $\{p_0,p_{a_0},p_{a_1},\dots,p_{a_k},p_{v_1},\dots, p_{v_{k+1}}\}$. D'après le lemme \ref{lemme_support_Jonquieres}, il existe une application de Jonquières de degré $k+2$ ayant cet ensemble de points comme points-base et $p_0$ comme point-base maximal. Par le fait \ref{fait_support_Jonq_pt_min}, $v_{k+1},\dots,v_{r}\in\text{MIN}$ ce qui est la contradiction attendue.	
    	\end{proof} 
    	
   	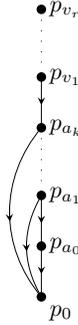
\begin{figure}[h]
  		\begin{center}
  		\scalebox{0.9}{
  			\begin{tikzpicture}
  			\coordinate (p0) at (0,0);
  			\coordinate (pa0) at (0,0.75);
  			\coordinate (pa1) at (0,1.5);
  			\coordinate (pak) at (0,2.5);
  			\coordinate (pv1) at (0,3.25);
  			\coordinate (pvr) at (0,4.25);
  			\fill (p0) circle (2pt);
  			\fill (pa0) circle (2pt);
  			\fill (pa1) circle (2pt);
  			\fill (pak) circle (2pt);
  			\fill (pv1) circle (2pt);
  			\fill (pvr) circle (2pt);
  			\node at (p0) [below right] {$p_0$};
  			\node at (pa0) [right] {$p_{a_0}$};
  			\node at (pa1) [right] {$p_{a_1}$};
  			\node at (pak) [right] {$p_{a_k}$};
  			\node at (pv1) [right] {$p_{v_1}$};
  			\node at (pvr) [right] {$p_{v_r}$};
   			\draw [loosely dotted](pak) to (pa1);
   			\draw [loosely dotted](pvr) to (pv1);  			
  			\draw [directed](pa0) to (p0);
  			\draw [directed](pa1) to (pa0);
  			\draw [directed](pv1) to (pak);
  			\draw [directed](pa1) to[bend right=30] (p0);
  			\draw [directed](pak) to[bend right=40] (p0);
  			\end{tikzpicture}}
  			\caption{Tour de points au-dessus de $p_0$\label{figure_fait_card_VNA_inferieur_ADH}}
  		\end{center}
   	\end{figure}

    	S'il existe un point-base majeur de $f$ indicé par ADH alors nous introduisons une dernière catégorie d'indice non disjointe des précédentes appelée JONQ. Construisons cet ensemble. \begin{itemize}[wide]\phantomsection\label{famille_Jonq}
    		\item tous les indices de ADH correspondant à des points majeurs pour $f$ sont dans JONQ.
    		\item Ensuite nous ajoutons un à un les indices des points majeurs de $f$, dans l'ordre décroissant des $\lambda_i$, de façon à ce que l'ensemble reste consistant à chaque étape.
    		\item Nous nous arrêtons lorsqu'il y a autant de points adhérents à $p_0$ que de points non adhérents à $p_0$. Cela arrive forcément d'après le lemme \ref{lemme_nbr_pt_adh_majeur}.  
    	\end{itemize}
    	
    	\begin{rmq}\label{rmq_JONQ_support_jonquieres}
    		D'après la proposition \ref{prop_supp_jonq_points_majeurs}, il existe une application de Jonquières ayant $p_0$ comme point-base maximal et possédant l'ensemble des points indicés par JONQ comme points-base de multiplicité $1$.
    	\end{rmq}
    	
    	\begin{fait}\label{fait_pt_majeur_f_indice_min}
    		Tout point-base majeur de $f$ non indicé par JONQ est indicé par MIN.
    	\end{fait}
    	\begin{proof}
    	Raisonnons par l'absurde et supposons qu'il existe $p_s$ un point-base majeur de $f$ tel que $s\notin\text{JONQ}\cup\text{MIN}$. Comme les indices de ADH sont tous dans JONQ, la seule possibilité est que $s\in\text{VNA}$. Considérons un indice $r$ de JONQ tel que pour tout $j\in\text{JONQ}$, $\lambda_r\leq \lambda_j$. Par la remarque \ref{rmq_JONQ_support_jonquieres} et le fait \ref{fait_support_Jonq_pt_min}, l'indice $r$ est dans MIN et non pas dans ADH. Par conséquent, c'est l'un des derniers points ajoutés et comme nous rangeons les indices dans JONQ dans l'ordre décroissant des $\lambda_i$ cela signifie que $\lambda_s\leq \lambda_r$. Par conséquent, $s$ appartient à MIN, ce qui nous donne la contradiction attendue.
    	\end{proof}
    	Soit $2\delta$ le cardinal de l'ensemble JONQ.
    	Nous renommons par $j_1,\dots, j_{2\delta}$ l'ensemble des indices de JONQ de sorte que $m_{j_1}\geq\dots\geq m_{j_{2\delta}}$ et que pour tout $1\leq i\leq 2\delta$, l'ensemble de points $\{p_{j_1},\dots,p_{j_i}\}$ muni de leur multiplicité respective est consistant. S'il n'existe pas de point-base majeur de $f$ indicé par ADH alors JONQ est un ensemble vide et l'entier $\delta$ ainsi que les multiplicités $m_{j_i}$ seront considérés comme nuls dans les calculs qui suivent.
    	
    	Commençons à réorganiser l'équation \eqref{equation_somme_terme_n_lambda}. Par le fait \ref{fait_adh_vna_min}, il est possible de répartir les termes de la façon suivante : 
    	\begin{equation}\label{equation_repartission}
    	\begin{split}
    	(d-1)n-\sum\limits_{p_i\in\supp(f)}m_i\lambda_i&=m_{j_{2\delta}}(\delta n-\delta\lambda_0-\lambda_{j_1}-\dots-\lambda_{j_{2\delta}})\\
    	&+(m_0-\delta m_{j_{2\delta}})(n-\lambda_0)\\
    	&+(d-1-m_0)n\\
    	&-\sum\limits_{i=1}^{2\delta-1}(m_{j_i}-m_{j_{2\delta}})\lambda_{j_i}\\
    	&-\sum\limits_{i\in(\text{VNA}\cup\text{ADH})\setminus \text{JONQ}}m_i\lambda_i\\
    	&-\sum\limits_{i\in\text{MIN}\setminus \text{JONQ}}m_i\lambda_i.
    	\end{split}
    	\end{equation}	
    	Lorsque nous sommes dans le cas des trois premières lignes, les termes en les $m_i$ sont appelés \emph{les poids} des \emph{termes en les multiplicités $\lambda_i$}.	
    	Notre but, à présent, est de répartir les $\lambda_i$ des sommes des trois dernières lignes afin que tous les termes soient positifs.
    	
    	Posons pour tout point-base $p_i$ de $f$ : \[\tilde{m}_{i}=\begin{cases}
    	    	m_{i}-m_{j_{2\delta}}  &\text{ si }  {i}\in\text{JONQ}\\
    	    	m_{i} & \text{ sinon }
    	    	\end{cases}.\]

     	Considérons les indices $i$ de ADH dont la multiplicité $\tilde{m}_i$ est non nulle. Notons-les $a_1,\dots,a_k$ de sorte que $\tilde{m}_{a_1}\geq\dots\geq \tilde{m}_{a_k}$ .
    	Faisons de même pour les indices de VNA et notons les $v_1,\dots,v_r$ de sorte que $\tilde{m}_{v_1}\geq\dots\geq \tilde{m}_{v_{r}}$. Malgré le fait \ref{fait_card_VNA_inferieur_ADH}, $k$ n'est pas forcément strictement supérieur à $r$. En effet il se pourrait qu'une situation analogue à l'exemple suivant arrive. 

	\begin{figure}[h]
    	    	    		\begin{center}
    	    	    		\scalebox{0.8}{
    	    	    			\begin{tikzpicture}
    	    	    			\coordinate (p0) at (0,0);
    	    	    			\coordinate (p1) at (0,1);
    	    	    			\coordinate (p2) at (0,2);
    	    	    			\coordinate (p3) at (0,3);
    	    	    			\fill (p0) circle (2pt);
    	    	    			\fill (p1) circle (2pt);
    	    	    			\fill (p2) circle (2pt);
    	    	    			\fill (p3) circle (2pt);
    	    	    			\node at (p0) [right] {$p_0$};
    	    	    			\node at (p1) [right] {$p_1$};
    	    	    			\node at (p2) [left] {$p_2$};
    	    	    			\node at (p3) [right] {$p_3$};
    	    	    			\draw [directed](p1) to (p0);
    	    	    			\draw [directed](p2) to (p1);
    	    	    			\draw [directed](p3) to (p2);
    	    	    			\draw[directed] (p2) to[bend right=30] (p0);
    	    	    			\end{tikzpicture}}
    	    	    		\end{center}
    	    	    		\caption{\label{figure_exemple_tilde}}
    	    	    	\end{figure}

    	Considérons la tour de points $p_0,p_1,p_2,p_3$ incluse dans l'ensemble des points-base de $f$ où chacun est adhérent au précédent, où $p_2$ est également adhérent à $p_0$ (voir Figure \ref{figure_exemple_tilde}) et où $1,2\in\text{ADH}$ et $3\in\text{VNA}$. Si nous supposons de plus que $1,2\in\text{JONQ}$, $3\notin\text{JONQ}$, que $m_{1}=m_{2}=m_{\j_{2\delta}}$ et que les autres points-base de $f$ sont dans $\P^2$ alors $k=0$ alors que $r=1$. Cette situation arrive si le point $p_3$ est un point-base de $f$ qui n'est pas majeur pour $f$ mais qui est majeur pour $c$.
    	
    	\begin{fait}\label{fait_k_plus_petit_l}
    		Si $r\geq k>0$ alors il existe au moins $r-k+1$ couples d'indices, notés $(a_{k+t},v_{i_t})_{1\leq t\leq r-k+1}$ tels que tous les indices sont deux à deux distincts et pour $1\leq t \leq r-k+1$ : \begin{itemize}[itemsep=0pt]
    			\item $a_{k+t}\in \text{ADH}\setminus \{a_1,\dots,a_k\}$,
    			\item $v_{i_t}\in  \{v_i\}_{1\leq i\leq r}$,
    			\item $p_{v_{i_t}}$ est voisin de $p_{a_{k+t}}$.
    		\end{itemize}
    	\end{fait}
    	Remarquons que la condition $a_{k+t}\in \text{ADH}\setminus \{a_1,\dots,a_k\}$ est équivalente à $a_{k+t}\in\text{ADH}\cap\text{JONQ}$ et $\tilde{m}_{a_{k+t}}=0$.
    	\begin{rmq}
    		L'hypothèse $k>0$ est indispensable car sinon nous pourrions avoir la situation précédent le fait \ref{fait_k_plus_petit_l}. Dans ce cas-là, $k=0$ et $r=1$. Il y a effectivement deux points indicés par ADH qui ont une multiplicité $\tilde{m}_i$ nulle mais il n'existe qu'un point indicé par VNA et voisin des points $p_{1}$ et $p_{2}$. \end{rmq}
    	\begin{proof}
    		Notons $p_{i_1},\dots,p_{i_N}$ les points libres indicés par ADH. Tous les points indicés par $\text{ADH}\cup\text{VNA}$ sont voisins ou égaux à ces points-là d'après la remarque \ref{rmq_point_satellite_libre_ADH}. Pour tout $j\in\{1,\dots,N\}$, considérons le point $p_{i_j}$ ainsi que ses points voisins indicés par $\text{ADH}\cup\text{VNA}$. D'après le fait \ref{fait_max_1ptindice_ADH_VNA}, ils forment une tour où chacun est adhérent uniquement au point précédent, sauf les points indicés par ADH qui sont également adhérents au point $p_0$. Notons $k_{i_j}$ et $r_{i_j}$ le cardinal des indices des points voisins ou égaux à $p_{i_j}$ dont la multiplicité $\tilde{m}_{i}$ est non nulle et qui sont respectivement indicés par ADH et VNA.
    		
    		Si $k_{i_j}\leq r_{i_j}$ alors d'après le fait \ref{fait_card_VNA_inferieur_ADH}, cela implique qu'il existe au moins $r_{i_j}-k_{i_j}+1$ points distincts de cette tour indicés par ADH qui ont une multiplicité $\tilde{m}_{i}$ nulle. 
    		Dans le cas où $k_{i_j}$ est non nul, $r_{i_j}-k_{i_j}+1\leq r_{i_j}$ donc il existe au moins $r_{i_j}-k_{i_j}+1$ points distincts de cette tour indicés par VNA qui ont une multiplicité $\tilde{m}_i$ non nulle. Si $k_{i_j}=0$ alors il y en a $r_{i_j}=r_{i_j}-k_{i_j}$. Par conséquent dans les deux cas il existe $r_{i_j}-k_{i_j}$ points distincts de cette tour indicés par VNA qui ont une multiplicité $\tilde{m_i}$ non nulle et autant indicés par ADH qui ont une multiplicité $\tilde{m}_i$ nulle. Nous pouvons en ajouter un de plus dans chaque famille lorsque $k_{i_j}$ est non nul. 		
    		
    		Notons $J$ l'ensemble des indices $j\in\{1,\dots,N\}$ tels que $k_{i_j}\leq r_{i_j}$. Comme nous avons supposé $k$ non nul cela signifie qu'il existe au moins un élément $s\in\{1,\dots,N\}$ tel que $k_{i_s}$ est non nul. S'il y en a plusieurs, nous en choisissons un. Posons : 
    		\[C_s=\begin{cases}
    		r_{i_s}-k_{i_s}+1+\sum\limits_{\substack{j\in J\\j\neq s}}(r_{i_j}-k_{i_j}) & \text{ si } s\in J\\
    		\sum\limits_{j\in J}(r_{i_j}-k_{i_j}) & \text{ sinon }
    		\end{cases}.\]
    		Nous avons montré que nous pouvons choisir au moins $C_s$ couples ayant les propriétés de l'énoncé.
    		Dans les deux cas de la définition de $C_s$, comme $r_{i_j}-k_{i_j}<0$ pour $j\notin J$, nous avons 
    		\[r-k = \sum\limits_{j\in J}(r_{i_j}-k_{i_j}) + \sum\limits_{j\notin J}(r_{i_j}-k_{i_j})<C_s,\]
    		par conséquent nous pouvons choisir $r-k+1$ couples comme dans l'énoncé.
    	\end{proof}
    	
    	\begin{fait}\label{fait_Jonq_devenu_mineur}
    		Nous avons $\tilde{m}_{a_1}\leq\frac{d-m_0}{2}$ et $\tilde{m}_{v_1}\leq\frac{d-m_0}{2}$.
    	\end{fait}
    	\begin{proof}
    		Montrons le résultat pour $\tilde{m}_{a_1}$, le résultat s'obtient de la même manière pour $\tilde{m}_{v_1}$. 
    		Si le point $p_{a_1}$ n'est pas un point-base majeur pour $f$, le résultat est immédiat puisque $\tilde{m}_{a_1}=m_{a_1}$. 
    		
    		Si le point $p_{a_1}$ est un point point-base majeur de $f$ alors $a_1\in\text{JONQ}$ (définition page \pageref{famille_Jonq}) sinon nous aurions une contradiction avec le fait \ref{fait_pt_majeur_f_indice_min} qui impliquerait que $a_1\in\text{MIN}$ or les ensembles ADH et MIN sont disjoints (définitions page \pageref{famille_poins}). Comme l'ensemble JONQ est pré-consistant, il existe un indice $i$ dans $\text{JONQ}\cap\text{ADH}$, possiblement égal à $a_1$ tel que $p_{a_1}$ est voisin de $p_i$ et tel que le point $p_i$ est libre. 
    		Raisonnons par l'absurde et supposons que $m_{a_1}-m_{j_{2\delta}}=\tilde{m}_{a_1}>\frac{d-m_0}{2}$. Par positivité de l'excès au point $p_i$ pour $f$ (Proposition \ref{prop_consistence}), et comme $p_{j_{2\delta}}$ est un point-base majeur pour $f$, nous avons : \[m_{i}\geq m_{a_1}>\frac{d-m_0}{2}+m_{j_{2\delta}}>d-m_0.\] Nous obtenons une contradiction avec le théorème de Bézout (Proposition \ref{prop_Bezout}) en considérant la droite passant par les points $p_0$ et $p_{i}$, puisque le point $p_{i}$ est, par hypothèse, libre et adhérent à $p_0$.
    	\end{proof}
    Le fait \ref{fait_Jonq_devenu_mineur} implique que $d-m_0-\tilde{m}_{a_1}-\tilde{m}_{v_1}\geq 0$.
    	Nous pouvons donc réarranger les termes en considérant deux cas, suivant si $d-m_0-\tilde{m}_{a_1}-\tilde{m}_{v_1}$ est nul ou pas.
    	\begin{enumerate}[wide, label=\roman*)]
    		\item\label{item_cas_moins_difficile} Si $d-m_0-\tilde{m}_{a_1}-\tilde{m}_{v_1}\geq 1$ alors nous répartissons les termes de \eqref{equation_repartission} de la façon suivante :
    		\begin{align}
    		(d-1)n-\sum\limits_{p_i\in\supp(f)}m_i\lambda_i&=m_{j_{2\delta}}(\delta n-\delta\lambda_0-\lambda_{j_1}-\dots-\lambda_{j_{2\delta}})\tag{$\ast_{jonq}$} \label{align_somme_1_jonq}\\
    		&\ +(m_0-\delta m_{j_{2\delta}})(n-\lambda_0)\tag{$\ast_{\lambda_0}$}\label{align_somme_1_lambda_0}\\
    		&\tag{$\ast_{adh}$}\left. \begin{aligned}\label{align_somme_1_adh}
    		&+\tilde{m}_{a_k}(n-\lambda_{a_1}-\dots-\lambda_{a_k})\\
    		&+(\tilde{m}_{a_{k-1}}-\tilde{m}_{a_k})(n-\lambda_{a_1}-\dots-\lambda_{a_{k-1}})\\
    		&\hspace{0.5em}\vdots\\
    		&+(\tilde{m}_{a_1}-\tilde{m}_{a_2})(n-\lambda_{a_1})
    		\end{aligned}\right\}\\
    		&\tag{$\ast_{vna}$}\left. \begin{aligned}\label{align_somme_1_vna}
    		&+\tilde{m}_{v_{r}}(n-\lambda_{v_1}-\dots-\lambda_{v_{r}})\\
    		&+(\tilde{m}_{v_{r-1}}-\tilde{m}_{v_{r}})(n-\lambda_{v_1}-\dots-\lambda_{v_{r-1}})\\
    		&\hspace{0.5em}\vdots\\
    		&+(\tilde{m}_{v_1}-\tilde{m}_{v_2})(n-\lambda_{v_1})
    		\end{aligned}\right\}\\
    		&\ +(d-1-m_0-\tilde{m}_{a_1}-\tilde{m}_{v_1})n\tag{$\ast_{n}$}\label{align_somme_1_n}\\
    		&\ -\sum_{i\in\text{MIN}}\tilde{m}_i\lambda_i.\nonumber
    		\end{align}
    		Remarquons que les sommes des expressions des lignes \eqref{align_somme_1_adh} et \eqref{align_somme_1_vna} sont respectivement égales à 
    		\[\tilde{m}_{a_1}n-\sum\limits_{i=1}^{k}\tilde{m}_{a_i}\lambda_{a_i}\  \text{ et } \  		\tilde{m}_{v_1}n-\sum\limits_{i=1}^{r}\tilde{m}_{v_i}\lambda_{v_i}.\]
    		Si $r=0$ la somme s'écrit de la même façon sans les termes de \eqref{align_somme_1_vna}. De même, si $k=0$ il n'y a plus les termes de \eqref{align_somme_1_adh}. Montrons que les termes de chaque ligne exceptée la dernière, sont positifs. Ensuite, nous répartirons les $\lambda_i$ où $i\in\text{MIN}$ de sorte que chaque terme reste positif.
    		\begin{itemize}[wide, label=$\ast$]
    			\item\textbf{Les poids sont positifs.} Les points $p_{j_1},\dots,p_{j_{\delta}}$ sont adhérents à $p_0$, par conséquent, par décroissance des multiplicités $m_{j_i}$ et par positivité des excès en $p_0$ pour $f$, nous avons : \[m_0 \underset{\ref{prop_consistence}}{\geq}\sum\limits_{i=1}^{\delta}m_{j_i}\geq  \delta m_{j_{2\delta}}.\] Ainsi, le poids $m_0-\delta m_{j_{2\delta}}$ est positif. Les autres poids le sont par hypothèse de décroissance des multiplicités et par hypothèse du cas \ref{item_cas_moins_difficile}.
    			\item\textbf{Les termes en $\lambda_i$ sont positifs.} Le terme en $\lambda_i$ de \eqref{align_somme_1_jonq} est positif par l'hypothèse faite sur $c$ (équation \eqref{eq_condition_sur_c}) et par la remarque \ref{rmq_JONQ_support_jonquieres} et l'équation \eqref{lemme_calcul_fc_id_c}.
    			Considérons les termes de \eqref{align_somme_1_adh}. En utilisant la positivité des excès en $p_0$ pour $c$, nous avons pour $1\leq i\leq k$ : 
    			\begin{equation}\label{eq_terme_adh}
    			n-\lambda_{a_1}-\lambda_{a_2}-\dots-\lambda_{a_i}\underset{\ref{proprietes_c}.\ref{propriete_classe_exces}}{\geq} n-\lambda_0\underset{\ref{propriete_classe_carre}}{>}0.
    			\end{equation}
    			Par le fait \ref{fait_card_VNA_inferieur_ADH} et par positivité des excès pour $c$ en chaque point adhérent (Proposition \ref{proprietes_c}.\ref{propriete_classe_exces}), nous avons : \[\sum\limits_{i\in\text{VNA}}\lambda_i\leq \sum\limits_{i\in\text{ADH}}\lambda_i.\]
    			Ainsi les termes de \eqref{align_somme_1_vna} sont strictement positifs, pour $1\leq i\leq r$ :
    			\begin{equation}\label{eq_terme_vna}
    			n-\lambda_{v_1}-\lambda_{v_2}-\dots-\lambda_{v_{i}}\geq n-\sum\limits_{i\in\text{VNA}}\lambda_i\geq n-\sum\limits_{i\in\text{ADH}}\lambda_i\geq n-\lambda_0\underset{\ref{propriete_classe_carre}}{>}0.
    			\end{equation}
    			\item\textbf{Répartissons les $\lambda_i$ où $i\in\text{MIN}$}. Nous voulons répartir pour tout $i\in\text{MIN}$ les $m_i$ multiplicités $\lambda_i$, de sorte que les termes soient comme dans le lemme \ref{lemme_reformulation_identite}, de la forme $n-\lambda_i-\lambda_j-\lambda_k$. Cependant, si $k\geq 4$ ou $r\geq 4$, certains termes \eqref{align_somme_1_adh} et \eqref{align_somme_1_vna} sont \emph{surchargés}, c'est-à-dire qu'ils contiennent plus de trois multiplicités $\lambda_i$, par conséquent d'autres termes devront être sous-chargés. Considérons le multi-ensemble $M$ constitué des $\lambda_i$ où $i\in\text{MIN}$ comptés avec multiplicités $m_i$. Dans chacun des $d-1-m_0-\tilde{m}_{a_1}-\tilde{m}_{v_1}$ termes de \eqref{align_somme_1_n} nous répartissons trois multiplicités de $M$. S'il reste des éléments dans $M$, nous mettons dans chacun des $m_0-\delta m_{j_{2\delta}}$ termes de \eqref{align_somme_1_lambda_0} deux éléments de $M$. S'il reste encore des éléments dans $M$, nous répartissons deux multiplicités de $M$ dans chacun des $\tilde{m}_{v_1}-\tilde{m}_{v_2}$ (respectivement $\tilde{m}_{a_1}-\tilde{m}_{a_2}$) derniers termes de \eqref{align_somme_1_vna} (respectivement \eqref{align_somme_1_adh}) et une multiplicité de $M$ dans chacun des $\tilde{m}_{v_2}-\tilde{m}_{v_3}$ (respectivement $\tilde{m}_{a_2}-\tilde{m}_{a_3}$) avant-derniers termes de  \eqref{align_somme_1_vna} (respectivement \eqref{align_somme_1_adh}). Ce procédé s'arrête dès que tous les éléments de $M$ ont été répartis. Les termes que nous venons de compléter sont positifs par la définition de l'ensemble MIN.
    		\end{itemize}   		
    		Nous venons de montrer que si $c$ satisfait la condition \eqref{eq_condition_sur_c} alors pour toute application $f$ qui a pour point-base maximal $p_0$ et dont les points-base satisfont \ref{item_cas_moins_difficile}, $\isome{f}(c)\cdot \l\geq c\cdot \l$.

    		\item\label{item_cas_difficile} Plaçons nous dans le cas où $d-m_0-\tilde{m}_{a_1}-\tilde{m}_{v_1}=0$.
    		 Le fait \ref{fait_Jonq_devenu_mineur} implique que \[\tilde{m}_{a_1}=\tilde{m}_{v_1}=\frac{d-m_0}{2}.\] En particulier, $k$ et $r$ sont strictement positifs. Nous ne pouvons pas répartir les termes de la somme comme dans le cas \ref{item_cas_moins_difficile} puisque le poids de \eqref{align_somme_1_n} est à présent négatif : \[-1-m_0-\tilde{m}_{a_1}-\tilde{m}_{v_1}=-1.\] 
    		 
			Si $r\geq k$, nous reprenons le résultat et les notations du fait \ref{fait_k_plus_petit_l}. 
			Posons $J=\{i_t\}_{1\leq t\leq r-k+1}$ l'ensemble des indices apparaissant dans le fait \ref{fait_k_plus_petit_l}. 
			Notons $J^c$ son complémentaire. 
			Les ensembles $J$ et $J^c$ forment une partition de l'ensemble des indices $\{1,\dots,r\}$. Dans le cas où $J^c$ est non vide, nous notons $t_0$ un de ses éléments. 
			S'il est vide les termes de \eqref{align_somme_2_v} n'existent pas et $\lambda_{v_{t_0}}=0$. Si $r<k$ alors $J$ est vide et $J^c$ contient les $r$ indices $\{1,\dots,r\}$.
    		
    		Par rapport à la répartition des termes dans la somme du cas \ref{item_cas_moins_difficile}, nous voulons répartir une fois les multiplicités $\lambda_{v_i}$ dans d'autres termes afin que le poids de \eqref{align_somme_2_n} soit positif. Pour cela, nous surchargeons le terme \eqref{align_somme_2_av} afin de ne pas avoir à compléter les termes \eqref{align_somme_2_v}. Nous arrangeons donc les termes de la somme de la façon suivante : 
    		
    		\begin{align*}
    		(d-1)n-\sum\limits_{p_i\in\supp(f)}m_i\lambda_i &=m_{j_{2\delta}}(\delta n-\delta\lambda_0-\lambda_{j_1}-\dots-\lambda_{j_{2\delta}})\tag{$\ast_{jonq}$}\label{align_somme_2_jonq}\\
    		&\  + n-\lambda_{a_1}-\lambda_{a_2}-\dots-\lambda_{a_k}-\lambda_{v_{t_{0}}}-\sum\limits_{t\in J}\lambda_{v_t} \tag{$\ast_{av}$}\label{align_somme_2_av}\\
    		&\tag{$\ast_{adh}$}\label{align_somme_2_adh}\left. \begin{aligned}
    		&+(\tilde{m}_{a_k}-1)(n-\lambda_{a_1}-\dots-\lambda_{a_k})\\
    		&+(\tilde{m}_{a_{k-1}}-\tilde{m}_{a_k})(n-\lambda_{a_1}-\dots-\lambda_{a_{k-1}})\\
    		&\hspace{0.5em}\vdots\\
    		&+(\tilde{m}_{a_1}-\tilde{m}_{a_2})(n-\lambda_{a_1})
    		\end{aligned}\right\}\\	
    		&\ +\sum\limits_{t\in J^c\setminus \{t_0\}}(n-\lambda_0-\lambda_{v_{t}})\tag{$\ast_{v}$}\label{align_somme_2_v}\\	
    		&\tag{$\ast_{vna}$}\label{align_somme_2_vna}\left. \begin{aligned}
    		&+(\tilde{m}_{v_{r}}-1)(n-\lambda_{v_1}-\dots-\lambda_{v_{r}})\\
    		&+(\tilde{m}_{v_{r-1}}-\tilde{m}_{v_{r}})(n-\lambda_{v_1}-\dots-\lambda_{v_{r-1}})\\
    		&\hspace{0.5em}\vdots\\
    		&+(\tilde{m}_{v_1}-\tilde{m}_{v_2})(n-\lambda_{v_1})
    		\end{aligned}\right\}\\
    		&+(m_0-\card(J^c)+1-\delta m_{j_{2\delta}})(n-\lambda_0)\tag{$\ast_{\lambda_0}$}\label{align_somme_2_0}\\
    		&+(d-1-m_0-\tilde{m}_{a_1}-\tilde{m}_{v_1}+1)n\tag{$\ast_n$}\label{align_somme_2_n}\\
    		&-\sum_{i\in\text{MIN}}m_i\lambda_i.
    		\end{align*} 
    	
    		Remarquons que la somme des expressions des lignes \eqref{align_somme_2_av} et \eqref{align_somme_2_adh} est égale à :
    		\[\tilde{m}_{a_1}n-\sum\limits_{i=1}^{k}\tilde{m}_{a_i}\lambda_{a_i}-\lambda_{v_{t_0}}-\sum\limits_{t\in J}\lambda_{v_t},\]
    		celle de la ligne \eqref{align_somme_2_v}, si elle existe, est égale à :
    		\[(\card(J^c)-1)(n-\lambda_0) -\sum\limits_{t\in J^c\setminus \{t_0\}}\lambda_{v_t},\]
    		et celle des lignes \eqref{align_somme_2_vna} est égale à :
    		\[(\tilde{m}_{v_1}-1)n-\sum\limits_{i=1}^{r}(\tilde{m}_{v_i}-1)\lambda_{v_i}.\]
    		Montrons que nous pouvons compléter les termes incomplets à l'aide des $\lambda_i$ où $i\in\text{MIN}$ afin que chaque terme soit positif.
    		\begin{itemize}[wide, label=$\ast$]
    			\item \textbf{Les poids sont positifs.} 
    			Montrons dans un premier temps que le cardinal de $J^c$ est inférieur ou égal à $k-1$ : \[ \card(J^c)=\begin{cases}
    			r\leq k-1  & \text{ si } r<k\\
    			r-(r-k+1)=k-1 &\text{ sinon}
    			\end{cases}.\]
    			En utilisant cela ainsi que la positivité des excès en $p_0$ pour $f$, nous obtenons :
    			\[m_0\underset{\ref{prop_consistence}}{\geq} \sum\limits_{i\in \text{ADH}}m_{i}= \sum\limits_{i=1 }^{k}\tilde{m}_{a_i}+\sum\limits_{i\in\text{JONQ}\cap\text{ADH}}\hspace{-0,5cm}m_{j_{2\delta}}\geq k+\delta m_{j_{2\delta}}>\card(J^c)+\delta m_{j_{2\delta}}.\] Ainsi le poids $m_0-\card(J^c)+1-\delta m_{j_2\delta}$ est strictement positif. Les autres le sont par décroissance des $\tilde{m}_i$ et par l'hypothèse faite dans le cas \ref{item_cas_difficile}.
    			\item \textbf{Les termes en $\lambda_i$ sont positifs.} Les termes en $\lambda_i$ de \eqref{align_somme_2_jonq}, \eqref{align_somme_2_adh} et \eqref{align_somme_2_vna} sont positifs par les mêmes arguments que dans le cas \ref{item_cas_moins_difficile} puisqu'ils n'ont pas changé.
    			S'il existe, le terme de \eqref{align_somme_2_v} est positif d'après l'inégalité de Bézout (Proposition \ref{proprietes_c}.\ref{propriete_classe_bezout}).
    			
    			Il reste à montrer la positivité du terme en $\lambda_i$ de \eqref{align_somme_2_av}. Si $\l\geq k$ alors d'après le fait \ref{fait_k_plus_petit_l}, il existe des points $p_{a_{k+1}},\dots, p_{a_{\l+1}}$ adhérents à $p_0$, distincts des $\{p_{a_i}\}_{1\leq i\leq k}$ et tels que pour $t\in J$, $p_{v_{i_t}}$ est voisin de $p_{a_{k+t}}$. Ainsi par positivité des excès au-dessus de chaque point adhérent : 
    			\[\sum\limits_{t\in J}\lambda_{v_{t}}\leq \sum\limits_{t\in J}\lambda_{a_{k+t}}.\]
    			Notons $p_a$ le point libre adhérent à $p_0$ tel que le point $p_{v_{t_0}}$ est voisin de $p_a$. Nous avons par positivité des excès pour les points $p_0$ et $p_a$, et par l'inégalité de Bézout : 
    			\begin{align}
    			\begin{split}\label{eq_av}
    			n-\lambda_{a_1}-\dots-\lambda_{a_k}-\lambda_{v_{t_0}}-\sum\limits_{t\in J}\lambda_{v_{t}} &\geq n-\sum\limits_{i=1}^{k}\lambda_{a_k}-\sum\limits_{t\in J}\lambda_{a_{{k+t}}} -\lambda_{v_{t_0}}\\
    			& \underset{\ref{proprietes_c}.\ref{propriete_classe_exces}}{\geq} n-\lambda_0 -\lambda_{v_{t_0}}\\
    			& \underset{\ref{proprietes_c}.\ref{propriete_classe_exces}}{\geq} n-\lambda_0 -\lambda_{a}\\
    			&\underset{\ref{proprietes_c}.\ref{propriete_classe_bezout}}{\geq}0,
    			\end{split}
    			\end{align}
    			et donc le terme en $\lambda_i$ de \eqref{align_somme_2_av} est positif.  Si $r<k$ alors $J$ est vide et nous obtenons le même résultat en utilisant la même inégalité qu'au-dessus en enlevant la somme sur $J$.
    			\item Montrons que nous pouvons \textbf{laisser incomplets les termes de \eqref{align_somme_2_v}}, s'ils existent. Si $k=1$ alors $J^c$ est vide et il n'y a par conséquent rien à faire. Sinon le terme de \eqref{align_somme_2_av} possède $k+\card(J)+1-3=k-2+\card(J)$ multiplicités $\lambda_i$ en trop qui compensent les $\card(J^c)-1\leq k-2$ multiplicités manquantes.
    			\item \textbf{Complétons les termes de la somme avec les $\lambda_i$ où $i\in\text{MIN}$.} D'après le point précédent, les termes à compléter sont les deux dernières lignes de \eqref{align_somme_2_adh} et \eqref{align_somme_2_vna}, ainsi que les termes \eqref{align_somme_2_0} et \eqref{align_somme_2_n}. Par la définition de MIN tous les termes complétés sont positifs.
    		\end{itemize}
    		Ainsi la somme est positive comme attendue.
    		
    		Ceci conclut le cas \ref{item_cas_difficile} et donc le point \ref{prop_cellule_identite_appartenir} du théorème \ref{prop_cellule_identite}.
    	\end{enumerate}

    	Montrons maintenant le point \ref{prop_cellule_identite_bord} du théorème \ref{prop_cellule_identite}. Supposons maintenant que $c\in \V(\id)\cap\V(f)$. Nous voulons montrer que $f$ est de caractéristique Jonquières dont le point-base de multiplicité maximale de l'inverse est $p_0$. Par définition : \begin{align*}
    	c\cdot \l&=c\cdot\isome{f}(\l)\\
    	&=\isome{f}^{-1}(c)\cdot \l.
    	\end{align*}
    	Ainsi, d'après le fait \ref{fait_p_0_pas_maximal}, le point $p_0$ doit être un point-base de $f^{-1}$ de multiplicité maximale.
    	De plus, la somme \eqref{equation_somme_terme_n_lambda} appliquée à $f^{-1}$ doit être nulle.  
    	Suivant si nous nous trouvons dans le cas \ref{item_cas_moins_difficile} ou dans le cas \ref{item_cas_difficile}, nous utilisons la répartition des termes faite dans \ref{item_cas_moins_difficile} ou dans \ref{item_cas_difficile}. Comme $c\in\V(\id)$, nous avons déjà montré, dans la première partie de la preuve, que tous les termes de ces sommes sont positifs ou nuls. Nous devons étudier à quelles conditions ils sont en fait tous nuls. 
    	\begin{enumerate}[wide, label=\alph*),itemsep=0pt]
    		\item \label{cas1} Les points-base de $f^{-1}$ ne peuvent pas satisfaire la condition du cas \ref{item_cas_difficile}. Supposons que ce soit le cas. Alors, le terme \eqref{align_somme_2_av} est nul ce qui implique en prenant le cas d'égalité des inégalités \eqref{eq_av} que $\lambda_{v_{t_0}}=n-\lambda_0$. Considérons la suite minimale de points du support de $f^{-1}$, hors $p_0$, qu'il a fallu éclater pour obtenir $p_{v_{t_0}}$ et notons-les $p_{i_1},\dots, p_{i_k}, p_{i_{k+1}}$ où $p_{i_{k+1}}=p_{v_{t_0}}$. Comme $v_{t_0}\in \text{VNA}$, par positivité des excès, les indices $i_1,\dots,i_k$ appartiennent à $\text{VNA}\cup \text{ADH}$. D'après le fait \ref{fait_max_1ptindice_ADH_VNA}, chacun est adhérent uniquement au point précédent sauf les points indicés par ADH qui sont également adhérents à $p_0$. Notons $k'$ le plus petit entier tel que $p_{i_{k'}}$ ne soit pas adhérent à $p_0$. Considérons $k'-2$ points de $\P^2$ (pas forcément dans le support de $c$) tels que deux points de cette famille ne soient pas alignés avec $p_0$. D'après le lemme \ref{lemme_support_Jonquieres} ces points ainsi que les points $p_0, p_{i_1},\dots, p_{i_k'}$ forment le support d'une application de Jonquières $\j$ de degré $k'$. Nous obtenons la contradiction suivante en utilisant la positivité des excès au-dessus des points $\{p_{i_s}\}_{1\leq s\leq k}$ : 
    		\begin{align}\label{contradiction_1}
    		\begin{split}
    			0&\underset{\eqref{eq_condition_sur_c}}{\leq }\isome{\j}(c)\cdot\l-c\cdot \l \\
    		    		&\underset{\eqref{lemme_calcul_fc_id_c}}{\leq} (k'-1)n-(k'-1)\lambda_0-\lambda_{i_1}-\dots-\lambda_{i_k'}\\
    		    		&\underset{\ref{proprietes_c}.\ref{propriete_classe_exces}}{\leq} (k'-1)n-(k'-1)\lambda_0 - k'\lambda_{v_{t_0}}\\
    		    		&= -(n-\lambda_0)\underset{\ref{propriete_classe_carre}}{<}0.
    		\end{split}
    		\end{align}
    		Par conséquent, dans le cas \ref{item_cas_difficile} il n'existe pas de $f$ telle que $c\in\V(\id)\cap\V(f)$.
    		\item\label{cas_2} Les points-base de $f^{-1}$ doivent donc satisfaire les conditions du cas \ref{item_cas_moins_difficile}.
    		Tous les termes de la somme du cas \ref{item_cas_moins_difficile} étant nuls, c'est en particulier le cas du terme  $\tilde{m}_{a_k}(n-\lambda_{a_1}-\lambda_{a_2}-\dots-\lambda_{a_k})$ de \eqref{align_somme_1_adh}. 
    		Par hypothèse sur le choix de l'indice $k$, $\tilde{m}_{a_k}$ est non nul, c'est donc $n-\lambda_{a_1}-\lambda_{a_2}-\dots-\lambda_{a_k}$ qui est nul. L'équation \eqref{eq_terme_adh} implique que $k\leq 2$ puisque pour $k=1$ et $k=2$ les termes ont été complétés à l'aide des $\lambda_i$ où $i\in\text{MIN}$. Mais le fait \ref{fait_positif_0} implique qu'alors $\lambda_i=n-\lambda_0$ ce qui est absurde puisque $i\in\text{MIN}$. Ainsi $k=0$ et il n'y a pas de termes \eqref{align_somme_1_adh}. Nous obtenons de même qu'il n'y a pas de termes \eqref{align_somme_1_vna}. Par le fait \ref{fait_positif_0}, le terme en $\lambda_i$ de \eqref{align_somme_1_n} est strictement positif. Par conséquent, pour que \eqref{align_somme_1_n} soit nul, il faut que son poids qui est égal à $d-1-m_0$ le soit. Ceci implique que $m_0=d-1$. Par conséquent $f^{-1}$ est une application de caractéristique Jonquières et dont le point-base maximal est $p_0$ d'après le point \ref{eq_m_0_jonquieres} du lemme \ref{lemme_equations_evidentes}. De plus, en notant $p_{i_1},\dots,p_{i_{2d-2}}$ les petits points-base de $f^{-1}$, la condition $c\cdot\l=\isome{j^{-1}}\cdot\l$ se réécrit grâce à l'équation \eqref{lemme_calcul_fc_id_c} : 
    		\[n=\lambda_0+\frac{1}{d-1}\sum\limits_{i\in\{i_1,\dots,i_{2d-2}\}}\lambda_i.\]
    		Ceci achève la preuve du théorème \ref{prop_cellule_identite}.\qedhere
    	\end{enumerate} 
    \end{proof}
    
    \begin{ex}\label{ex_classe_speciale2}
    	Il existe des classes spéciales dans la cellule $\V(\id)$. Considérons $p_0\in\P^2$ et deux points libres et adhérents à $p_0$ notés $p_1$ et $p_2$. Considérons la classe \[c=\frac{1}{5}(7\l-4e_{p_0}-2e_{p_1}-2e_{p_2}).\]
    	Montrons qu'elle appartient à $\orbl$. 
    	Elle est d'auto-intersection $1$. Les excès sont positifs. Elle est positive contre l'anti-canonique. Montrons qu'elle satisfait également la condition de Bézout. C'est clair pour les droites. Pour les coniques aussi (à noter que la seule conique passant par les trois points est la réunion de deux droites). Il existe une cubique $C$ passant avec multiplicité $2$ en $p_0$ et avec multiplicité $1$ aux points $p_1$ et $p_2$. Considérons une courbe $D$ de $\P^2$ passant avec multiplicité $\mu_i$ aux points $p_i$. Nous avons alors d'après le théorème de Bézout : \[3d=C\cdot D\geq 2\mu_0+\mu_1+\mu_2.\]
    	Par conséquent, $7d\geq 6d\geq 4\mu_0+2\mu_1+2\mu_2$. Et la condition de Bézout est vérifiée pour toute courbe.
    	Ainsi, la classe $c$ appartient à $\orbl$ (Définition \ref{proprietes_c}). 
    	
    	De plus, les points $p_1$ et $p_2$ sont adhérents à $p_0$ et le degré de $c$ est strictement plus petit que la somme de ses trois plus grandes multiplicités. Par conséquent, $c$ est une classe spéciale. Montrons que $c\in\V(\id)$. 
    	
    	D'après le théorème \ref{prop_cellule_identite}, il suffit de vérifier que pour toute application $\j$ de caractéristique Jonquières dont l'inverse a pour point-base maximal $p_0$ nous avons \[\isome{\j}(c)\cdot \l\geq \l\cdot c  \Leftrightarrow c\cdot \isome{\j}^{-1}(\l) \geq c\cdot \l.\] 
    	
    	Il n'existe pas d'application quadratique dont les points-base ont cette configuration (par positivité des excès). Soit $\q$ une application quadratique dont les points-base de l'inverse sont les points $p_0$, $p_1$ et un troisième noté $q$. 
    	Dans ce cas, nous avons \[c\cdot\isome{\q}(\l)=\frac{1}{5}(14-4-2)=\frac{8}{5}>\frac{7}{5}=c\cdot \l.\]
    	Pour tout $d\geq 3$, il existe une application de Jonquières $\j$ de degré $d$ telle que les points $p_0$, $p_1$ et $p_2$ soient des points-base de $\j^{-1}$ avec $p_0$ point-base maximal. 
    	Dans ces cas, nous avons : 
    	\[c\cdot\isome{\j}(\l)=\frac{1}{5}(7d-4(d-1)-2-2)=\frac{3d}{5}>\frac{7}{5}.\]
    	Par conséquent, $c$ appartient à la cellule identité et n'appartient à aucune autre cellule.
    \end{ex}
    
    L'exemple suivant montre que les germes de cellules contenant une classe spéciale en commun avec la cellule $\V(\id)$ ne sont pas forcément des applications de caractéristique Jonquières dont les points-base sont en configuration spéciale.
    
    \begin{ex}\label{classe_spéciale1}
    	Reprenons la même configuration que dans l'exemple précédent. Considérons deux points $p_1$ et $p_2$ libres et adhérents à un point $p_0\in\P^2$. Posons :
    	\[c=\frac{1}{\sqrt{23}}(7\l-4e_{p_0}-3e_{p_1}-e_{p_2}).\]
    En faisant exactement les mêmes étapes que dans l'exemple précédent, excepté pour l'inégalité de Bézout où nous devons en plus remarquer que $d\geq \mu_1$ nous montrons dans un premier temps que la classe $c$ est dans $\orbl$. Puis en faisant les mêmes calculs d'intersection, nous montrons que $c$ appartient à $\V(\id)$ et aux cellules dont le germe est une application quadratique dont les points-base de l'inverse sont les points $p_0$, $p_1$ et un troisième point. De plus, $c$ n'appartient à aucune autre cellule. 
    \end{ex}
\section{Cellules adjacentes à la cellule associée à l'identité}\label{section_cellule_adjacente}

 Deux cellules de Voronoï sont dites \emph{adjacentes} si leur intersection est non vide. Dans cette partie, nous déterminons les germes des cellules de Voronoï qui sont adjacentes à la cellule $\V(\id)$ (Théorème \ref{thm_principal_germe_cellule_voisine_identite}) et nous caractérisons les classes se trouvant dans de telles intersections.
    
\begin{thm}\label{thm_principal_germe_cellule_voisine_identite}
	Soit $c=n\l-\sum\limits_{i\in I}\lambda_ie_{p_i}\in\orbl$ une classe ordonnée.
\begin{enumerate}
	\item Si $c$ est une classe spéciale appartenant à $\V(\id)\cap\V(f)$ alors $f$ est une application de caractéristique Jonquières dont l'inverse a pour point-base maximal $p_0$. De plus en notant $p_{i_1},\dots,p_{i_{2d-2}}$ les petits points-base de $f^{-1}$, $c$ vérifie : 
	\[n=\lambda_0+\frac{1}{d-1}\sum\limits_{i\in\{i_1,\dots,i_{2d-2}\}}\lambda_i.\]
	\item Dans le cas où $c$ n'est pas une classe spéciale, quitte à permuter l'ordre des points $p_i$ qui ont même multiplicité $\lambda_i$, nous avons $c\in\V(\id)\cap\V(f)$ si et seulement si nous sommes dans l'une des trois situations suivantes.
	\begin{enumerate}[label=\alph*)]
		\item\label{item_quadra} L'application $f$ est quadratique, $p_0$, $p_1$ et $p_2$ sont les points-base de $f^{-1}$ et $n=\lambda_0+\lambda_1+\lambda_2$ où $\lambda_2$ peut éventuellement être nulle.
		\item\label{item_jonq}L'application $f$ est de caractéristique Jonquières non quadratique, $p_0$ est le point-base de multiplicité maximale de $f^{-1}$, les points $\{p_i\}_{1\leq i\leq 2d-2}$ sont ses autres points-base et \[c=n\l -\lambda_0e_{p_0} -\frac{n-\lambda_0}{2}\sum\limits_{i=1}^{2d-2}e_{p_i}-\sum\limits_{i\geq  2d-1}\lambda_ie_{p_i}\] où $\lambda_0\geq\frac{n-\lambda_0}{2}\geq\lambda_i$ pour tout $i\geq 2d-1$.
		\item\label{item_gen}L'application $f$ n'est pas de caractéristique Jonquières et possède $r\in\{6,7,8\}$ points-base en position presque générale, les points $\{p_i\}_{0\leq i \leq r-1}$ sont les points-base de $f^{-1}$ et 
		\[c=n\l -\frac{n}{3}\sum\limits_{i=0}^{r-1}e_{p_i}-\sum\limits_{i\geq  r}\lambda_ie_{p_i}\] où $\frac{n}{3}\geq\lambda_i$ pour tout $i\geq r$.
	\end{enumerate}
\end{enumerate}  
\end{thm}

\begin{rmq}\label{rmq_pt-base_f_f_inverse}
Par le lemme \ref{lemme_config_f_config_inverse}, les points-base de $f^{-1}$ sont en position presque générale si et seulement si il en est de même des points-base de $f$. 
\end{rmq}

\begin{rmq}
Une classe $c$ non-spéciale qui contient seulement deux points dans son support et telle que la somme de ses deux plus grandes multiplicités est égale à son degré, appartient à une infinité de cellules de Voronoï. En effet, $c$ est au bord de $\V(\id)$ et appartient à toutes les cellules $\V(\q)$ où $\q$ est une application quadratique telle que le support de $c$ est inclus dans le support de $\q^{-1}$. Ces cellules sont en nombre infini puisqu'il y a un nombre infini d'applications quadratiques modulo $\PGL(3,\kk)$.
Une classe spéciale peut également appartenir à une infinité de cellules de Voronoï, comme par exemple dans \ref{classe_spéciale1}.  
Dans les autres cas, les classes n'appartiennent qu'à un nombre fini de cellules de Voronoï.

\end{rmq}

\begin{proof}
	Si $c$ est une classe spéciale et $c\in\V(\id)\cap\V(f)$ alors d'après le point \ref{prop_cellule_identite_bord} de la proposition \ref{prop_cellule_identite}, $f^{-1}$ doit être de caractéristique Jonquières et doit avoir $p_0$ comme point-base de multiplicité maximale.
	
	Soit $c=n\l-\sum\limits\lambda_ie_{p_i}\in\orbl$ une classe non spéciale et ordonnée. 
	La classe $c\in\V(f)$ si et seulement si \begin{align*}
	c\cdot \l &=c\cdot \isome{{f}}(\l)\\
	&=\isome{f}^{-1}(c)\cdot \l.
	\end{align*}
	Notons $d$ le degré de $f^{-1}$ et $m_i$ les multiplicités de ses points-base $p_i$. 
	D'après l'équation \eqref{lemme_calcul_fc_id_c} et en reprenant les notations et le résultat du lemme \ref{lemme_reformulation_identite}, il existe un ensemble $T$ de $d-1$ triplets d'indices des points-base de $f^{-1}$ dont chaque triplet est formé d'indices deux à deux disjoints et tel que l'égalité ci-dessus est équivalente à 
	\begin{equation}\label{eq_termes_nuls}
	0=(d-1)n-\sum\limits_{p_i\in\supp(f^{-1})}m_i\lambda_i=\sum\limits_{\{\!\{ i,j,k\}\!\}\in T} n-\lambda_i-\lambda_j-\lambda_k.
	\end{equation}
	Dans la somme de droite les multiplicités $\lambda_i$ apparaissent $m_i$ fois. Comme la classe $c$ n'est pas spéciale, elle appartient à $\V(\id)$ si et seulement si pour tout $i<j<k$ : \begin{equation}\label{eq_4}
	n-\lambda_i-\lambda_j-\lambda_k\geq  n-\lambda_0-\lambda_1-\lambda_2\underset{\ref{rmq_classe_non-spe_V(id)}}{\geq} 0.
	\end{equation} Ainsi $c\in\V(\id)\cap\V(f)$ si et seulement si tous les termes de la somme de droite de \eqref{eq_termes_nuls} sont nuls. Si nous sommes dans l'un des trois cas \ref{item_quadra}, \ref{item_jonq} ou \ref{item_gen} alors $c$ appartient à $\V(\id)\cap\V(f)$. Montrons à présent l'implication et supposons que $c$ appartient à $\V(\id)\cap\V(f)$. Distinguons les cas suivant si $f$ est une application quadratique, si $f$ est une application de caractéristique Jonquières non quadratique et enfin si $f$ n'est pas de caractéristique Jonquières.
	\begin{enumerate}[wide, itemsep=0pt, label=\alph*)]
		\item Si $f$ est une application quadratique dont l'inverse a pour points-base $p_{i_0},p_{i_1},p_{i_2}$ alors la somme \eqref{eq_termes_nuls} s'écrit \[n-\lambda_{i_0}-\lambda_{i_1}-\lambda_{i_2}=0.\]
		Les multiplicités $\lambda_{i_0}$, $\lambda_{i_1}$ et $\lambda_{i_2}$ sont les trois plus grandes multiplicités de $c$ car sinon : 
		\[0=n-\lambda_{i_0}-\lambda_{i_1}-\lambda_{i_2}> n-\lambda_0-\lambda_1-\lambda_2,\] ce qui contredirait l'équation \eqref{eq_4}.
		Ainsi, les points $p_0$, $p_1$ et $p_2$ sont les points-base de $f^{-1}$ et $n=\lambda_0+\lambda_1+\lambda_2$.
		\item Dans le cas où $f$ est de caractéristique Jonquières de degré strictement supérieur à $2$, notons $\lambda_{i_0}$ la multiplicité pour $c$ correspondant au point-base maximal de $f^{-1}$ et $\lambda_{i_1},\dots,\lambda_{i_{2d-2}}$ les multiplicités pour $c$ correspondant aux points-base $p_{i_1},\dots,p_{i_{2d-2}}$ de $f^{-1}$ de sorte que $\lambda_{i_1}\geq\cdots \geq \lambda_{i_{2d-2}}$. 
		La somme \eqref{eq_termes_nuls} se réécrit : 
		\[0= \sum_{j=1}^{d-1}(n-\lambda_{i_0}-\lambda_{i_{2j-1}}-\lambda_{i_{2j}}).\]	
		Comme tous les termes sont positifs ou nuls, ils sont en fait nuls. Ainsi, nous obtenons que $\lambda_{i_1}=\dots=\lambda_{i_{2d-2}}=\frac{n-\lambda_{i_0}}{2}$. De même que précédemment pour ne pas contredire l'équation \eqref{eq_4}, nous devons avoir $\lambda_{i_0}=\lambda_0$ et $\lambda_1=\lambda_{i_1}$ comme attendu.
		\item Considérons enfin le cas où $f$ n'est pas de caractéristique Jonquières. Notons $\lambda_{i_0},\lambda_{i_1},\dots, \lambda_{i_{r-1}}$ les multiplicités de $c$ correspondant aux points-base $p_{i_0},\dots,p_{i_{r-1}}$ de $f^{-1}$ de sorte que $\lambda_{i_0}\geq\lambda_{i_1}\geq \dots \geq \lambda_{i_{r-1}}$. Comme $f$ n'est pas une application de Jonquières elle est de degré au moins $4$ et par le lemme \ref{lemme_nbr_pt_base_degre}.\ref{lemme_nbre_pt_base_degre4} elle possède au moins six points-base. De plus, par les points \ref{eq_mi_pluspetit_d} et \ref{eq_m_0_jonquieres} du lemme \ref{lemme_equations_evidentes}, pour tout $i$, $m_i\leq d-2$. Par conséquent, nous pouvons arranger les termes de la somme \eqref{eq_termes_nuls} de sorte qu'il y en ait un de la forme : $n-\lambda_{i_0}-\lambda_{i_1}-\lambda_{i_2}$ et un autre de la forme : $n-\lambda_{i_{r-3}}-\lambda_{i_{r-2}}-\lambda_{i_{r-1}}$ où $r-3>2$. Comme $c$ appartient à $\V(f)$, ils sont tous les deux nuls :
		\[n=\lambda_{i_0}+\lambda_{i_1}+\lambda_{i_2}=\lambda_{i_{r_2}}+\lambda_{i_{r-1}}+\lambda_{i_r}.\] Par décroissance des $\lambda_{i_j}$ cela signifie qu'ils sont tous égaux à $\frac{n}{3}$. 
		Si maintenant ce ne sont pas les multiplicités maximales pour $c$ cela implique en particulier que $\lambda_0>\frac{n}{3}$ par conséquent : 
		\[ n-\lambda_0-\lambda_1-\lambda_2<  n-3\frac{n}{3}=0.\]Or, ceci contredit l'équation \eqref{eq_4}. Ainsi, nous avons montré que si $c$ appartient à $\V(\id)\cap\V(f)$ alors $c$ est $r'$-symétrique avec $r'\geq r$, $c$ a pour multiplicité maximale $\frac{n}{3}$ et les points-base de $f^{-1}$ sont inclus dans les points du support de $c$ ayant la multiplicité maximale, comme attendu. De plus, d'après le lemme \ref{lemme_symetrique}, une telle classe $r'$-symétrique vérifie $r'\leq 8$, par conséquent $f^{-1}$ ne possède pas plus de $8$ points-base. Si $f^{-1}$ possède quatre points-base alignés ou sept points-base sur une conique cela contredit l'inégalité de Bézout (Définition \ref{proprietes_c}.\ref{propriete_classe_bezout}) et donc cela contredit le fait que $c\in\orbl$. De même si $f^{-1}$ possède deux points-base adhérents à un même troisième point-base nous obtenons une contradiction avec la positivité des excès (Définition \ref{proprietes_c}.\ref{propriete_classe_exces}). La remarque \ref{rmq_pt-base_f_f_inverse} permet de conclure. \qedhere
	\end{enumerate}
\end{proof}
    	
\begin{ex} Voici des exemples de classes appartenant à l'intersection entre $\V(\id)$ et une autre cellule.
\begin{itemize}[wide]
\item La classe :
	\[c_{\j}= \sqrt{\frac{2}{d+1}}\left(\frac{d+1}{2}\l-\frac{d-1}{2}e_{p_0}-\frac{1}{2}\sum\limits_{i=1}^{2d-2}e_{p_i}\right)\] appartient à $\V(\id)\cap\V(\j)$ où $\j$ est une application de caractéristique Jonquières de degré $d$ dont l'inverse a pour point-base maximal $p_0$ et pour autres points-base $\{p_i\}_{1\leq i \leq 2d-2}$.
\item La classe : \[c_f=\frac{3}{\sqrt{9-r}}\l -\frac{1}{\sqrt{9-r}}\sum\limits_{i=0}^{r-1}e_{p_i}\] appartient à $\V(\id)\cap\V(f)$ où $f^{-1}$ une application ayant au plus $8$ points-base $\{p_i\}_{0\leq i\leq r-1}$ qui sont en position presque générale ($r\leq 8$).
\end{itemize}
\end{ex}

Le corollaire \ref{cor_cellule_adjacente} résume tous les germes des cellules adjacentes à $\V(\id)$.

\begin{cor}\label{cor_cellule_adjacente}
	L'ensemble des germes des cellules adjacentes à $\V(\id)$ est constitué de toutes : 
	\begin{enumerate}[wide]
	\item\label{cor_jonq} les applications de caractéristique Jonquières,
	\item\label{cor_f_pas_aligne_conique} les applications ayant au plus $8$ points-base, qui ne sont pas de caractéristiques Jonquières et dont les points-base de l'inverse sont en position presque générale.
	\end{enumerate}
\end{cor}

	\begin{rmq}
	Les conditions du point \ref{cor_f_pas_aligne_conique} du corollaire \ref{cor_cellule_adjacente} sont nécessaires. Soient $p_0,\dots,p_4\in\P^2$ tels que trois points ne soient pas alignés. Ce sont les points-base d'une application de Jonquières $\j$ de degré $3$ et de point-base maximal $p_0$. Soit $L$ la droite passant par les points $p_1$ et $p_2$. Considérons trois points formant le support d'une application quadratique $\q_1$ dont deux sont sur $\j(L)$ mais sont différents des points-base de $\j^{-1}$. Alors l'application $\q_1\circ\j$ est une application ayant huit points-base dont quatre alignés et de caractéristique $(6;4,2^4,1^3)$. 
	
	Soit $C$ une conique passant par les points $p_1,\dots,p_4$. Considérons trois points formant le support d'une application quadratique $\q_2$ qui sont sur $\j(C)$ mais qui ne sont pas des points-base de $\j^{-1}$. L'application $\q_2\circ\j$ est une application ayant huit points-base dont sept sur une conique.
	
	Enfin, si maintenant nous composons l'application de Jonquières de degré trois $(x,y)\mapsto(y^3-x,y)$ par une application quadratique dont les trois points-base sont dans $\P^2$ différents du point $[1:0:0]$, la composée possède huit points-base dont deux adhérents à un même troisième. 
	
	Ces trois applications, bien qu'ayant huit points-base, ne sont pas des germes de cellules adjacentes à l'identité d'après le corollaire \ref{cor_cellule_adjacente}.
	\end{rmq}
	
	Le corollaire suivant est une conséquence du corollaire \ref{cor_cellule_adjacente} et du lemme \ref{lemme_nbr_pt_base_degre}.\ref{lemme_nbre_pt_base8_degre}.
	\begin{cor}\label{cor_deg_germe_cel_adjacente_id}
	Les germes des cellules adjacentes à la cellule $\V(\id)$ qui ne sont pas des applications de caractéristique Jonquières sont de degré inférieur ou égal à $17$.
	\end{cor}
	Remarquons que si $\V(f)$ et $\V(g)$ sont deux cellules adjacentes, le segment géodésique $[\isome{f}(\l),\isome{g}(\l)]$ n'est pas forcément inclus dans l'union des cellules $\V(f)\cup\V(g)$. 
	\begin{ex}\label{ex_cellule_adjacente_pas_par_segment_geodesique}
	Considérons une application $f$ dont les points-base de l'inverse sont en position presque générale et qui a pour caractéristique $(4,3^3,1^3)$. Les classes dans l'intersection de $\V(f)$ et $\V(\id)$ sont $6$-symétriques, ce qui n'est le cas d'aucune classe du segment $[\isome{f}(\l),\l]$. Par conséquent le segment traverse au moins une autre cellule de Voronoï. 
	\end{ex}
	Cependant si $f$ est une application de caractéristique Jonquières ou une application symétrique cela n'arrive jamais. C'est l'objet du lemme suivant. Rappelons qu'une application est symétrique si elle est de degré strictement supérieur à $1$ et si les multiplicités de ses points-base sont toutes égales.
	
	\begin{lemme}\label{lemme_jonq_sym_collées}
	Soit $f$ une application symétrique ou de caractéristique Jonquières. Alors nous avons l'inclusion \[[\l,\isome{f}(\l)]\subset \V(\id)\cup\V(f).\] 
	\end{lemme}
	\begin{proof}
	Si $\j$ est une application de Jonquières de degré $d$, dont l'inverse a pour point-base maximal $p_0$ et $\{p_1,\dots,p_{2d-2}\}$ comme petits points-base. Notons 
	\[c_{\j}=\sqrt{\frac{1}{2(d+1)}}\bigg((d+1)\l-(d-1)e_{p_0}-\sum\limits_{i=1}^{2d-2}e_{p_i}\bigg).\] 
	Si $s$ est une application symétrique de degré $d$ dont l'inverse possède $r$ points-base $\{p_0,\dots,p_{r-1}\}$, nous posons \[c_s=\sqrt{\frac{d+1}{2}}\bigg(\l-\frac{1}{3}\sum\limits_{i=0}^{r-1}e_{p_i}\bigg).\]
	Les classes $c_{\j}$ et $c_{s}$ sont les normalisations des milieux des segments de l'espace de Hilbert ambiant reliant respectivement les classes $\l$ et $\isome{\j}(\l)$ et $\l$ et $\isome{s}(\l)$. Par conséquent les classes $c_{\j}$ et $c_s$ appartiennent respectivement aux segments géodésiques $[\l,\isome{\j}(\l)]$ et $[\l,\isome{s}(\l)]$. 
	L'application $s^{-1}$ est aussi une application symétrique de degré $d$ et d'après le lemme \ref{lemme_appli_sym} la multiplicité de ses points-base est égale à $\frac{d+1}{3}$. Par conséquent, elle ne peut pas posséder trois points-base alignés, ni six points-base sur une conique car sinon cela contredirait l'inégalité de Bézout (Proposition \ref{prop_Bezout}). De même, il ne peut pas y avoir deux points-base adhérents à un même troisième sinon cela contredirait la positivité des excès pour $s^{-1}$ (Proposition \ref{prop_consistence}).
	Ainsi en utilisant le corollaire \ref{cor_cellule_adjacente} et le théorème \ref{thm_principal_germe_cellule_voisine_identite}, $c_{\j}\in\V(\id)\cap\V(\j)$ et $c_s\in\V(\id)\cap\V(s)$. Ainsi, $c_{\j}$ est à égale distance de $\l$ et de $\isome{j}(\l)$ et $c_s$ de $\l$ et de $\isome{s}(\l)$ et comme elles appartiennent respectivement aux segments géodésiques $[\l,\isome{\j}(\l)]$ et $[\l,\isome{s}(\l)]$ ce sont les milieux de ces segments. Ainsi, ces segments sont inclus dans la réunion des deux cellules de Voronoï.
\end{proof}

\section{Cellules quasi-adjacentes à la cellule associée à l'identité}\label{section_cellule_quasiadjacente}
	Rappelons que le bord à l'infini de $\orbl$ est constitué des classes du bord de $\H$, c'est-à-dire des classes d'auto-intersection nulle et de degré strictement positif, qui sont limites d'une suite de classes vivant dans l'hyperplan $\{\l=1\}$ et proportionnelles à des classes de $\orbl$. Nous définissons de même le bord à l'infini d'une cellule de Voronoï $\V(f)$. Ce sont les classes du bord à l'infini de $\orbl$ qui sont limites de classes vivant dans l'hyperplan $\{\l=1\}$ et qui sont proportionnelles à des classes de $\V(f)$. Deux cellules ayant une classe en commun dans leur bord à l'infini sont dites \emph{quasi-adjacentes}.
	Dans cette section, le but est de caractériser les applications dont la cellule de Voronoï est quasi-adjacente à la cellule associée à l'identité. Avant cela, nous étudions les classes candidates à être dans de telles intersections et nous montrons au passage que les cellules adjacentes à $\V(\id)$ sont quasi-adjacentes à $\V(\id)$ (Corollaire \ref{cor_ajacente_quasiadjacente}). 

\begin{lemme}\label{lemme_classe_jonq_orbl}
	Soit $\{p_0,\dots,p_{2d-2}\}$ le support d'une application de Jonquières de degré $d$ et de point-base maximal $p_0$. Pour tout $n$ strictement positif et pour tout $\frac{d-1}{d}n\leq \lambda_0<n$ la classe : 
	\[c_{\lambda_0}= n\l -\lambda_0e_{p_0} -\frac{n-\lambda_0}{2}\sum\limits_{i=1}^{2d-2}e_{p_i}\in\PM(\P^2)\]
	est d'auto-intersection positive et est proportionnelle à une classe de $\orbl$. 
\end{lemme}
\begin{proof}
	Pour prouver que pour tout $\frac{d-1}{d}n\leq \lambda_0<n$ la classe $c_{\lambda_0}$ est proportionnelle à une classe de $\orbl$, nous devons montrer que cette classe satisfait les points \ref{propriete_coeff_classe}, \ref{propriete_classe_canonique}, \ref{propriete_classe_exces} et \ref{propriete_classe_bezout} de la définition \ref{proprietes_c}.

	Pour $\lambda_0 <n$, les multiplicités de la classe $c_{\lambda_0}$ sont toutes positives et donc cette classe satisfait \ref{proprietes_c}.\ref{propriete_coeff_classe}.
	Montrons que pour $n$ strictement positif fixé, et pour $\lambda_0$ satisfaisant $\frac{d-1}{d}n\leq \lambda_0<n$, $c_{\lambda_0}$ est d'auto-intersection positive et vérifie les points \ref{propriete_classe_canonique}, \ref{propriete_classe_exces} et \ref{propriete_classe_bezout} de la définition \ref{proprietes_c}. 
	\begin{itemize}[wide]
		\item Posons \[f(\lambda_0)=c_{\lambda_0}^2=n^2-\lambda_0^2-\Big(\frac{n-\lambda_0}{2}\Big)^2(2d-2).\] Alors $f'(\lambda_0)=-\lambda_0(d+1)+n(d-1)$. Ainsi la fonction $f$ est croissante sur l'intervalle $[\frac{n}{3},\frac{n(d-1)}{d+1}]$ puis décroissante sur l'intervalle $[\frac{n(d-1)}{d+1},n[$. De plus, $\underset{\lambda_0\rightarrow n}{\lim}f(\lambda_0)=0$. Par conséquent, pour $\frac{d-1}{d+1}n\leq \frac{d-1}{d}n\leq \lambda_0<n$, $c_{\lambda_0}$ est d'auto-intersection positive. 
		\item La condition \ref{propriete_classe_canonique} qui est la positivité contre l’anti-canonique correspond à : \[3n-\lambda_0-(d-1)(n-\lambda_0)\geq 0.\] Elle est vérifiée sans restriction sur $\lambda_0$ lorsque $d=2$ et sinon $\lambda_0$ doit satisfaire : 
		\[\lambda_0 \geq \frac{n(d-4)}{d-2}.\]
		De plus, pour $d>2$ nous avons $\frac{d-1}{d}\geq \frac{d-4}{d-2}$ et cette condition est satisfaite pour $\lambda_0\geq \frac{d-1}{d}n$.
		\item Nous devons vérifier à quelle condition sur $\lambda_0$, les excès en tous les points du support de $c_{\lambda_0}$ sont positifs (Définition \ref{proprietes_c}.\ref{propriete_classe_exces}). D'après la remarque \ref{rmq_pas_deuxpoints_adherents_petit_jonquieres}, il suffit juste de regarder l'excès au-dessus du point $p_0$. De plus, il y a au plus $d-1$ points adhérents à $p_0$. Nous voulons avoir : \[\lambda_0\geq (d-1)\frac{n-\lambda_0}{2}.\] 
		Ceci est satisfait lorsque $\lambda_0\geq  \frac{d-1}{d+1}n$, ce qui est le cas lorsque $\lambda_0\geq\frac{d-1}{d}n$.
		\item Nous nous intéressons à présent à l'inégalité de Bézout (Définition \ref{proprietes_c}.\ref{propriete_classe_bezout}). Soit $D$ une courbe de degré $\delta$ passant avec multiplicités $\{\mu_i\}_{0\leq i\leq 2d-2}$ en les points $\{p_i\}_{0\leq i\leq 2d-2}$. Nous cherchons à savoir pour quelles valeurs de $\lambda_0$, le terme suivant est positif : 
		\begin{equation}\label{eq_bezout_jonq}
		n\delta-\lambda_0\mu_0-\frac{n-\lambda_0}{2}\sum_{i=1}^{2d-2}\mu_i = n\bigg(\delta-\frac{1}{2}\sum_{i=1}^{2d-2}\mu_i\bigg)+\lambda_0\bigg(\frac{1}{2}\sum_{i=1}^{2d-2}\mu_i-\mu_0\bigg).
		\end{equation}
		Si toutes les multiplicités $\{\mu_i\}_{1\leq i\leq 2d-2}$ sont nulles alors le terme \eqref{eq_bezout_jonq} est égal à $n\delta-\lambda_0\mu_0> n(\delta-\mu_0)\geq 0$ et est positif. Considérons le cas où au moins l'un des $\mu_i$ est strictement positif pour $1\leq i\leq 2d-2$. Par positivité des excès des points par lesquels passe la courbe $D$ il existe alors un point de $\P^2$ ou un point libre et adhérent à $p_0$, noté $p_{i_k}\in \{p_i\}_{1\leq i\leq 2d-2}$ tel que $\mu_{i_k}>0$. 
		En considérant la droite passant par les points $p_0$ et $p_{i_k}$ le théorème de Bézout pour la courbe $D$ implique que $\delta\geq \mu_0+\mu_{i_k}$. En particulier, $\delta$ est strictement supérieur à $\mu_0$ : $\delta>\mu_0$.
		\begin{itemize}
			\item Si $\delta\geq\frac{1}{2}\sum_{i=1}^{2d-2}\mu_i$ et $\frac{1}{2}\sum_{i=1}^{2d-2}\mu_i\geq\mu_0$ alors \eqref{eq_bezout_jonq} est positif pour tout $\lambda_0$ positif donc en particulier pour $\frac{d-1}{d}n\leq \lambda_0<n$.
			\item Considérons le cas où $\delta\geq\frac{1}{2}\sum_{i=1}^{2d-2}\mu_i$ et $\frac{1}{2}\sum_{i=1}^{2d-2}\mu_i<\mu_0$. Alors, pour tout $\lambda_0$ positif le terme \eqref{eq_bezout_jonq} est supérieur ou égal à : 
			\[(n-\lambda_0)\bigg(\mu_0-\frac{1}{2}\sum_{i=1}^{2d-2}\mu_i\bigg)\geq 0.\] C'est donc en particulier le cas pour $\frac{d-1}{d}n\leq \lambda_0<n$.
			\item Si $\delta<\frac{1}{2}\sum_{i=1}^{2d-2}\mu_i$ alors $\frac{1}{2}\sum_{i=1}^{2d-2}\mu_i>\mu_0$. Par conséquent, pour que \eqref{eq_bezout_jonq} soit positif, il faut que nous ayons :
			\begin{equation}\label{eq_bezout_lambda_0}
			n>\lambda_0\geq nt \hspace{1cm}\text{ où }\hspace{1cm} t =\frac{\frac{1}{2}\sum_{i=1}^{2d-2}\mu_i-\delta}{\frac{1}{2}\sum_{i=1}^{2d-2}\mu_i-\mu_0}.
			\end{equation}
			Cependant, le facteur $t$ dépend de la courbe $D$ et nous voulons un coefficient $\lambda_0$ uniforme. En particulier, il ne faudrait pas qu'il existe une suite de courbes $\{D_i\}$ telle que la suite de facteurs $\{t_i\}$ associés converge vers $1$. Montrons que nous pouvons majorer $t$ indépendamment de $D$. Considérons une courbe $C$ du système linéaire associé à l'application de Jonquières. En appliquant le théorème de Bézout entre $C$ et $D$, nous obtenons : 
			\[0\leq \delta d-\mu_0(d-1)-\sum\limits_{i=1}^{2d-2}\mu_i.\]
			Cela se réécrit : 
			\begin{equation}\label{eq_bezout_jonq_D}
			-\delta \leq -\frac{d-1}{d}\mu_0-\frac{1}{d}\sum\limits_{i=1}^{2d-2}\mu_i.
			\end{equation}
			Par conséquent, nous avons : 
			\begin{align*}
			\frac{1}{2}\sum_{i=1}^{2d-2}\mu_i-\delta & \underset{\eqref{eq_bezout_jonq_D}}{\leq}\Big(\frac{1}{2}-\frac{1}{d}\Big)\sum\limits_{i=1}^{2d-2}\mu_i -\frac{d-1}{d}\mu_0\\
			& =\frac{d-2}{2d}\sum\limits_{i=1}^{2d-2}\mu_i -\frac{d-1}{d}\mu_0\\
			&<\frac{d-1}{d}\bigg(\frac{1}{2}\sum_{i=1}^{2d-2}\mu_i-\mu_0\bigg).
			\end{align*}
			Ainsi, $t$ est majoré par $\frac{d-1}{d}$ et pour tout $\frac{d-1}{d}n\leq \lambda_0<n$, $\lambda_0$ satisfait l'équation \eqref{eq_bezout_lambda_0}.
		\end{itemize}
	\end{itemize}
	Finalement, nous avons montré que pour tout $\lambda_0$ tel que $\frac{d-1}{d}n\leq \lambda_0<n$, la classe $c_{\lambda_0}$ est d'auto-intersection strictement positive et est proportionnelle à une classe de $\orbl$ comme annoncé.
\end{proof}

La proposition suivante fournit des exemples de cellules quasi-adjacentes à la cellule associée à l'identité.
\begin{prop}\phantomsection	\label{prop_classe_bord_infini_cellule_adjacente}
\begin{enumerate}[wide]
		\item\label{prop_classe_bord_infini_jonq} Soit $\j$ une application de caractéristique Jonquières dont l'inverse a pour point-base maximal $p_0\in\P^2$. Alors, nous avons : \[\l-e_{p_0}\in\partial_{\infty}\V(\id)\cap\partial_{\infty}\V(\j).\]
		\item\label{prop_classe_bord_infini_gen} Soit $f$ une application dont les $r\leq 8$ points-base $\{p_0,\dots,p_{r-1}\}$ sont en position presque générale. Alors pour toute complétion en un ensemble de neuf points $\{p_0,\dots,p_8\}$ en position presque générale, nous avons : 
		\[3\l-\sum_{i=0}^{8}e_{p_i}\in\partial_{\infty}\V(\id)\cap\partial_{\infty}\V(f).\]
	\end{enumerate}
\end{prop}
\begin{proof}
	\begin{enumerate}[wide]
		\item Soit $\j$ une application de caractéristique Jonquières de degré $d$ dont l'inverse a pour point-base maximal $p_0\in\P^2$ et pour petits points-base $\{p_i\}_{1\leq i\leq 2d-2}$. Fixons $n$ positif. Soit $\frac{n}{3}\leq \lambda_0<n$ et considérons la classe \[c_{\lambda_0}= n\l -\lambda_0e_{p_0} -\frac{n-\lambda_0}{2}\sum\limits_{i=1}^{2d-2}e_{p_i}\in\PM(\P^2).\] 
		D'après le lemme \ref{lemme_classe_jonq_orbl}, pour tout $\lambda_0$ tel que $ \frac{d-1}{d}n\leq \lambda_0<n$, cette classe est d'auto-intersection strictement positive et est proportionnelle à une classe de $\orbl$. Par conséquent, la normalisation $\eta(c_{\lambda_0})$ de la classe $c_{\lambda_0}$ introduite \ref{application_normalisation} est d'auto-intersection $1$ et appartient à $\orbl$ pour tout $ \frac{d-1}{d}n\leq \lambda_0<n$. Le théorème \ref{thm_principal_germe_cellule_voisine_identite} implique que $\eta(c_{\lambda_0})$ appartient à l'intersection $\V(\id)\cap\V(\j)$. Par conséquent, la classe $\frac{c_{\lambda_0}}{n}$ est proportionnelle à une classe appartenant à l'intersection des cellules $\V(\id)$ et $\V(\j)$. En prenant la limite quand $\lambda_0$ tend vers $n$ nous obtenons :  \[\l-e_{p_0}\in\partial_{\infty}\V(\id)\cap\partial_{\infty}\V(\j)\]comme annoncé.
		\item Considérons pour tout $0<\epsilon\leq 1$ la classe
		\[c_{\epsilon}=3\l -\sum\limits_{i=0}^{r-1}e_{p_i}-(1-\epsilon)\sum\limits_{i=r}^{8}e_{p_i}.\]
		Son auto-intersection est strictement positive : $9-r-(1-\epsilon)^2(9-r)>0$. Nous pouvons par conséquent considérer sa normalisation $\eta(c_{\epsilon})$ qui est d'auto-intersection $1$, qui est positive contre l'anti-canonique et dont les multiplicités et le degré sont positifs. Elle vérifie donc les points \ref{propriete_coeff_classe} et \ref{propriete_classe_canonique} de la définition \ref{proprietes_c}. Le point \ref{propriete_classe_exces} est vérifié car pour tout $0<\epsilon\leq 1$, les points du support de la classe $c_{\epsilon}$ sont en position presque générale et par conséquent les excès sont tous positifs. Par le même argument, l'inégalité de Bézout est vérifiée pour les droites et les coniques. Il reste à vérifier que l'inégalité de Bézout est satisfaite pour les courbes $D$ de degré $d$ supérieur ou égal à $3$ et passant avec multiplicité $\mu_i$ par les points $p_i$ pour $0\leq i\leq 8$. Comme les points $\{p_0,\dots,p_7\}$ sont en position presque générale, d'après la proposition \ref{prop_weak_del_Pezzo_presque_general}, la surface obtenue en éclatant ces points est faiblement del Pezzo. Avec le même argument que dans la preuve du corollaire \ref{cor_cellule_adjacente}, le théorème de Riemann-Roch implique que la forme anti-canonique sur cette surface correspond à un pinceau de cubiques. Par conséquent, il existe une cubique $C$ de ce pinceau passant par le point $p_8$. D'après l'inégalité de Bézout appliquée aux courbes $C$ et $D$, nous avons :
		\[3d \geq \sum\limits_{i=0}^{8}\mu_i\geq \sum\limits_{i=0}^{r-1}\mu_i+(1-\epsilon)\sum\limits_{i=r}^{8}\mu_i.\] Cela montre que la classe $c_{\epsilon}$ satisfait le point \ref{proprietes_c}.\ref{propriete_classe_bezout} et achève donc de prouver que $c_{\epsilon}$ est proportionnelle à une classe de $\orbl$. 
		En utilisant le théorème \ref{thm_principal_germe_cellule_voisine_identite}, nous obtenons finalement que pour tout $0<\epsilon \leq 1$, la classe $c_{\epsilon}$ est proportionnelle à une classe vivant dans l'intersection $\V(\id)\cap\V(f)$. En passant à la limite quand $\epsilon$ tend vers $0$, nous avons :	
		\[3\l-\sum_{i=0}^{8}e_{p_i}\in\partial_{\infty}\V(\id)\cap\partial_{\infty}\V(f)\] qui est le résultat annoncé. \qedhere
	\end{enumerate}
\end{proof}

\begin{cor}\label{cor_ajacente_quasiadjacente}
	Une cellule adjacente à $\V(\id)$ est quasi-adjacente à $\V(\id)$.
\end{cor}
\begin{proof}
	D'après le corollaire \ref{cor_cellule_adjacente}, il n'y a que deux sortes de cellules adjacentes à $\V(\id)$, celles dont le germe est de caractéristique Jonquières et celles dont le germe est une application dont les $r\leq 8$ points-base de l'inverse sont en position presque générale. Par la proposition \ref{prop_classe_bord_infini_cellule_adjacente}, elles sont toutes quasi-adjacentes à $\V(\id)$.
\end{proof}

\begin{rmq}\label{existence_pointe_9_sym}
	Par définition du bord à l'infini $\partial_{\infty}\orbl$, nous obtenons que les classes de type $3\l-\sum_{i=0}^{8}e_{p_i}$ appartiennent à $\partial_{\infty}\orbl$ si et seulement si l'ensemble $\{p_0,\dots,p_8\}$ est en position presque générale. Ces classes appartiennent en fait à $\partial_{\infty}\V(\id)$.
\end{rmq}

\begin{lemme}\label{lemme_classe_9sym_1sym}
	Soit $c\in\partial_{\infty}\V(\id)\cap \partial_{\infty}\V(f)$. Si $c$ n'est pas une classe spéciale alors $c$ est une classe $1$-symétrique pure ou $9$-symétrique pure. Sinon $c$ est une classe $1$-symétrique.
\end{lemme}
\begin{proof}
	Supposons $c=n\l-\sum_{i\in I}\lambda_ie_{p_i}$ ordonnée.
	Si $c$ est une classe spéciale alors les points $p_1$ et $p_2$ sont adhérents au point $p_0$. Par positivité des excès en $p_0$, la multiplicité $\lambda_1$ doit être strictement plus petite que $\lambda_0$. Par conséquent $c$ est une classe $1$-symétrique. 
	
	Considérons maintenant le cas où la classe $c$ n'est pas spéciale. Par définition du bord à l'infini, $c$ est limite d'une suite de classes proportionnelles à des classes de $\V(\id)\cap \V(f)$. Quitte à prendre une sous-suite nous pouvons supposer que les classes de cette suite sont soit toutes spéciales soit toutes non spéciales. Si ces classes sont spéciales le degré de chacune des classes de cette suite est strictement inférieur à la somme de ses trois plus grandes multiplicités. Or $c$ n'est pas une classe spéciale donc à la limite $c$ satisfait $n=\lambda_0+\lambda_1+\lambda_2$. 
	Si toutes les classes de cette suite ne sont pas spéciales alors le degré de chacune est égal à la somme de ses trois plus grandes multiplicités (Théorème \ref{thm_principal_germe_cellule_voisine_identite}). Ainsi, à la limite la classe $c$ possède la même propriété :
	\begin{align}\label{calcul_deg_mult}
	n=\lambda_0+\lambda_1+\lambda_2.
	\end{align}
	
	Si la classe $c$ est $k$-symétrique avec $k\geq 3$, alors ses multiplicités maximales sont égales au tiers du degré. Nous sommes dans les conditions du lemme \ref{lemme_symetrique} et $c$ est $9$-symétrique pure. 
	Supposons à présent que $c$ est $k$-symétrique pour $k\leq 2$. Comme $c^2=0$, nous avons \begin{equation}\label{eq_c_carre_nul}
	n^2=\sum\limits_{i\in I}\lambda_i^2.
	\end{equation}
	D'après \ref{calcul_deg_mult} et \ref{eq_c_carre_nul}, si $\lambda_2=0$ il en est de même de $\lambda_1$ et donc la classe $c$ est $1$-symétrique pure. Montrons que $\lambda_2$ est forcément nul. Supposons le contraire. Nous avons :
	\begin{equation}\label{eq_contradiction}
	\begin{split}
	\lambda_2\Big(3n-\sum\limits_{i\geq 0}\lambda_i\Big) &\underset{\eqref{calcul_deg_mult}}{=} \lambda_2\Big(2n-\sum\limits_{i\geq 3}\lambda_i\Big)\\
	&\underset{\eqref{eq_c_carre_nul}}{=} -n^2+\lambda_0^2+\lambda_1^2+\lambda_2^2+2n\lambda_2+\sum\limits_{i\geq 3}\lambda_i(\lambda_i-\lambda_2)\\
	&\hspace{0.33cm}=-(n-\lambda_2)^2+\lambda_0^2+\lambda_1^2+2\lambda_2^2+\sum\limits_{i\geq 3}\lambda_i(\lambda_i-\lambda_2)\\
	&\underset{\eqref{calcul_deg_mult}}{=}-(\lambda_0+\lambda_1)^2+\lambda_0^2+\lambda_1^2+2\lambda_2^2+\sum\limits_{i\geq 3}\lambda_i(\lambda_i-\lambda_2)\\
	&\hspace{0.33cm}=-2\lambda_0\lambda_1+2\lambda_2^2+\sum\limits_{i\geq 3}\lambda_i(\lambda_i-\lambda_2)
	\end{split}
	\end{equation}

	Comme la classe $c$ est $k$-symétrique pour $k\leq2$, $\lambda_2<\lambda_0$. De plus, comme les multiplicités sont décroissantes le dernier terme de \eqref{eq_contradiction} est strictement négatif. Ceci implique qu'il en est de même de $3n-\sum\limits_{i\geq 0}\lambda_i$ puisque $\lambda_2$ est strictement positif. Mais cela contredit le fait qu'une classe au bord est positive contre l'anti-canonique.
\end{proof}

Remarquons qu'il n'existe peut-être pas de classe spéciale au bord à l'infini de $\V(\id)$. Nous n'avons ni réussi à en construire une ni réussi à montrer qu'il n'en existait pas.

\begin{prop}\label{prop_classe_infini_au_bord}
	L'ensemble des classes non spéciales au bord à l'infini de $\V(\id)$ et qui sont également dans le bord à l'infini d'une cellule quasi-adjacente à $\V(\id)$ est constitué des classes suivantes :
	\begin{enumerate}[wide]
		\item\label{prop_classe_infini_au_bord_cas_jonq} $\l-e_{p_0}$ où $p_0\in\P^2$,
		\item\label{prop_classe_infini_au_bord_cas_gen} $3\l-\sum\limits_{i=0}^{8}e_{p_i}$ où les points de l'ensemble $\{p_0,\dots,p_8\}$ sont en position presque générale.
	\end{enumerate}
\end{prop}
\begin{proof}
	D'après la proposition \ref{prop_classe_bord_infini_cellule_adjacente}, les cas \ref{prop_classe_infini_au_bord_cas_jonq} et \ref{prop_classe_infini_au_bord_cas_gen} sont dans l'intersection entre le bord à l'infini de $\V(\id)$ et celui d'une cellule adjacente. 
	Le lemme \ref{lemme_classe_9sym_1sym} nous dit que les seules classes non spéciales vérifiant l'hypothèse de la proposition sont $1$-symétrique pure ou $9$-symétrique pure. Une classe $9$-symétrique pure doit avoir comme support un ensemble de $9$ points en position presque générale car sinon cela contredirait le fait que cette classe appartient à $\partial_{\infty}\orbl$ (Remarque \ref{existence_pointe_9_sym}).
\end{proof}

\begin{prop}\label{prop_classe_infini_type_jonq}
	Soient $p_0\in\P^2$ et $f\in\Bir(\P^2)$. La classe $\l-e_{p_0}$ appartient à l'intersection $\partial_{\infty}\V(\id)\cap\partial_{\infty}\V(f)$ si et seulement si $f$ est une application de caractéristique Jonquières dont l'inverse a pour point-base maximal $p_0$. 
\end{prop}
\begin{proof}
	L'implication indirecte découle de la proposition \ref{prop_classe_bord_infini_cellule_adjacente}.\ref{prop_classe_bord_infini_jonq}.
	
	Soit $f\in\Bir(\P^2)$ telle que $\l-e_{p_0}$ soit dans l'intersection $\partial_{\infty}\V(\id)\cap\partial_{\infty}\V(f)$. Si la cellule $\V(f)$ est une cellule adjacente à la cellule associée à l'identité alors pour toute classe $c\in\V(\id)\cap\V(f)$ nous avons par convexité des cellules de Voronoï l'inclusion suivante : \[[c,\l-e_{p_0}[\ \subset\V(\id)\cap\V(f).\] D'après le théorème \ref{thm_principal_germe_cellule_voisine_identite}, si $f$ n'est pas une application de caractéristique Jonquières les classes de $\V(\id)\cap\V(f)$ sont au moins $r$-symétriques où $r\geq 6$ est le nombre de points-base de $f$. Par conséquent, en notant $p_0,\dots,p_r$ les points-base de $f^{-1}$ toute classe multiple d'une classe du segment géodésique $[c,\l-e_{p_0}[$ vit dans le sous-espace vectoriel de l'espace de Picard-Manin obtenu comme intersection des hyperplans $\{\lambda_{p_0}=\lambda_{p_1}\}$, $\{\lambda_{p_0}=\lambda_{p_2}\}$, \dots, $\{\lambda_{p_0}=\lambda_{p_{r-1}}\}$. Or, la classe $\l-e_{p_0}$ n'appartient pas au bord à l'infini de ce sous-espace ce qui mène à une contradiction. Avec le même type d'argument nous obtenons une contradiction si nous supposons que $f$ est une application de caractéristique Jonquières dont le point-base maximal de $f^{-1}$ n'est pas le point $p_0$.
	
	Montrons que $f$ ne peut pas être le germe d'une cellule quasi-adjacente non adjacente à $\V(\id)$. Raisonnons par l'absurde et supposons le contraire. L'espace $\orbl$ est convexe et les cellules de Voronoï recouvrent $\orbl$. De plus, les cellules adjacentes à la cellule $\V(\id)$ et contenant la classe $\l-e_{p_0}$ sont des applications de caractéristiques Jonquières dont le point-base maximal de l'inverse est $p_0$. Nous pouvons par conséquent supposer que la cellule $\V(f)$ est adjacente à une cellule $\V(\j)$ où $\j$ est un représentant qui est une application de Jonquières et dont l'inverse a pour point-base maximal $p_0$. 
	Faisons agir à présent l'application $\j^{-1}$. Elle fixe la classe à l'infini $\l-e_{p_0}$ ainsi la cellule $\V(\j^{-1}\circ f)$ est une cellule adjacente de la cellule identité et elle possède la classe $\l-e_{p_0}$ dans son bord à l'infini. Par le point précédent, nous obtenons que $\j^{-1}\circ f$ est une application de caractéristique Jonquières dont l'inverse a pour point-base maximal $p_0$. C'est donc aussi le cas de $f$ ce qui mène à une contradiction. Ainsi, $\V(f)$ est forcément adjacente à $\V(\id)$ et $f$ est de caractéristique Jonquières et son inverse possède le point $p_0$ comme point-base maximal.
\end{proof}

\begin{rmq}\label{classe_infini_cas_special}
S'il existe une classe spéciale $c\in\V(\id)\cap\V(f)$ et si $\V(f)$ est adjacente à $\V(\id)$ alors $f^{-1}$ est une application de caractéristique Jonquières dont le point-base maximal est le point de multiplicité maximale de $c$. Cela découle du théorème \ref{thm_principal_germe_cellule_voisine_identite} et du fait qu'il n'existe pas une suite de classes non spéciales de $\V(\id)\cap\V(f)$ convergeant vers $c$ puisque leur degré est égal à la somme de leurs trois plus grandes multiplicités contrairement à $c$. Par le même argument que dans la preuve de la proposition \ref{prop_classe_infini_type_jonq}, ce sont les seules applications dont les cellules de Voronoï contiennent cette classe.
\end{rmq}

\begin{lemme}\label{lemme_image_9_pure}
Soient $f\in\Bir(\P^2)$ et $c=3\l-\sum_{i=0}^{8}e_{p_i}\in\partial_{\infty}\V(\id)$. Si le support de $f^{-1}$ est inclus dans le support de $c$, alors le support de $f^{-1}(c)$ contient neuf points en position presque générale et 
\[\isome{f}^{-1}(c)=3\l-\sum_{q\in\supp(f)}e_{q}-\sum_{\substack{p\in\supp(c)\\ p\notin\supp(f^{-1})}}\isome{f}^{-1}(e_p)\in\partial_{\infty}\V(\id).\]
\end{lemme}
\begin{proof}
Notons $\{m_q\}_{q\in\supp(f)}$ (respectivement $\{m_p'\}_{p\in\supp(f^{-1})}$) les multiplicités des points-base de $f$ (respectivement de $f^{-1}$) et $\{a_{p,q}\}_{q\in\supp(f)}$ les autres coefficients de la matrice caractéristique. Nous avons d'après l'équation \eqref{eq_action_f_sur_c} :
\begin{multline*}
\isome{f}^{-1}(c)=\Big(3d-\sum\limits_{p\in\supp(f^{-1})}m_p'\Big)\l-\sum\limits_{q\in\supp(f)}\Big(3m_q-\sum\limits_{p\in\supp(f^{-1})}a_{p,q}\Big)e_{q}\\-\sum_{\substack{p\in\supp(c)\\ p\notin\supp(f^{-1})}}\isome{f}^{-1}(e_p).
\end{multline*}
		Or d'après les égalités \ref{eq_Noethercanonique} et \ref{eq_aij} du lemme \ref{lemme_equations_evidentes}, nous avons : \[ \begin{cases}
		3d-\sum\limits_{p\in\supp(f^{-1})}m_p'=3\\
		3m_q-\sum\limits_{p\in\supp(f^{-1})}a_{q,p}=1,\  \text{ pour tout } q\in\supp(f)
		\end{cases}\]
		ce qui implique que \[\isome{f}^{-1}(c)=3\l-\sum\limits_{q\in\supp(f)}e_q-\sum_{\substack{p\in\supp(c)\\ p\notin\supp(f^{-1})}}\isome{f}^{-1}(e_p)\] comme attendu. Comme le groupe de Cremona préserve 
	$\partial_{\infty}\orbl$, la classe $\isome{f}^{-1}(c)$ appartient à $\partial_{\infty}\orbl$, elle est par conséquent d'auto-intersection nulle et possède ainsi $9$ points dans son support. D'après la remarque \ref{existence_pointe_9_sym}, cela signifie qu'elle appartient à $\partial_{\infty}\V(\id)$ et que le support de cette classe est un ensemble de neuf points en position presque générale.
\end{proof}

\begin{lemme}\label{lemme_vois_inf}
Soit $c=3\l-\sum\limits_{i=0}^{8}e_{p_i}\in\partial_{\infty}\V(\id)$. Soit $c_1$ une classe appartenant à l'intersection de $V(\id)$ et d'une autre cellule $\V(f)$ telle que $c_1\cdot \frac{1}{3}c< \frac{1}{18}$. Pour toute application $g\in\Bir(\P^2)$ telle que $c_1\in\V(g)$, les points-base de $g^{-1}$ sont inclus dans le support de $c$.
\end{lemme}
\begin{rmq}
Le nombre d'intersection entre une classe de $\orbl$ et de $\partial_{\infty}\orbl$ est bien défini seulement si nous considérons un représentant de la classe à l'infini qui possède un coefficient $1$ devant $\l$. 
\end{rmq}
\begin{proof}
Soit $c_1=n\l-\sum_{i\in I}\lambda_ie_{p_i}$ non-nécessairement ordonnée. Par hypothèse sur $c_1$ et par positivité contre l'anti-canonique (\ref{propriete_classe_canonique} définition \ref{proprietes_c}), nous obtenons: \begin{equation}\label{ineq_mult}
\sum\limits_{i\in I\setminus\{0,\dots,8\}}\lambda_i\leq 3n-\sum\limits_{i=0}^{8}\lambda_i<\frac{1}{6}
\end{equation}
où $\lambda_i$ peut être nulle si le point $p_i$ pour $0\leq i\leq 8$ n'appartient pas au support de $c_1$.
Montrons dans un premier temps que $c_1$ n'est pas une classe spéciale. Supposons le contraire et notons $\mu_0$ son unique multiplicité maximale correspondant au point noté $q_0$. Notons que $q_0$ correspond à l'un des $p_i$ pour $i\in I$. D'après la remarque \ref{rmq_classe_speciale_lambda0_nsur2}, $\mu_0>\frac{n}{2}$. Notons $\mu_1=\max\{\lambda_i\mid 0\leq i\leq 8 \text{ et }p_i\neq q_0\}$ et $q_1$ l'un des points $\{p_i\}_{0\leq i\leq 8}$ correspondant à la multiplicité $\mu_1$. D'après la remarque \ref{existence_pointe_9_sym}, les points du support de $c$ sont en position presque générale. Par conséquent, comme $c_1\in\V(\id)$, pour tout $i\in\{0,\dots,8\}$ avec $p_i\notin\{q_0,q_1\}$, nous avons par la proposition \ref{prop_points_alignes_posgen} et par positivité des excès \ref{propriete_classe_exces} dans le cas où $p_i$ n'est pas un point de $\P^2$ : \[n-\mu_0-\mu_1-\lambda_i\geq 0.\]
Ainsi tous les $\{\lambda_i\}_{0\leq i\leq 8}$ sauf peut-être deux (correspondant à $q_0$ et $q_1$) satisfont : $\lambda_i\leq \frac{n-\mu_0}{2}<\frac{n}{4}$. De plus, en considérant la droite passant par $q_0$ et $q_1$ nous avons par l'inégalité de Bézout \ref{propriete_classe_bezout} : $\mu_0+\mu_1\leq n$. Par conséquent, nous obtenons $\sum\limits_{i=0}^{8}\lambda_i< n+7\frac{n}{4}$.
D'après la remarque \ref{propriete_classe_carre}, $n\geq 1$ par conséquent nous arrivons à la contradiction suivante :
\[\frac{1}{4}\leq \frac{n}{4}=2n-7\frac{n}{4}< 3n-\sum\limits_{i=0}^{8}\lambda_i<\frac{1}{6}.\]
Ainsi, $c$ n'est pas une classe spéciale.

Considérons trois cas, selon si $g$ n'est pas une application de caractéristique Jonquières, est une application de caractéristique Jonquières non quadratique, est une application quadratique.
\begin{itemize}[wide]
\item Si $g$ n'est pas une application de caractéristique Jonquières, par le théorème \ref{thm_principal_germe_cellule_voisine_identite}, pour tout point-base $p_i$ de $g^{-1}$ la multiplicité correspondante est égale à  $\lambda_i=\frac{n}{3}$. De plus, $n$ est supérieur ou égal à $1$ donc $\lambda_i\geq \frac{1}{3}$. Par conséquent, tout point-base $p_i$ de $g^{-1}$ doit être inclus dans le support de $c$ sinon cela contredit l'inégalité \eqref{ineq_mult}.
\item Si $g$ est une application de caractéristique Jonquières non quadratique, alors par le même argument que précédemment nous obtenons que le point-base de multiplicité maximale doit être inclus dans le support de $c$ (la multiplicité pour $c_1$ correspondant à ce point est supérieure ou égale à $\frac{n}{3}$ d'après le théorème \ref{thm_principal_germe_cellule_voisine_identite}).
Si un point-base de $g^{-1}$ de multiplicité $1$ n'est pas dans le support de $c$ alors par l'inégalité \ref{ineq_mult}, nous avons que sa multiplicité correspondante est strictement inférieure à $\frac{1}{6}$ et c'est le cas, par le théorème \ref{thm_principal_germe_cellule_voisine_identite}, de toutes les multiplicités de $c_1$ excepté celle correspondant au point-base maximal de $g^{-1}$. Par conséquent, $\sum\limits_{i=0}^{8}\lambda_i< n + 8\frac{1}{6}$ ce qui mène à la contradiction \[\frac{2}{3}=2-\frac{4}{3}\leq 2n-\frac{4}{3}< 3n-\sum\limits_{i=0}^{8}\lambda_i<\frac{1}{6}.\]
\item Considérons le cas où $g$ est une application quadratique. Si un point-base de $g^{-1}$ n'est pas dans le support de $c$ alors sa multiplicité correspondante est, d'après l'inégalité \ref{ineq_mult}, strictement inférieure à $\frac{1}{6}$. Ceci implique d'après le théorème \ref{thm_principal_germe_cellule_voisine_identite} que $\sum\limits_{i=0}^{8}\lambda_i< n+6\frac{1}{6}$. Par conséquent, avec le même argument que précédemment nous obtenons la contradiction $1\leq \frac{1}{6}$.
\end{itemize}
Ainsi, nous avons montré que pour toute application $g$ dont la cellule contient la classe $c_1$, les points-base de $g^{-1}$ sont inclus dans le support de $c$.
\end{proof}

\begin{prop}\label{prop_classes_bord_voisin_id}
	Soit $\{p_0,\dots,p_8\}$ un ensemble de points en position presque générale. L'application $f$ satisfait \[3\l-\sum_{i=0}^{8}e_{p_i}\in \partial_{\infty}\V(\id)\cap\partial_{\infty}\V(f)\] si et seulement si $f$ possède au plus neuf points-base et le support de son inverse est inclus dans $\{p_0,\dots,p_8\}$.
\end{prop}
\begin{proof}
	Posons $c=3\l-\sum_{i=0}^{8}e_{p_i}$. D'après la proposition \ref{prop_classe_infini_au_bord} elle appartient au bord à l'infini de $\V(\id)$.
	Montrons la réciproque. Si le support de $f^{-1}$ est inclus dans le support de $c$ alors le lemme \ref{lemme_image_9_pure} implique que la classe $\isome{f}^{-1}(c)$ est une classe $9$-symétrique pure dont le support est constitué de $9$ points en position presque générale. Par la proposition \ref{prop_classe_infini_au_bord}, nous obtenons que la classe $\isome{f}^{-1}(c)$ appartient à $\partial_{\infty}\V(\id)$. Par conséquent, $c$ appartient à $\partial_{\infty}\V(f)$ et nous pouvons en conclure que \[c\in\partial_{\infty}\V(\id)\cap \partial_{\infty}\V(f),\]comme attendu.

	Montrons l'implication directe.
	Soit $f\in\Bir(\P^2)$ telle que $c\in \partial_{\infty}\V(\id)\cap\partial_{\infty}\V(f)$. 
	Soit $c_1$ une classe appartenant à l'intersection entre $\V(\id)$ et à au moins une autre cellule de Voronoï telle que $c_1\cdot \frac{1}{3}c<\frac{1}{18}$. Soit $s\in\V(f)$ telle que $s\cdot \frac{1}{3}c<\frac{1}{18}$. Les classes $c_1$ et $s$ existent d'après la remarque \ref{rmq_limite_intersection_avec_classe_infini}.
	Remarquons que toute classe $c'$ du segment géodésique $[c_1,s]$ satisfait $c'\cdot \frac{1}{3}c<\frac{1}{18}$. En effet, ceci découle du fait que la forme d'intersection est bilinéaire, que $(tc_1+(1-t)s)^2\geq 1$ et que $c'=\eta(tc_1+(1-t)s)$ pour $t\in[0,1]$ et où $\eta$ et l'application de normalisation vue dans \ref{application_normalisation}.
	
	Nous paramétrons le segment $[c_1,s ]$ par $\gamma(t)$ pour $t \in [0, 1]$ avec $\gamma(0) = c_1$ et $\gamma(1) = s$.
	  	Nous construisons une suite de points $c_i=(\gamma(t_i))$ par le procédé de récurrence suivant.
	  	Initialisons en posant $t_1 = 0$ et $c_1 = \gamma(0)$.
	  	Pour $i \geq 2$, si $t_{i-1} \in [0,1[$ est déjà construit, nous posons 
	  	\[ t_i = \inf \{t \in ]t_{i-1},1]\mid \exists f \in \Bir(\P^2), \gamma(t) \in \V(f) \text{ et } \gamma(t_{i-1})\notin \V(f)\}.\] Si l'ensemble ci-dessus est vide, nous posons $t_i=1$.
	  	D'après le corollaire \ref{cor_nbre_fini_cell_traversees}, cette suite est finie. Notons $n$ le nombre de termes de cette suite. Par construction et par convexité des cellules de Voronoï, pour chaque $i\geq 1$, il existe $g_i\in\Bir(\P^2)$ telle que $[c_i,c_{i+1}]\subset \V(g_i)$.
	  	
	  	Montrons que pour tout $1\leq k\leq n$, le support de $g_k^{-1}$ est inclus dans le support de $c$.
	  	Raisonnons par récurrence sur $n$ pour prouver le résultat. Si $n=1$ alors $f$ est une cellule adjacente à $\V(\id)$ et $c_1\in\V(\id)\cap\V(f)$. Par le lemme \ref{lemme_vois_inf}, le support de $f^{-1}$ est inclus dans le support de $c$. 
	  	Supposons le résultat vrai pour $1\leq k\leq n-1$ et montrons qu'il reste vrai pour $k+1$. 
	  	Faisons agir $g_k^{-1}$ sur ce segment. Ainsi la classe $\isome{{g_k^{-1}}}(c_{k+1})\in\V(\id)\cap\V(g^{-1}_k\circ g_{k+1})$. 
		De plus, $\isome{{g^{-1}_k}}(c_{k+1})\cdot\frac{1}{3} \isome{{g^{-1}_k}}(c)<\frac{1}{18}$. D'après le lemme \ref{lemme_vois_inf}, nous en déduisons que le support de $(g^{-1}_k\circ g_{k+1})^{-1}$ est inclus dans le support de $\isome{{g^{-1}_k}}(c)$. De plus, par hypothèse de récurrence, le support de $g^{-1}_k$ est inclus dans le support de $c$.
	 D'après le lemme \ref{lemme_image_9_pure}, 
	\[\isome{{g_k^{-1}}}(c)=3\l-\sum_{q\in\supp(g_k)}e_{q}-\sum_{\substack{p\in\supp(c)\\ p\notin\supp(g_k^{-1})}}\isome{{g_k^{-1}}}(e_p)\] et son support est constitué de $9$ points en position presque générale.
	Les points-base de $g_{k+1}^{-1}\circ g_k$ sont inclus dans l'union de l'ensemble des points-base de $g_k$ et de l'ensemble des points images par $g_k^{-1}$ des points-base de $g_{k+1}^{-1}$ qui ne sont pas des points-base de $g_k^{-1}$. Par conséquent les points-base de $g_{k+1}^{-1}$ sont inclus dans le support de $c$ comme attendu. Ceci achève la récurrence.
	\end{proof}
	Le corollaire suivant se déduit de la remarque \ref{classe_infini_cas_special} et des propositions \ref{prop_classe_infini_au_bord}, \ref{prop_classe_infini_type_jonq} et \ref{prop_classes_bord_voisin_id}.
\begin{cor}\label{cor_cellules_quasi_adjacentes}
	L'ensemble des germes des cellules quasi-adjacentes à $\V(\id)$ est constitué de :
	\begin{itemize}[wide]
		\item toutes les applications de caractéristique Jonquières,
		\item de toutes les applications ayant au plus $9$ points-base et dont les points-base de l'inverse sont en position presque générale.
	\end{itemize}
\end{cor}

\begin{cor}\label{cor_2cellules_quasi_adjacentes}
	Soient $f,g\in\Bir(\P^2)$ telles que les cellules $\V(f)$ et $\V(g)$ soient quasi-adjacentes à la cellule associée à l'identité. Notons $p_0,\dots,p_{k-1}$ l'union des points-base de $f^{-1}$ et de $g^{-1}$. Les trois cellules $\V(\id)$, $\V(f)$ et $\V(g)$ ont une classe commune dans leur bord à l'infini si et seulement si nous sommes dans l'un des deux cas suivants : 
	\begin{itemize}[wide]
		\item $f^{-1}$ et $g^{-1}$ sont deux applications de caractéristique Jonquières ayant le même point-base maximal, ou un même point-base maximal dans le cas où l'une est quadratique,
		\item les points $\{p_0,\dots,p_{k-1}\}$ sont en position presque générale et $k\leq 9$.
	\end{itemize}
\end{cor}

\begin{proof}
	Supposons que les cellules $\V(\id)$, $\V(f)$ et $\V(g)$ possèdent une classe $c$ commune au bord à l'infini et que $f^{-1}$ et $g^{-1}$ ne soient pas des applications de caractéristique Jonquières ayant le même point-base maximal. Si la classe $c$ est spéciale alors par la remarque \ref{classe_infini_cas_special}, les applications $f$ et $g$ doivent être des applications de Jonquières de même point-base ce qui contredit notre hypothèse. La classe $c$ n'est donc pas spéciale. Ainsi, la proposition \ref{prop_classe_infini_au_bord} nous dit qu'il n'y a que deux possibilités pour cette classe. Les propositions \ref{prop_classe_infini_type_jonq} et \ref{prop_classes_bord_voisin_id} impliquent que $c$ est une classe $9$-symétrique pure dont le support contient les supports de $f^{-1}$ et de $g^{-1}$ : \[c=3\l-\sum\limits_{i=0}^{8}e_{p_i}\in\partial_{\infty}\V(f)\cap\partial_{\infty}\V(g)\cap\partial_{\infty}\V(\id).\] Ainsi, $k\leq 9$ et, d'après la remarque \ref{existence_pointe_9_sym}, les points du support de $c$ (et par conséquent $\{p_0,\dots,p_{k-1}\}$) sont en position presque générale.
	
	Montrons la réciproque. Si $f^{-1}$ et $g^{-1}$ sont des applications de caractéristique Jonquières et de même point-base maximal $p_0$ alors la classe $\l-e_{p_0}$ appartient au bord à l'infini des trois cellules par la proposition \ref{prop_classe_bord_infini_cellule_adjacente}.\ref{prop_classe_bord_infini_jonq}. 
	Considérons le cas où $f^{-1}$ et $g^{-1}$ ne sont pas des applications de caractéristique Jonquières ayant même point-base maximal. 
	Notons $\{p_0,\dots,p_{k-1}\}$ l'union des points-base de $f^{-1}$ et de $g^{-1}$, et supposons que cet ensemble soit en position presque générale et $k\leq 9$. Alors nous pouvons compléter cette famille de sorte que $\{p_0,\dots,p_8\}$ soit un ensemble de points en position presque générale. D'après la remarque \ref{existence_pointe_9_sym}, la classe $3\l-\sum\limits_{i=0}^{8}e_{p_i}$ appartient au bord à l'infini de $\V(\id)$, et d'après la proposition \ref{prop_classes_bord_voisin_id} nous obtenons qu'elle appartient également au bord à l'infini de $\V(g)$ et de $\V(f)$ comme annoncé.
\end{proof}
	\bibliographystyle{smfalpha}
	\bibliography{biblio}

	\end{document}